\documentclass[11pt,letterpaper]{amsart}
\usepackage{amsmath,amsfonts}
\usepackage{graphicx}
\usepackage{amssymb}
\usepackage{amsthm}
\usepackage{yfonts}
\usepackage{kpfonts}
\usepackage{tikz-cd}
\usepackage{stmaryrd}
\usepackage{thm-restate}
\usepackage[toc]{appendix}

\usepackage{thmtools}
\usepackage{thm-restate}

\usepackage{hyperref}

\usepackage{mathtools}
\mathtoolsset{showonlyrefs}

\numberwithin{equation}{section}

\usepackage[usenames,dvipsnames]{pstricks}
\usepackage{epsfig}
\usepackage{rotating}
\usepackage{pst-plot}
\usepackage{pst-eps}
\usepackage{pst-grad}
\usepackage{pstricks-add}
\usepackage{lmodern}
\usepackage{xcolor}
\usepackage{graphicx}
\usepackage[top=3cm,bottom=3cm,left=3cm,right=3cm,a4paper]{geometry}

\usepackage{cite}

\def\eps{{\varepsilon}}

\newcommand{\qq}{{\mathbb Q}}

\newcommand{\pp}{{\mathbb P}}

\newcommand{\Pic}{{\rm{Pic}}}

\newcommand{\Hbars}[2]{\overline{\mathcal{H}}_{{#1}, {#2}}}

\newcommand{\Nefbar}{\overline{\text{Nef}}}

\newcommand{\Mbar}[2]{\overline{\mathcal{M}}_{{#1}, {#2}}}

\theoremstyle{plain}
\newtheorem{theorem}{Theorem}[section]
\newtheorem{lemma}[theorem]{Lemma}
\newtheorem{proposition}[theorem]{Proposition}
\newtheorem{corollary}[theorem]{Corollary}

\theoremstyle{definition}
\newtheorem{remark}[theorem]{Remark}

\title{The Kodaira classification of the moduli of hyperelliptic curves}
\author[I. Barros]{Ignacio Barros}
\author[S. Mullane]{Scott Mullane\smallskip\\
\MakeLowercase{with an appendix by} Irene Schwarz}

\address[I. Barros]{
Department of Mathematics, Universiteit Antwerpen\\
Middelheimlaan 1, 2020 Antwerpen, Belgium} 
\email{ignacio.barros@uantwerpen.be}

\address[S. Mullane]{
School of Mathematics and Statistics, University of Melbourne, VIC 3010, Australia} 
   \email{{\tt mullanes@unimelb.edu.au}}

\address[I. Schwarz]{
Departement Mathematik, ETH Z\"{u}rich\\
  R\"{a}mistrasse 101, 8092 Z\"{u}rich, Switzerland} 
   \email{{\tt irene.schwarz@math.ethz.ch}}

\vspace{-1cm}
\begin{document}

\maketitle
\vspace{-0.9cm}
\begin{abstract}
We study the birational geometry of the moduli spaces of hyperelliptic curves with marked points. We show that these moduli spaces have non $\mathbb{Q}$-factorial singularities. We complete the Kodaira classification by proving that these spaces have Kodaira dimension $4g+3$ when the number of markings is $4g+6$ and are of general type when the number of markings is $n\geq4g+7$. Similarly, we consider the natural finite cover given by ordering the Weierstrass points. In this case, we provide a full Kodaira classification showing that the Kodaira dimension is negative when $n\leq3$, one when $n=4$, and of general type when $n\geq 5$. For this, we carry out a singularity analysis of ordered and unordered pointed Hurwitz spaces. We show that the ordered space has canonical singularities and the unordered space has non-canonical singularities. We describe all non-canonical points and show that pluricanonical forms defined on the full regular locus extend to any resolution. Further, we provide a full classification of the structure of the pseudo-effective cone of Cartier divisors for the moduli space of hyperelliptic curves with marked points. We show the cone is non-polyhedral when the number of markings is at least two and polyhedral in the remaining cases.  
\end{abstract}

\tableofcontents


\nopagebreak



\section{Introduction}

The Kodaira dimension measures the size of the canonical model, and hence is a measure of the complexity of a projective variety $X$, while the structure of the pseudo-effective cone of divisors broadly dictates the geometry of the variety. For instance, when $X$ is $\mathbb{Q}$-factorial and the Cox ring is finitely generated, then the effective cone is not only polyhedral but further decomposes into finitely many convex chambers each representing a different birational model of the variety. Such a variety is known as a Mori dream space and when the variety is a moduli space, the chambers often have a modular meaning. If $X$ has canonical singularities then $X$ is of maximal Kodaira dimension, or of general type, when the canonical divisor lies in the interior of the pseudo-effective cone, and is of minimal Kodaira dimension and uniruled when the canonical divisor is not pseudo-effective, cf. \cite{BDPP}. When the Kodaira dimension of $X$ lies somewhere between these two extremes, $X$ is of intermediate type and the canonical divisor lies on the boundary of the pseudo-effective cone.\\

The computation of the Kodaira dimension of $\overline{\mathcal{M}}_{g,n}$ for $g\geq 2$ has been a central problem over the past decades. It is well established \cite{HM, Lo} that the coarse moduli spaces $\overline{\rm{M}}_{g,n}$ are all of general type with only finitely many exceptions. When the genus and number of markings are small these spaces are usually unirational or uniruled, but as $g$ or $n$ grows, the problem of determining the Kodaira dimension becomes much harder with long-standing open cases such as $\overline{\rm{M}}_{g}$ with $17\leq g\leq21$. The transition from negative to maximal Kodaira dimension is poorly understood and only two such moduli spaces are known to be of intermediate Kodaira dimension; ${\rm{Kod}}\left(\overline{\rm{M}}_{11,11}\right)=19$ and ${\rm{Kod}}\left(\overline{\rm{M}}_{10,10}\right)=0$, cf. \cite{FV, BM}. {In general terms, $\overline{\rm{M}}_{g,n}$ becomes more complicated as either $g$ or $n$ grows.} \\

Similarly, the structure of the pseudo-effective cone and the related Mori dream space question on $\overline{\rm{M}}_{g,n}$ for $g\geq 2$ has presented a challenge. The cone is non-polyhedral for $n\geq2$ and they are also not Mori dream spaces for $g\geq 3$ and $n\geq1$, cf. \cite{Mu, Ke}. A full chamber decomposition of the cone is known only in the remaining cases for very small genera $g=2,3$, cf. \cite{Ru}. Both questions remain open on $\overline{\rm{M}}_{g}$ for $g\geq 4$ and despite many authors contributing to the known effective divisors in $\overline{\rm{M}}_{g}$, e.g.~\cite{HM,FarkasPopa, F}, there are no known extremal rays of the pseudo-effective cone other than the boundary divisors for $g\geq12$. \\

{Perhaps the most fundamental locus in $\rm{M}_{g,n}$ is the locus of hyperelliptic curves.} In this paper, we study the coarse moduli space of marked hyperelliptic curves and their stable degenerations $\overline{\rm{H}}_{g,n}\subset \overline{\rm{M}}_{g,n}$. We have two main results seemingly pointing in opposite directions. The first might advocate for the simplicity and the second for the complexity of $\overline{\rm{H}}_{g,n}$. In strong contrast to $\overline{\rm{M}}_{g,n}$, the transition from negative to maximal Kodaira dimension is surprisingly uniform and happens exactly at $n=4g+6$ for all $g\geq 2$. Our first main result is the following:

\begin{restatable}{theorem}{thmone}
\label{thm1}
For all $g\geq 2$, the moduli space $\overline{\rm{H}}_{g,n}$ has Kodaira dimension 
$${\rm{Kod}}\left(\overline{\rm{H}}_{g,n}\right)=\left\{\begin{array}{lcl}
n-3&\hbox{if}& n=4g+6,\\
2g-1+n&\hbox{if}&n>4g+6.\end{array}\right.$$
\end{restatable}

Further, our investigation uncovers a new proof of the known uniruled results~\cite{Be, AB} which together provide a complete Kodaira classification.
It is remarkable that for fixed $n\gg0$, as the genus grows the moduli spaces $\overline{\rm{H}}_{g,n}$ change from being of general type to being uniruled. In some sense, they become ``simpler'', in contrast to the situation for $\overline{\rm{M}}_{g,n}$. In further contrast with $\overline{\rm{M}}_{g,n}$, that has only finite-quotient singularities, the moduli space $\overline{\rm{H}}_{g,n}$ is NOT $\qq$-factorial, making divisor calculations and positivity study of the canonical class unmanageable. Since the automorphism groups are finite, this implies that the stack $\overline{\mathcal{H}}_{g,n}$ is not smooth, answering a question posted in \cite{EH}. 

\begin{proposition}
\label{thm:Hgnsing}
For $g\geq 3$ and $n\geq 2$ the moduli space $\overline{\rm{H}}_{g,n}$ is normal and not $\qq$-factorial.
\end{proposition} 

In light of Proposition \ref{thm:Hgnsing} we study better behaved birational model of $\overline{\rm{H}}_{g,n}$. This is provided by the Hurwitz space $\overline{\rm{H}ur}_{g,2}^n$ parameterising $n$-pointed admissible double covers \cite{HM}.\\ 

We also consider the natural generically finite cover 
$$\widehat{\rm{H}}_{g,n}\longrightarrow \overline{\rm{H}}_{g,n}$$
induced by marking and ordering the Weierstrass points. The moduli space $\widehat{\rm{H}}_{g,n}\subset \overline{\rm{M}}_{g, 2g+2+n}$ is the closure of the locus of $(2g+2+n)$-pointed curves
$$\left(C, w_1,\ldots,w_{2g+2},p_1,\ldots,p_n\right)\in \rm{M}_{g,2g+2+n}$$
where $C$ is hyperelliptic, and $w_1,\ldots,w_{2g+2}$ are the Weierstrass points. The space $\widehat{\rm{H}}_{g,n}$ is again not $\mathbb{Q}$-factorial when $g\geq3$ and $n\geq2$ and the birational model we study in this case is the Hurwitz space $\widehat{\rm{H}ur}_{g,2}^n$ parameterising $(2g+2+n)$-pointed admissible double covers, where the first $2g+2$ makings are Weierstrass points. In this case we obtain the full Kodaira classification:

\begin{restatable}{theorem}{thmtwo}
\label{thm2}
For all $g\geq 2$, the moduli space $\widehat{\rm{H}}_{g,n}$ has Kodaira dimension 
$${\rm{Kod}}\left(\widehat{\rm{H}}_{g,n}\right)=\left\{\begin{array}{lcl}
-\infty& \hbox{if}& n\leq 3\\
1&\hbox{if}& n=4,\\
2g-1+n&\hbox{if}&n>4.\end{array}\right.$$
\end{restatable}

Moduli spaces of intermediate type parameterising curves of $g\geq 2$ are rare in nature. The transition from negative to maximal Kodaira dimension as the number of markings grows is rather sudden due to the positivity of the relative dualising sheaf for the universal family, see for instance the discussion in \cite[Section 1]{BM} applied to the hyperelliptic context. Theorems~\ref{thm1} and \ref{thm2} provide the first infinite collection of moduli spaces of curves of genus $g\geq2$ that are of intermediate type. \\

{Our second main result provides} a full classification for the structure of the cone of pseudo-effective Cartier divisors $\overline{\rm{Eff}}\left(\overline{\rm{H}}_{g,n}\right)$ which is as complicated as that of $\overline{\rm{M}}_{g,n}$ in the cases where the structure of the latter is known.

\begin{theorem}
\label{thm3}
The pseudo-effective cone of $\overline{\rm{H}}_{g,n}$ is
\begin{enumerate}
\item
generated by the irreducible components of the boundary for $n=0$,
\item
generated by the Weierstrass divisor and the irreducible components of the boundary for $n=1$,
\item
non-polyhedral and hence generated by infinitely many extremal rays for $n\geq 2$.
\end{enumerate}
\end{theorem}

We denote by $\overline{\mathcal{M}}_{0,[2g+2]+n}$ the stack quotient of $\overline{\mathcal{M}}_{0,2g+2+n}$ by the symmetric group $S_{2g+2}$ that acts by permuting the first $2g+2$ markings and we denote by $\overline{\rm{M}}_{0,[2g+2]+n}$ its coarse moduli space. When $n=0$ there is a regular birational map $\overline{\rm{H}}_g\rightarrow\overline{\rm{M}}_{0,[2g+2]}$ already used in \cite{Co} to described the Picard group of $\overline{\rm{H}}_g$. This map completely determines the birational type of the varieties $\overline{\rm{H}}_g$ making them rational. When $n>0$, there is still a generically finite rational map $\overline{\rm{H}}_{g,n}\dashrightarrow\overline{\rm{M}}_{0,[2g+2]+n}$ relating the geometry of $\overline{\rm{H}}_{g,n}$ with symmetric quotients of moduli spaces of pointed rational curves. When $n>0$ this rational map is resolved by the Hurwitz correspondence 

\begin{equation}
\label{intro:eq1}
\begin{tikzcd}
&\overline{\mathcal{H}ur}_{g,2}^n\arrow[dr, "\pi"]\arrow[dl, "\phi"']&\\
\overline{\mathcal{H}}_{g,n}&&\overline{\mathcal{M}}_{0,[2g+2]+n}.
\end{tikzcd}
\end{equation}
Here $\overline{\mathcal{H}ur}_{g,2}^n$ is the Deligne--Mumford stack of admissible double covers that parameterises stable $n$-pointed double covers $\left[f:C\to R, p_1,\ldots,p_n\right]$
with $R$ a rational curve and $p_i$ a marked point on the genus $g$ source curve $C$, and the image of all $p_i$'s and the branch points of $f$ are distinct. See Section \ref{section:hur} for definitions. The map $\phi$ forgets the cover, and $\pi$ forgets the source curve while remembering branch points and the images under $f$ of the markings. The generically finite forgetful morphism $\widehat{\mathcal{H}}_{g,n}\to\overline{\mathcal{H}}_{g,n}$ lifts to the finite $S_{2g+2}$-quotient on admissible covers 
\[
\widehat{\mathcal{H}ur}_{g,2}^n\longrightarrow\overline{\mathcal{H}ur}_{g,2}^n,
\]
where $\widehat{\mathcal{H}ur}_{g,2}^n$ parameterises $n$-pointed admissible double covers where we mark the $2g+2$ ramification points of the cover outside the nodes. 
Both unordered and ordered Hurwitz spaces are largely studied objects in their own right. However, they also have a long history of crucial importance in known results on the geometry of $\overline{\mathcal{M}}_{g,n}$ including the dimension, irreducibility, dimension of proper subvarieties, and the birational classification, cf. ~\cite{Riemann, Clebsch, Hurwitz, Diaz, HM}.
A modern example of this is in studying the tautological ring $R^\star(\overline{\mathcal{M}}_{g,n})$, cf. \cite{GP, FPa, SvZ, CL3}. Despite this interest and utility, many more questions remain open on the birational geometry of these spaces.\\

The stacks $\widehat{\mathcal{H}ur}_{g,2}^n$, and $\overline{\mathcal{H}ur}_{g,2}^n$ are smooth Deligne--Mumford \cite{JKK, SvZ} and their coarse spaces have finite quotient singularities. Further, the branch morphism on coarse spaces
$$\pi:\overline{\rm{H}ur}_{g,2}^n\to\overline{\rm{M}}_{0,[2g+2]+n}$$
and the forgetful morphism from the ordered to unordered Hurwitz space composed with this branch morphism
\[
\hat{\pi}:\widehat{\rm{H}ur}_{g,2}^n\to\overline{\rm{M}}_{0,[2g+2]+n}
\]
both share the property of saturated branching with simple ramification in codimension one. That is, the branch and ramification divisors are related as
$${\rm{Ram}}\left({\pi}\right)=\frac{1}{2}\pi^*{\rm{Branch}}\left({\pi}\right) \;\; \;\;\;\text{and}\;\;\;\;\;{\rm{Ram}}\left(\hat{\pi}\right)=\frac{1}{2}\pi^*{\rm{Branch}}\left(\hat{\pi}\right).  $$
This allows us to express both canonical divisors $K_{\overline{\rm{H}ur}_{g,2}^n}$ and $K_{\widehat{\rm{H}ur}_{g,2}^n}$ as the pullbacks of classes on $\overline{\rm{M}}_{0,[2g+2]+n}$ under the finite morphisms $\pi $ and $\hat{\pi}$ respectively. Hence the positivity properties of the the canonical divisors are dictated by the positivity properties of the corresponding divisors on $\overline{\rm{M}}_{0,[2g+2]+n}$. \\

The Picard group ${\rm{Pic}}_{\mathbb{Q}}\left(\overline{\rm{M}}_{0,[2g+2]+n}\right)$ is generated by boundary divisors subject to large collection of relations known as {\textit{Keel's relations}}. They all come from the fundamental relation on $\overline{\mathcal{M}}_{0,4}\cong \mathbb{P}^1$ by pulling back via different forgetful morphisms. We use these relations to find the right expression for the canonical classes ensuring effectivity when $n\geq 4g+6$ in the unordered case and $n\geq 4$ in the ordered case. Further, we show the canonical class is big for $n> 4g+6$ in the unordered case and $n> 4$ in the ordered case.\\

A big divisor has maximal Iitaka dimension and lies on the interior of the effective cone, while a divisor of minimal Iitaka dimension is not effective and hence lies outside of the effective cone. Outside of these two extremes, are divisors of intermediate Iitaka type which lie on the boundary of the effective cone. Determining the Iitaka dimension of such a divisor is often more challenging. \if{as understanding the asymptotic behaviour of global sections in these intermediate cases is a more delicate proposition}\fi This added complexity is evidenced by the existence of only two confirmed cases where $\overline{\rm{M}}_{g, n}$ is of intermediate type.\\

In this paper we present a successful new approach to these more challenging intermediate type cases via
the cone of moving or nef curves that is dual to the cone of pseudo-effective Cartier divisors. Our investigation of the curve geometry of $\overline{\rm{M}}_{0,[2g+2]+n}$, leads us to the existence of a special nef curve $M$ with the property that any effective divisor $D$ satisfying $M\cdot[D]=0$ must satisfy the strong bound
$$\kappa\left(\overline{\rm{M}}_{0,[2g+2]+n}, D\right)\leq \max\{0,n-3\}.$$ 
This property utilises the existence of a large number of effective decompositions or splittings of this nef curve, a new approach that may prove applicable to other moduli spaces. 
This means that for $k$ big enough and $n\geq3$, the dimension of the image of the rational map induced by the linear system $\left|kD\right|$ 
$$\phi_{\left|kD\right|}:\overline{\rm{M}}_{0,[2g+2]+n}\dashrightarrow \pp^N$$
has dimension at most $n-3$. The nef curve $M$ intersects $D$, the divisor in $\overline{\rm{M}}_{0,[2g+2]+n}$ such that $\pi^*D=K_{\overline{\rm{H}ur}_{g,2}^n}$, negatively when $n\leq 4g+5$ and has zero intersection when $n=4g+6$. Hence the canonical divisor is not effective for $n\leq 4g+5$ which immediately replicates this known bound on negative Kodaira dimension by Benzo and Agostini--Barros \cite{Be, AB}. However, with our singularity analysis this further shows that the Kodaira dimension $\overline{\rm{H}ur}_{g,2}^n$ (and therefore $\overline{\rm{H}}_{g,n}$) is bounded above by $n-3$ when $n=4g+6$.\\

The lower bound
\begin{equation}
\label{intro:eq_low}
n-3\leq\kappa\left(\overline{\rm{H}ur}_{g,2}^n, K_{\overline{\rm{H}ur}_{g,2}^n}\right)  
\end{equation}
when $n=4g+6$ is obtained by expressing $D$, the divisor such that $\pi^*D=K_{\overline{\rm{H}ur}_{g,2}^n}$, as an effective sum of boundary divisors and the pullback of a big divisor on $\overline{\rm{M}}_{0,n}$ via the forgetful morphism $\overline{\rm{M}}_{0,[2g+2]+n}\longrightarrow \overline{\rm{M}}_{0,n}$. The argument for the ordered space is similar. In order for all of our Iitaka dimension calculations to translate into actual Kodaira dimension we need to control the singularities of $\widehat{\rm{H}ur}_{g,2}^n$ and $\overline{\rm{H}ur}_{g,2}^n$. Using the smoothness of the fine moduli space $\overline{\mathcal{M}}_{0,2g+2+n}$ we show the following theorem:

\begin{theorem}
\label{thmfive}
The coarse moduli space $\widehat{\rm{H}ur}_{g,2}^n$ has canonical singularities for all $g,n$ with $g\geq 2$ and $n\geq 0$.
\end{theorem}

The unordered space $\overline{\rm{H}ur}_{g,2}^n$ requires a much more delicate treatment because it has non-canonical singularities. Our study begins with the Reid--Tai criterion, cf. \cite{HM, Lu, FL}. We use the extra structure provided by the cover to obtain lower bounds on the \textit{age} of a singularity by looking at the induced action of the group of automorphisms of $f$ on the $([2g+2]+n)$-pointed rational curve. We obtain the following:

\begin{theorem}
\label{thmseven}
An $n$-pointed admissible cover $[f:C\to R, p_1,\ldots,p_n]\in\overline{\rm{H}ur}_{g,2}^n$ of genus $g\geq 2$ is a non-canonical singularity if and only if it has an elliptic tail $E$ with no markings and $j$-invariant $0$. If we fix the node at the origin $0\in E$, then $f$ restricts to the quotient $f:E\to E\big/\langle\pm 1\rangle\cong \mathbb{P}_1$ and the remaining Weierstrass points are sent to a $\mathbb{Z}\big/3\mathbb{Z}$-orbit $\{z, \eta z, \eta^{2}z\}\subset \mathbb{P}^1$, with $\eta$ a cubic root of unity.
\end{theorem}

Curves with an elliptic tail of $j$-invariant $0$ are a recurrent source of non-canonical singularity on moduli spaces of curves. They admit an automorphism $g$ of order $6$ where $g^3$ is the elliptic involution and $g^2$ permutes the three non-zero $2$-torsion points. One observes $g^2$ is induced by the order three automorphism on the quotient $E/g^3\cong \mathbb{P}^1$. This appeared first \cite{HM}, and then in \cite[Thm. 3.1]{Lu} for moduli spaces of spin curves and \cite[Thm. 6.7]{FL} for moduli spaces of Prym curves. The degree two cover structure plays an essential role throughout our argument.

Once the non-canonical locus is described, the argument for extending pluricanonical forms works more or less uniformly in all these cases. In our case we have:

\begin{theorem}
\label{thmsix}
For $g\geq 2$ and $n\geq 0$ the Kodaira dimension of $\overline{\rm{H}ur}_{g,2}^n$ and the Iitaka dimension of the canonical divisor $\kappa\left(\overline{\rm{H}ur}_{g,2}^n, K_{\overline{\rm{H}ur}_{g,2}^n}\right)$ coincide. In other words, if
\[
h:Y\longrightarrow \overline{\rm{H}ur}_{g,2}^n
\]
is a resolution of singularities and $\left(\overline{\rm{H}ur}_{g,2}^n\right)_{reg}\subset \overline{\rm{H}ur}_{g,2}^n$ is the regular locus, then for some integer $m>0$ there is an isomorphism 
\[
H^0\left(\left(\overline{\rm{H}ur}_{g,2}^n\right)_{reg}, K_{\left(\overline{\rm{H}ur}_{g,2}^n\right)_{reg}}^{\otimes m}\right)\cong H^0\left(Y, K_Y^{\otimes m}\right).
\]
\end{theorem}

Finally, to prove Theorem \ref{thm3} we follow the strategy of~\cite{Mu} and identify an extremal ray of the dual cone of nef curves $\overline{\rm{Nef}}_1(\overline{\rm{H}}_{g,n})$  where the cone is not polyhedral. As in $\overline{\rm{M}}_{g,n}$, the general fibre of the forgetful morphism
\[
\pi:\overline{\rm{H}}_{g,n}\longrightarrow\overline{\rm{H}}_{g,n-1}
\]
forms a nef curve class that we denote $[F]$. For any curve class $[B]$ we define the pseudo-effective dual space by 
\[
[B]^\vee:=\left\{[D]\in \overline{\rm{Eff}}^1(\overline{\rm{H}}_{g,n})\hspace{0.3cm}\big|\hspace{0.3cm} [B]\cdot[D]=0\right\}.
\]
By analysing the fibres of the forgetful morphisms we bound the rank of the strictly effective divisors appearing in $[F]^\vee$. Then by considering the intersection of any divisor in $[F]^\vee$ with the surfaces obtained from 
\[
i:C\times C\longrightarrow \overline{\rm{H}}_{g,n}
\]
for fixed general genus $g$ hyperelliptic curve $C$, we obtain in Proposition~\ref{prop:corank}, bounds on the rank of $[F]^\vee$. For $g\geq 2$ and $n\geq2$, 
\[
\rho(\overline{\rm{H}}_{g,n})-n\leq\text{rank}([F]^{\vee}\otimes\mathbb{R})\leq \rho(\overline{\rm{H}}_{g,n})-2,
\]
where $\rho(\overline{\rm{H}}_{g,n})$ denotes the Picard number of $\overline{\rm{H}}_{g,n}$.\\

The remaining task is to show that $[F]$ is indeed an extremal nef curve. This is achieved by identifying a series of effective divisors that contradict the nefness of curve classes appearing in any non-trivial nef decomposition of the curve $[F]$. The required divisors come from Hurwitz space constructions that Lemma~\ref{lem:eff} shows provides codimension one conditions. The classes for the corresponding divisors in $\overline{\rm{M}}_{g,n}$ were computed in~\cite{Mu} and pullback to give the required classes in $\overline{\rm{H}}_{g,n}$.\\

Finally, the morphisms $\overline{\rm{H}}_{g}\overset{\sim}{\longrightarrow} \overline{\rm{M}}_{0,[2g+2]}$ and $ \overline{\rm{M}}_{0,[2g+2]+1}{\longrightarrow}\overline{{\rm{H}}}_{g,1}$ are used to show polyhedrality of the pseudo-effective cone when $n=0,1$.\\

\subsection*{Outline}The outline of the paper is as follows. In Section \ref{section:hur} we recall some basic facts about pointed admissible double covers and Hurwitz spaces and show Proposition \ref{thm:Hgnsing} by showing that exceptional locus of the forgetful map $\overline{\rm{H}ur}_{g,2}^n\to\overline{\rm{H}}_{g,n}$ contains an irreducible component of codimension two. In Section \ref{section:sing_ordered} we study the singularities of the coarse spaces $\widehat{\rm{H}ur}_{g,2}^n$ and prove Theorem \ref{thmfive}. In Section \ref{section:can_class} we use the branch morphisms to compute the canonical class of $\widehat{\rm{H}ur}_{g,2}^n$ and $\overline{\rm{H}ur}_{g,2}^n$ and show effectivity for $n\geq 4$ (resp. $n\geq 4g+6$) and bigness for $n=5$ (resp. $n\geq 4g+7$). In Section \ref{section:M} we construct nef curves on $\overline{\rm{M}}_{0,[2g+2]+n}$ and study the numerics to conclude various splittings. This is the main tool for the intermediate Kodaira dimension results. In Section \ref{section:KodDim} we complete the proof of Theorem \ref{thm2} and we prove Theorem \ref{thm1} assuming Theorem \ref{thmsix}. In Section \ref{section:eff} we study the cone of effective Cartier divisors of $\overline{\rm{H}}_{g,n}$ and prove Theorem \ref{thm3}. Finally in Section \ref{section:sing_unordered} we carry out the singularity analysis of unordered pointed Hurwitz spaces and prove Theorems \ref{thmseven} and \ref{thmsix}. In Appendix \ref{Appendix} Irene Schwarz provides a calculation of the canonical class of $\overline{\rm{H}}_{g,n}$ and shows that is effective when $n\geq 4g+6$ and can be written as the sum of big plus effective when $n\geq 4g+7$.\\

\subsection*{Conventions}
We work over the complex numbers. We will denote by $\overline{\mathcal{M}}_{g,n}$ the moduli stack and $\overline{\rm{M}}_{g,n}$ the (coarse) moduli space of stable $n$-pointed curves of genus $g$. Similarly, $\overline{\mathcal{H}}_{g,n}$ is defined as the stack-theoretic inverse image of the closure of the stack of hyperelliptic curves ${\mathcal{H}}_{g}\hookrightarrow\overline{\mathcal{M}}_g$ via the forgetful morphism $\overline{\mathcal{M}}_{g,n}\to \overline{\mathcal{M}}_g$. Objects in $\overline{\mathcal{H}}_{g,n}(T)$ are $T$-families of stable models of $n$-pointed admissible double covers when forgetting the cover structure. We denote by $\overline{\rm{H}}_{g,n}\subset\overline{\rm{M}}_{g,n}$ the coarse space. Similarly, $\widehat{\mathcal{H}ur}_{g,2}^n$ and $\overline{\mathcal{H}ur}_{g,2}^n$ will denote the stack of ordered/unordered pointed admissible double covers and $\widehat{\rm{H}ur}_{g,2}^n$, $\overline{\rm{H}ur}_{g,2}^n$ the coarse spaces.\\

For a normal projective variety $X$ and a $\mathbb{Q}$-Cartier divisor $D$, we denote by $\kappa\left(X, D\right)$ the {\textit{Iitaka dimension}} of $D$ defined as follows. Let $k>0$ be the smallest positive integer such that $kD$ is Cartier and consider the set
\[
N(D)=\left\{m\in \mathbb{Z}_{>0}\mid H^0\left(X, \mathcal{O}_{X}\left(kD\right)^{\otimes m}\right)\neq 0\right\}.
\]
Then
\[
\kappa\left(X,D\right)=\left\{\begin{array}{lcl}-\infty&\hbox{if}&N(D)=\varnothing,\\
{\rm{max}}\left\{{\rm{dim}}(\phi_{m, D}(X))\mid m\in N(D)\right\}&\hbox{otherwise.}&\end{array}\right.
\]
Here $\phi_{m,D}:X\dashrightarrow \mathbb{P}^n$ is the rational map induced by the linear system $\left|mkD\right|$ on $X$. The {\textit{Kodaira dimension}} of $X$, denoted ${\rm{Kod}}\left(X\right)$, is defined as the Iitaka dimension of the canonical divisor $\kappa\left(Y,K_Y\right)$, where $Y\longrightarrow X$ is a resolution of singularities of $X$. Recall that $\kappa\left(Y,K_Y\right)$ is a birational invariant between smooth projective varieties, in particular the definition does not depend on the resolution, making ${\rm{Kod}}(X)$ birational invariant.

Let $\mathcal{X}$ be a smooth Deligne--Mumford stack and $\mathcal{X}\longrightarrow X$ the coarse space. Then the push-forward induces a canonical isomorphism \cite{Vis} of graded rings ${\rm{CH}}_{\mathbb{Q}}^{\star}\left(\mathcal{X}\right)\cong {\rm{CH}}_{\mathbb{Q}}^{\star}\left(X\right)$. We use the common convention of defining divisors via reduced geometric loci on the stack and using the same notation to denote the unique divisor on the coarse space that pulls back to give the reduced geometric locus on the stack (see Remark~\ref{rem:stackram}).  Hence though our intersection theoretic calculations take place on the stack, the same holds for the intersection of the pushforward of our curves on the coarse space.

\subsection*{Acknowledgments}

This paper benefited from correspondence and conversations with Daniele Agostini, Fabio Bernasconi, Gabi Farkas, Brian Lehmann, Emanuele Macr\`{i}, Siddharth Mathur, Johannes Schmitt, and Irene Schwarz. Special thanks go to Ana-Maria Castravet and Jenia Tevelev for pointing out an issue with an earlier version of this paper. Further, we would like to warmly thank the anonymous referee for many useful comments that helped us improve the mathematics, exposition, and presentation of this paper. The first author was supported by the ERC Synergy Grant ERC-2020-SyG-854361-HyperK, the Deutsche Forschungsgemeinschaft (DFG, German Research Foundation) -- SFB-TRR 358/1 2023 -- 491392403, and the Research Foundation Flanders (FWO) within the framework of the
Odysseus program project number G0D9323N. The second author was supported by the Alexander von Humboldt Foundation, ERC Advanced Grant SYZYGY, and DECRA Grant DE220100918 from the Australian Research Council.
\section{Pointed admissible double covers and the hyperelliptic locus}
\label{section:hur}

Let ${\rm{H}ur}_{g,2}^n$ be the Hurwitz space parameterising isomorphism classes of marked degree two smooth covers 
\[
\left[f:C\overset{2:1}{\longrightarrow}\pp^1, p_1,\ldots,p_n\right]\in{\rm{H}ur}_{g,2}^n
\]
where the marked points $p_1,\dots,p_n$ are distinct points of $C$, unramified in $\pi$, and no two have the same image under $\pi$. We denote by $\overline{{\rm{H}ur}}_{g,2}^n$ the compactification by admissible covers defined in \cite{HM}. In the boundary this imposes the conditions that the markings $p_i$ cannot be nodes or ramification points of the nodal cover and any two markings $p_i, p_j$ on $C$ do not map to the same point on the target stable nodal rational curve marked with the branch points and the images of the marked points. More concretely, $\overline{\rm{H}ur}_{g,2}^{n}$ is the coarse space of the smooth stack $\overline{\mathcal{H}ur}_{g,2}^n$ where a $T$-point $T\longrightarrow\overline{\mathcal{H}ur}_{g,2}^n$ are flat $T$-families;
\begin{equation}
\label{eqn:univ_ad.cover}
\begin{tikzcd}
\mathcal{C}\arrow[rr, "f_T"]\arrow[dr]&&\mathcal{P}\arrow[dl]\\
&T,\arrow[bend left=40, ul, "\sigma_j"]&
\end{tikzcd}
\end{equation}
such that:
\begin{itemize}
\item[(i)] The map $\mathcal{C}\longrightarrow T$ is a family of $n$-marked semi-stable curves of genus $g$ with $\sigma_1,\ldots,\sigma_n:T\longrightarrow\mathcal{C}$ the $n$ disjoint sections not intersecting the nodes and Weierstrass points of the fibres. With the extra condition that no two points in $\sigma_1(t),\ldots,\sigma_n(t)$ are conjugate.
\item[(ii)] The family $\mathcal{P}\longrightarrow T$ is fibrewise a tree of rational curves and $f_T:\mathcal{C}\to\mathcal{P}$ is fibrewise an admissible cover of degree $2$.  
\item[(iii)] The map $\mathcal{P}\longrightarrow T$ together with the induced sections $\tau_1,\ldots,\tau_n:T\longrightarrow\mathcal{P}$, and the relative branch divisor away from the nodes is an element in the $S_{2g+2}$-quotient of $\overline{\mathcal{M}}_{0,2g+2+n}(T)$, where the symmetric group acts by permuting the first $2g+2$ markings. 

\end{itemize}

Further, a $T$-isomorphism between two $n$-pointed admissible covers $\left[f_T:\mathcal{C}\longrightarrow \mathcal{P}, \sigma_1,\ldots,\sigma_n\right]$ and $\left[f'_T:\mathcal{C}'\longrightarrow \mathcal{P}', \sigma'_1,\ldots,\sigma'_n\right]$ is a commutative diagram over $T$:
\begin{equation}
\label{eq:isoHur}
\begin{tikzcd}
\mathcal{C}\arrow[r, "g"]\arrow[d, "f"']&\mathcal{C}'\arrow[d,"f'"]\\
\mathcal{P}\arrow[r, "h"]&\mathcal{P}',
\end{tikzcd}
\end{equation}
where $g$ is a $T$-isomorphism of $n$-pointed curves and $h$ is a $T$-isomorphism of the induced $([2g+2]+n)$-pointed rational curves, i.e., it sends $f\circ\sigma_i(T)$ to $f'\circ\sigma'_i(T)$ and the branch divisor of $f_T$ to the branch divisor of $f'_T$.\\

The stack $\overline{\mathcal{H}ur}_{g,2}^n$ comes with two natural forgetful morphisms that we call $\phi$ and $\pi$ (we keep the same notation for the induced maps on coarse spaces). The map $\phi:\overline{\mathcal{H}ur}_{g,2}^n\to\overline{\mathcal{H}}_{g,n}$ forgets the cover and remembers the stable model of the pointed source curve. The second map $\pi:\overline{\mathcal{H}ur}_{g,2}^n\to\overline{\mathcal{M}}_{0,[2g+2]+n}$ is a forgetful morphism that remembers only the target rational curve together with the image of the markings and ramification points. They induce the following diagram on coarse spaces:
\begin{equation}  
\label{eqn:Hurwitz}
\begin{tikzcd}
&\overline{\rm{H}ur}_{g,2}^n\arrow[dl, "\phi"']\arrow[dr, "\pi"]&\\
\overline{\rm{H}}_{g,n}&&\overline{\rm{M}}_{0,\left[2g+2\right]+n}.
\end{tikzcd}
\end{equation}

The map $\pi$ is finite of degree $2^{n-1}$ for $n\geq1$ and of degree $1$ for $n=0$. On the other hand, $\phi$ is only generically finite. Analogously one defines the stack $\widehat{\mathcal{H}ur}_{g,2}^n$ with the extra condition that $\mathcal{C}\to T$ has $2g+2$ extra sections corresponding to the Weierstrass points. An isomorphism between two {\textit{ordered}} admissible covers $f$ and $f'$ is the same as in \eqref{eq:isoHur} where $g$ is an isomorphism between $(2g+2+n)$-pointed curves and $h$ an isomorphism of the corresponding ordered $(2g+2+n)$-rational curves. It sits in the following diagram:

 \begin{equation}  
\label{eqn:Hurwitzhat}
\begin{tikzcd}
&\widehat{\mathcal{H}ur}_{g,2}^n\arrow[dl, "\widehat{\phi}"']\arrow[dr, "\widehat{\pi}"]&\\
\widehat{\mathcal{H}}_{g,n}\subset \overline{\mathcal{M}}_{g,2g+2+n}&&\overline{\mathcal{M}}_{0,2g+2+n}.
\end{tikzcd}
\end{equation}

There are several constructions of spaces of admissible covers in the literature \cite{HM, ACV, JKK, SvZ}. In particular, the stack $\overline{\mathcal{H}}_{g,\mu_2,\xi}$ of admissible $G$-covers defined in \cite[Section 3.2]{SvZ} coincides with $\widehat{\mathcal{H}ur}_{g,2}^n$ when the group $G$ is the group of $2$-roots of unity $\mu_2=\{\pm1\}$, the number of markings is $2g+2+n$, and the monodromy data is given by 
\[
\xi=(\underbrace{-1,\ldots,-1}_{2g+2},\underbrace{1,\ldots,1}_{n}).
\]
See the remark before \cite[Definition 3.3]{SvZ}. A $T$-point $T\longrightarrow \overline{\mathcal{H}}_{g,\mu_2,\xi}$ corresponds to a $\mu_2$-admissible cover $\mu_2\curvearrowright\mathcal{C}\longrightarrow \mathcal{P}$ over $T$ with sections $x_1,\ldots,x_{2g+2}:T\longrightarrow \mathcal{C}$ and 
\[
p_{1,1}, p_{1,-1},\ldots,p_{n,1}, p_{n,-1}:T\longrightarrow \mathcal{C}
\]
such that $p_{i,1}$ and $p_{i,-1}$ are $\mu_2$ conjugates, and the $x_i's$ are ramification points of the cover outside the nodes, i.e., points for which the stabilizer of the $\mu_2$ action is $\langle -1\rangle=\mu_2$.

\begin{lemma}
\label{sec2:lemmaSvZ}
Let $\overline{\mathcal{H}}_{g,\mu_2,\xi}$ be the stack of pointed admissible $\mu_2$-covers with monodromy datum defined above. Then $\overline{\mathcal{H}}_{g,\mu_2,\xi}$ and $\widehat{\mathcal{H}ur}_{g,2}^n$ are isomorphic. In particular, both $\widehat{\mathcal{H}ur}_{g,2}^n$ and $\overline{\mathcal{H}ur}_{g,2}^n$ are smooth Deligne--Mumford stacks and their coarse spaces have finite quotient singularities.
\end{lemma}
\begin{proof}
An isomorphism between two $T$-points of $\overline{\mathcal{H}}_{g,\mu2,\xi}$, $\alpha=[\mathcal{C}\longrightarrow \mathcal{P}, x_i, p_{i,\pm1}]$ and $\alpha'=[\mathcal{C}'\longrightarrow \mathcal{P}', {x'_i}, p'_{i,\pm1}]$ is a $\mu_2$-equivariant $T$-map $\mathcal{C}\longrightarrow \mathcal{C}'$ preserving the markings. Note that this is the same as an isomorphism of ordered admissible covers $\beta=[\mathcal{C}\longrightarrow \mathcal{P}, x_i, p_{i,1}]$ and $\beta'=[\mathcal{C}'\longrightarrow \mathcal{P}', {x'_i}, p'_{i,1}]$. Indeed if $\phi_T: \mathcal{C}\longrightarrow \mathcal{C}'$ is a $\mu_2$-equivariant map preserving the $2g+2+2n$ markings, then is induces a diagram \eqref{eq:isoHur} where the marked Weierstrass points $x_i:T\to\mathcal{C}$ are sent to the marked Weierstrass points $x_i':T\longrightarrow \mathcal{C}'$, and the markings $p_{i,1}:T\to\mathcal{C}$ to the markings $p'_{i,1}$. Conversely if $\phi_T$ is an isomorphism of $T$-points in $\widehat{\mathcal{H}ur}_{g,2}^n$, i.e., ordered admissible covers $\beta$ and $\beta'$, since it is given by a $T$-diagram \eqref{eq:isoHur}, the map $g: \mathcal{C}\longrightarrow \mathcal{C}'$ is $\mu_2$-equivariant and by declaring $p_{i,-1}$ the be the congujate of the marked point $p_{i,1}$, the (sections) markings $x_i$ are sent to $x_i'$, $p_{i,1}$ to $p'_{i,1}$, and the conjugate points $p_{i,-1}$ must be sent to $p_{i,-1}$. The morphism of stacks $\overline{\mathcal{H}}_{g,\mu_2,\xi}\longrightarrow \widehat{\mathcal{H}ur}_{g,2}^n$ defined on $T$-points by 
\[
\alpha\mapsto \beta\;\;\hbox{and}\;\; \left({\rm{Isom}}_{\overline{\mathcal{H}}_{g,\mu_2,\xi}}\left(\alpha,\alpha'\right)\to T\right)\overset{=}\mapsto\left({\rm{Isom}}_{\widehat{\mathcal{H}ur}_{g,2}^n}\left(\beta, \beta'\right)\to T\right)
\]
is an isomorphism with inverse given on objects by marking the conjugate points $p_{i,-1}$ and by equality on morphisms. From \cite[Thm. 3.7]{SvZ} follows that $\widehat{\mathcal{H}ur}_{g,2}^n$ is a smooth and proper Deligne--Mumford stack. Further, the $S_{2g+2}$-quotient $\widehat{\mathcal{H}ur}_{g,2}^n\longrightarrow \overline{\mathcal{H}ur}_{g,2}^n$ is \'{e}tale.
\end{proof}

\subsection{The exceptional locus}

Recall the standard isomorphism $\Pic_{\mathbb{Q}}(\overline{\mathcal{M}}_{g, n})\cong \Pic_{\mathbb{Q}}\left(\overline{\rm{M}}_{g,n}\right)$ induced by the coarse map $\overline{\mathcal{M}}_{g, n}\longrightarrow\overline{\rm{M}}_{g, n}$. The Picard group is generated by $\lambda$, the first Chern class of the Hodge bundle, $\psi_i$ the first Chern class of the cotangent bundle on $\overline{\mathcal{M}}_{g,n}$ associated with the $i$-th marking for $i=1,\dots,n$, and the classes of the irreducible components of the boundary. We denote by $\delta_{\rm{irr}}$ the class being sent to the reduced locus $\Delta_{\rm{irr}}$ of curves with a non-separating node and $\delta_{i:S}$ for $0\leq i\leq g$, $S\subset\{1,\dots,n\}$ the class of the locus $\Delta_{i:S}$ of curves with a separating node that separates the curve such that one component has genus $i$ and contains precisely the markings of $S$. Hence $\delta_{i:S}=\delta_{g-i:S^c}$ and we require $|S|\geq 2$ for $i=0$ and $|S|\leq n-2$ for $i=g$. For $g\geq 3$ these divisors freely generate $\Pic_{\mathbb{Q}}(\overline{\mathcal{M}}_{g, n})$. Further, $\delta_{1:\varnothing}$ is defined as the class in $\overline{\mathcal{M}}_{g,n}$ that corresponds to $\frac{1}{2}\Delta_{1:\varnothing}$ on the coarse space $\overline{\rm{M}}_{g,n}$, accounting for the fact that the general element in $\Delta_{1:\varnothing}$ has automorphism group of order $2$. \\

We will denote by $\overline{\mathcal{M}}_{g,\left[n\right]+m}$ the stack $S_n$-quotient of $\overline{\mathcal{M}}_{g,n+m}$, where $S_n$ acts freely by permuting the first $n$ marked points. For every subset $S\subset\{n+1,\ldots,n+m\}$, the divisor 
\begin{equation}
\label{eq:sec2_delta}\delta_{i:t+S}=\sum_{\substack{T\subset \{1,\ldots,n\}\\|T|=t}}\delta_{i:T\cup S}\in {\rm{Pic}}_\qq\left(\overline{\mathcal{M}}_{g,n+m}\right)
\end{equation}
is $S_n$-invariant and descends to a divisor
$$\delta_{i:\left[t\right]+S}\in{\rm{Pic}}_\qq\left(\overline{\mathcal{M}}_{g,\left[n\right]+m}\right).$$
When $m=0$ we simply write $\delta_{i:\left[t\right]}$. We will be mostly working with $\overline{\mathcal{M}}_{0,[2g+2]+n}$. When $g=0$ instead of $\delta_{0:[t]+S}$ we write $\delta_{[t]+S}$. We denote by $\psi_j$ the usual $\psi$-classes on the ordered markings $j=n+1,\ldots,n+m$ and $\psi_\star$ the class in $\overline{\mathcal{M}}_{g,\left[n\right]+m}$ induced by the $S_n$-invariant class $\psi_1+\ldots+\psi_n$ on $\overline{\mathcal{M}}_{g,n+m}$. Recall that ${\rm{Pic}}_\qq\left(\overline{\mathcal{M}}_{g,\left[n\right]+m}\right)$ is generated by $\lambda, \psi_\star, \psi_1,\ldots,\psi_m$ and boundary divisors $\delta_{i:\left[t\right]+S}$. We will denote by $\delta_{i:[t]+j}$ the sum
$$\delta_{i:[t]+j}:=\sum_{\left|S\right|=j}\delta_{i:[t]+S}.$$
and again, when $g=0$ we write
$$\delta_{[t]+j}:=\sum_{\left|S\right|=j}\delta_{[t]+S}.$$

\begin{figure}[h]
\includegraphics[scale=0.15]{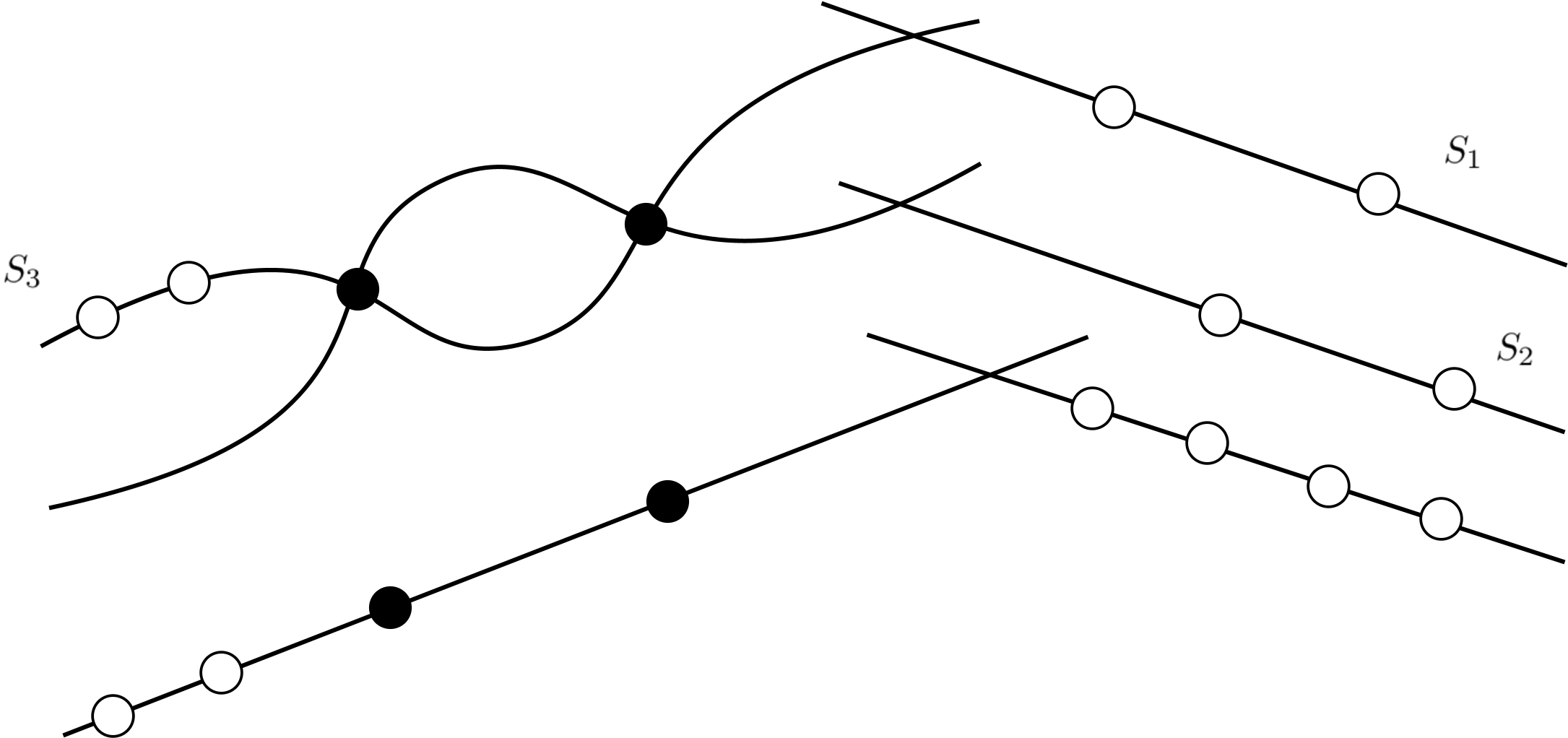}
\caption{Marked admissible double cover in the boundary $E_{(S_1,S_2)}$ (filled points are unordered).}
\label{figure:Pic1}
\end{figure}

As usual, we will denote by $\psi_j$ the $\psi$-class in 
$\overline{\mathcal{H}}_{g,n}$, $\overline{\mathcal{H}ur}_{g,2}^n$, and $\overline{\mathcal{M}}_{0,\left[2g+2\right]+n}$ corresponding to the $j$-th ordered marking. More precisely, let $p:\mathcal{C}\to B$ be a $B$-point of any of the moduli stacks in question, and $\sigma_j$ is a section. Then $\psi_j$ on $B$ is the first Chern class of $\sigma_j^*\Omega_p^1$, where $\Omega_p^1$ is the sheaf of relative $1$-forms of the family. Finally for any $3$-partition of the set of ordered markings 
\[
\mu=\left(S_1,S_2,S_3\right)\quad\hbox{with}\quad S_1\sqcup S_2\sqcup S_3=\{1,\ldots,n\},
\] 
we denote by $E_{\mu}$ or $E_{\left(S_1,S_2\right)}$ the divisor on $\overline{{\rm{H}ur}}_{g,2}^n$ (resp. $\widehat{E}_{\mu}$ or $\widehat{E}_{\left(S_1,S_2\right)}$ in $\widehat{{\rm{H}ur}}_{g,2}^n$) whose general element corresponds to a cover having two rational tails exchanged by the hyperelliptic involution, where the markings indexed by $S_1$ sit on one rational tail, the markings indexed by $S_2$ on the conjugate rational tail, and the markings indexed by $S_3$ sit on the smooth genus $g$ component. Observe that $E_{(S_1,S_2)}=E_{(S_2,S_1)}$. See Figure 1.\\ 

\begin{proposition}
\label{prop:exc_loc}
The divisorial exceptional locus of $\phi:\overline{\rm{Hur}}_{g,2}^n\longrightarrow \overline{\rm{H}}_{g,n}$ is given by the union 
\[
{\rm{div.exc}}(\phi)=\bigcup_{\substack{S_1,S_2\neq\varnothing\\\left|S_1\cup S_2\right|>2}} E_{(S_1,S_2)}\cup \bigcup_{\substack{i=1,2\\ i+\left|S\right|\geq 3}}\pi^{-1}\left(\delta_{[i]+S}\right).
\]
Similarly, the divisorial exceptional locus of $\widehat{\phi}:\widehat{\rm{Hur}}_{g,2}^n\longrightarrow \widehat{\rm{H}}_{g,n}$ is given by the same union replacing $E_{(S_1,S_2)}$ with $\widehat{E}_{(S_1,S_2)}$ and $\delta_{[i]+S}$ with $\delta_{i+S}$.
\end{proposition} 
\begin{proof}
A generic hyperelliptic curve admits a unique $g_2^1$. By Lemma \ref{sec2:lemmaSvZ} and \cite[Thm. 3.7]{SvZ}, the map 
\[
\widehat{\phi}':\widehat{\mathcal{H}ur}_{g,2}^n\longrightarrow \overline{\mathcal{M}}_{g,2g+2+2n}
\]
is an isomorphism onto its image. The map is $S_{2g+2}$-equivariant and descends to an isomorphism onto the image $\phi':\overline{\mathcal{H}ur}_{g,2}^n\longrightarrow \overline{\mathcal{M}}_{g,[2g+2]+2n}$. Consider the induced morphism on coarse spaces
\begin{equation}
\label{eq:phi_tilde}
\begin{array}{rcl}
\phi':\overline{\rm{H}ur}_{g,2}^n&\to&\overline{\rm{M}}_{g,[2g+2]+2n}\\
\left[f:C\longrightarrow\pp^1, p_1,\ldots,p_n\right]&\mapsto&\left[C, x_1+\ldots+x_{2g+2}, p_1, \overline{p}_1, \ldots,p_n,\overline{p}_n\right],
\end{array}
\end{equation}
where as before, $x_1+\ldots+x_{2g+2}$ is the ramification divisor, $p_i$ are the ordered marked points and $\overline{p}_i$ the conjugates. The morphism $\phi'$ is an isomorphism onto the image and all contracted divisors of $\phi$ must come from the forgetful morphism
\[
\overline{\rm{M}}_{g,[2g+2]+2n}\longrightarrow\overline{\rm{M}}_{g,n}
\]
restricted to the image of $\phi'$. At this point we observe that the map $\phi$ is an isomorphism onto the image when restricted to the interior $\rm{H}ur_{g,2}^n$. A generic cover on the boundary has exactly one node in the target stable rational curve corresponding to either one or two nodes in the image stable pointed curve in $\overline{\rm{M}}_{g,[2g+2]+2n}$. Moreover, a pointed curve on the image of $\phi'$ can only be in a positive dimensional fibre of the restricted map
\[
\overline{\rm{M}}_{g,[2g+2]+2n}\supset {\rm{Im}}\left(\phi'\right)\longrightarrow\overline{\rm{H}}_{g,n}
\]
if at least one of the forgotten $2g+2$ unordered ramification points or $n$ conjugate points lies on a rational component. Recall that a rational component of an admissible double cover cannot contain more than two Weierstrass points. If none of the forgotten (unordered) Weierstrass points lie on a rational component, then we are in the situation described in Figure 1 and at least three of the forgotten $n$ conjugate markings lie on conjugate rational components. One observes that $\phi'\left(E_{(S_1,S_2)}\right)$ coincides with the image of the map
\[
\phi'\left(\overline{\rm{Hur}}_{g,2}^{(S_1\cup S_2)^c\cup\{p\}}\right)\times\overline{\rm{M}}_{0,S_1\cup S_2\cup\{q\}}\longrightarrow \overline{\rm{M}}_{g,[2g+2]+2n}
\]
gluing two copies of an $(S_1\cup S_2)\cup\{q\}$-pointed rational curve, one by identifying $q$ with $p$ and the other one by identifying $q$ with the conjugate $\overline{p}$. As expected,
\[
\dim E_{(S_1,S_2)}=\dim \phi'(E_{(S_1,S_2)})=\dim\left(\overline{\rm{Hur}}_{g,2}^{(S_1\cup S_2)^c\cup\{p\}}\times\overline{\rm{M}}_{0,S_1\cup S_2\cup\{q\}}\right)=2g-2+n.
\]
On the other hand, after forgetting the marked conjugate points, the image of $E_{(S_1,S_2)}$ via $\phi$ coincides with the image of the gluing map
\begin{equation}
\label{eq1:lemma_exc}
\phi'\left(\overline{\rm{Hur}}_{g,2}^{(S_1\cup S_2)^c\cup\{p\}}\right)\times\overline{\rm{M}}_{0,S_1\cup\{q_1\}}\times\overline{\rm{M}}_{0,S_2\cup\{q_2\}}\longrightarrow \overline{\rm{M}}_{g,n}
\end{equation}
defined by identifying $q_1$ with $p$ and $q_2$ with the conjugate $\overline{p}$. With the convention that if $\left|S_i\right|<2$, then $\overline{\mathcal{M}}_{0,S_i\cup\{q_i\}}=\{pt\}$. Computing dimensions
\[
{\rm{codim}}\;\phi\left(E_{(S_1,S_2)}\right)=\left\{\begin{array}{lcl}1&\hbox{if}&S_i=\varnothing\hbox{ or } \left|S_1\right|=\left|S_2\right|=1 ,\\
2&\hbox{if}& \left|S_i\right|=1\hbox{ and }\left|S_1\right|+\left|S_2\right| \geq 3,\\
3&\hbox{otherwise.}&\end{array}\right.
\]
Similarly, generic points of $\pi^{-1}(\delta_{[1]+S})$ are characterised by the admissible cover having exactly one Weierstrass point together with the markings and conjugate markings labeled by $S$ lying on a rational component. Let $\widehat{\phi}'\left(\widehat{\rm{H}ur}_{g,2}^n\right)\subset \overline{\rm{M}}_{g,2g+2+2n}$ be the finite cover of ${\rm{Im}}\left(\phi'\right)\subset \overline{\rm{M}}_{g,[2g+2]+2n}$ given by forcing an order on the Weierstrass points. One observes that the image of $\pi^{-1}\left(\delta_{[1]+S}\right)$ via $\phi'$ can be identified with the image of the finite map given by the composition
\[
\widehat{\phi}' \left(\widehat{\rm{H}ur}_{g,2}^{S^c}\times\widehat{\rm{H}ur}_{0,2}^{S}\right)=\widehat{\phi}' \left(\widehat{\rm{H}ur}_{g,2}^{S^c}\right)\times\widehat{\phi}' \left(\widehat{\rm{H}ur}_{0,2}^{S}\right)\longrightarrow\overline{\rm{M}}_{g,2g+2+2n}\longrightarrow\overline{\rm{M}}_{g,[2g+2]+2n}
\] 
given by identifying the first ordered Weierstrass point in the two covers followed by the $S_{2g+2}$-quotient. Similarly, after forgetting the conjugate points and then the Weierstrass points, the image of $\pi^{-1}\left(\delta_{[1]+S}\right)$ via $\phi$ can be identified with the image of the map
\[
\widehat{\phi}'\left(\widehat{\rm{H}ur}_{g,2}^{S^c}\right)\times \overline{\rm{M}}_{0,S\cup\{p,q\}}\longrightarrow\overline{\rm{M}}_{g,2g+2+n}\longrightarrow\overline{\rm{H}}_{g,n}\subset\overline{\rm{M}}_{g,n}
\] 
obtained by attaching an $S\cup\{p,q\}$-pointed rational component by identifying $p$ with the first ordered Weierstrass point followed by forgetting the $S^c$-conjugate points to obtain a pointed stable curve in $\overline{\rm{M}}_{g,2g+2+n}$ then forgetting the $2g+2$ ordered Weierstrass points to obtain a stable pointed curve in $\overline{\rm{H}}_{g,n}$. Computing dimensions, we obtain 
\[
{\rm{codim}}\;\phi\left(\pi^{-1}\left(\delta_{[1]+S}\right)\right)=\left\{\begin{array}{lcl}1&\hbox{if}&\left|S\right|=1,\\
2&\hbox{if}& \left|S\right|\geq 2.\end{array}\right.
\]
The case $i=2$ is analogous. The argument for the ordered case is the same.
\end{proof}

\begin{figure}[h]
\includegraphics[scale=0.15]{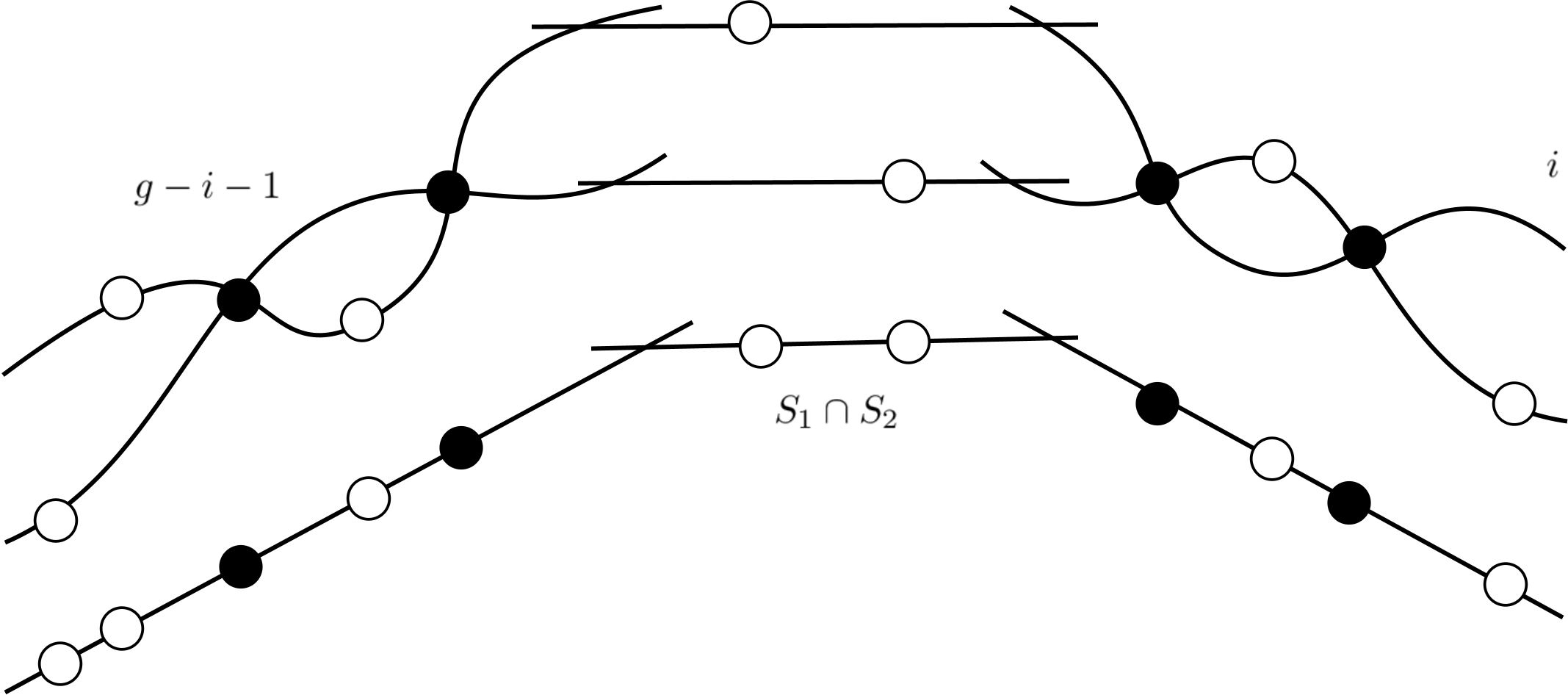}
\caption{Small contraction (filled points are unordered).}
\label{figure:contraction}
\end{figure}

\begin{proposition}
\label{prop:small_contraction}
Let $n\geq 2$. For any $0\leq i\leq \lfloor\frac{g}{2}\rfloor$ and $S_1,S_2\subset \{1,\ldots n\}$ with 
$$S_1\cup S_2=\{1,\ldots,n\}\quad\hbox{and}\quad |S_1\cap S_2|\geq2,$$ 
the pure codimension two locus in $\overline{\rm{H}ur}_{g,2}^n$ given by 
\[
\pi^{-1}\left(\delta_{[2i+2]+S_1}\cap \delta_{[2g-2i]+S_2}\right)\subset \overline{\rm{H}ur}_{g,2}^n
\]
has an irreducible component that is contracted via $\phi$ to a codimension three locus in $\overline{\rm{H}}_{g,n}$. Similarly, if $T\subset \{1,\ldots 2g+2\}$ is a set of markings corresponding to the Weierstrass points, and $|T|$ is even, then 
 the pure codimension two locus
\[
 \hat{\pi}^{-1}\left(\delta_{T+S_1}\cap \delta_{T^c+S_2}\right)\subset \widehat{\rm{H}ur}_{g,2}^n
 \]
has an irreducible component that is contracted via $\widehat{\phi}$ to a codimension three locus in $\widehat{\rm{H}}_{g,n}$.
\end{proposition}

\begin{proof}
For a general point on the intersection $\delta_{[2i+2]+S_1}\cap \delta_{[2g-2i]+S_2}$ there are two distinguished sets of preimages, one where all the markings over $S_1\cap S_2$ lie on one rational bridge and one where both rational bridges contain markings over $S_1\cap S_2$. The latter is shown in Figure \ref{figure:contraction}. Observe that the map $\phi$ restricted to the locus of covers described in Figure \ref{figure:contraction} where both rational bridges have marked points over $S_1\cap S_2$ has one-dimensional fibres, whereas the locus of covers where the markings over $S_1\cap S_2$ lie on the same rational bridge is not contracted. One observes that with the same notation as the proof of Proposition \ref{prop:exc_loc} the image of $\phi$ in $\overline{\rm{M}}_{g,[2g+2]+n}$ restricted to the locus described in Figure \ref{figure:contraction} coincides with the image of
\begin{equation}
\label{eq:contractionF2}
\phi'\left({\rm{H}ur}_{g-i-1,2}^{S_1'\cup \{p\}}\right)\times {\rm{M}}_{0,T\cup\{r,s\}}\times {\rm{M}}_{0,T'\cup\{t,u\}}\times\phi'\left({\rm{H}ur}_{i,2}^{S_2'\cup \{q\}}\right)\longrightarrow \overline{\rm{M}}_{g,[2g+2]+n},
\end{equation}
where $S_i'=S_i\setminus S_1\cap S_2$, $T\cup T'=S_1\cap S_2$ and the map is given by the identification $p=r$, $\overline{p}=t$, $q=s$, and $\overline{q}=u$. Computing dimensions we obtain the image of this locus has codimension three. To construct the $1$-parameter family of ordered admissible covers having the same image in $\overline{\rm{M}}_{g,2g+2+n}$ consider the map
\begin{equation*}
 {\overline{\rm{M}}}_{0,T\cup T'\cup\{r,s\}}\longrightarrow {\overline{\rm{M}}}_{0,T\cup\{r,s\}}\times  {\overline{\rm{M}}}_{0,T'\cup\{r,s\}}
\end{equation*}
obtained by forgetting the markings of $T$ and $T'$ to the components respectively. A general fibre of this map has dimension one. By gluing appropriately marked general fixed pointed rational curves at $r$ and $s$ we obtain a curve in $\delta_{2i+2+S_1}\cap \delta_{2g-2i+S_2}\subset\overline{\rm{M}}_{0,2g+2+n}$. Exactly one irreducible component of the pullback of this curve under the branch morphism lies in the locus of interest and gives a $1$-parameter family of admissible covers contracted by $\hat{\phi}$. Quotienting by the $S_{2g+2}$-action gives the required $1$-parameter family in the unordered case.
\end{proof}

As a consequence, we have the following.

\begin{proposition}
\label{prop:nonQfactorial}
When $n\geq 2$ and $g\geq 3$ the varieties $\overline{\rm{H}}_{g,n}$ and $\widehat{\rm{H}}_{g,n}$ are not $\mathbb{Q}$-factorial.
\end{proposition}

\begin{proof}
We have to rule out the case when the codimension $2$ locus singled out in Proposition \ref{prop:small_contraction} lies inside a codimension one component in the full exceptional locus of $\phi$. By Proposition \ref{prop:exc_loc}, when $1\leq i\leq g-2$, the divisors $\pi^*\delta_{[2i+2]+S_1}$ and $\pi^*\delta_{[2g-2i]+S_2}$ are not contracted via $\phi$. Moreover, they intersect transversally along an irreducible codimension $2$ subvariety. A general point on the intersection admits an analytic open neighborhood $U$ intersecting only $\delta_{[2i+2]+S_1}$ and $\delta_{[2g-2i]+S_2}$ and no other boundary divisor of $\overline{\mathcal{M}}_{0:[2g+2]+n}$. The map $\phi$ is an isomorphism on the interior $\rm{H}ur_{g,2}^n\subset\overline{\rm{H}ur}_{g,2}^n$, and $\pi^{-1}(\rm{M}_{0:[2g+2]+n})=\rm{H}ur_{g,2}^n$. In particular $\phi$ is an isomorphism when restricted to $\pi^{-1}\left(U\right)\setminus\pi^{-1}\left(\delta_{[2i+2]+S_1}\cap\delta_{[2g-2i]+S_2}\right)$. This shows that there exists an irreducible component of the full exceptional locus ${\rm{exc}}(\phi)$ of codimension $2$ contracted by $\phi$. The existence of a small contraction over $\widehat{\rm{H}}_{g,n}$ follows analogously.
\end{proof}

\begin{proposition}
The variety $\overline{\rm{H}}_{g,n}$ is normal. 
\end{proposition}

\begin{proof}
Let $U$ be a smooth scheme and $U\rightarrow \overline{\mathcal{M}}_{g,n}$ is an \'etale presentation giving $\overline{\mathcal{M}}_{g,n}$ the structure of smooth a Deligne--Mumford stack. Since the inclusion $\overline{\mathcal{H}}_{g,n}\longrightarrow \overline{\mathcal{M}}_{g,n}$ is representable, $U'=U\times_{\overline{\mathcal{M}}_{g,n}}\overline{\mathcal{H}}_{g,n}$ is a scheme and $U'\longrightarrow \overline{\mathcal{H}}_{g,n}$ is also an \'etale presentation. In particular $\overline{\mathcal{H}}_{g,n}$ is a DM stack, \'etale locally of the form $\left[{\rm{Spec}}(A)\big/G\right]$ with $G$ finite. 

We proceed by induction on the number of markings. For $n=0$, $\overline{\mathcal{H}}_g$ is a smooth proper Deligne-Mumford stack \cite[Ch. XI, Lemma 6.15]{ACG}, see also \cite{Co}. In particular Cohen-Macauly and normal. The coarse space $\overline{\rm{H}}_g$ has finite quotient singularities, thus $\overline{\rm{H}}_g$ is normal. Assume $\overline{\mathcal{H}}_{g,n-1}$ is Cohen--Macauly and normal. The forgetful morphism $p:\overline{\mathcal{H}}_{g,n}\to \overline{\mathcal{H}}_{g,n-1}$ is flat \cite[Lemma 101.25.3, ]{S-P} as it is the base change of a flat morphism $\overline{\mathcal{M}}_{g,n}\to \overline{\mathcal{M}}_{g, n-1}$. Moreover, every fibre of $p$ is a stable curve, thus Cohen-Macaulay. This means that the morphism $p$ is Cohen-Macaulay. By assumption $\pi:\overline{\mathcal{H}}_{g, n-1}\to {\rm{Spec}}(\mathbb{C})$ is Cohen-Macaulay and this property is preserved under compositions \cite[Lemma 37.20.4]{S-P}. We conclude that $\overline{\mathcal{H}}_{g,n}$ is Cohen-Macaulay. Smoothness in codimension one is proved in \cite[Prop. 5.4]{Sca}. This gives us normality of the stack $\overline{\mathcal{H}}_{g,n}$. Finally, \'etale locally the coarse map $\overline{\mathcal{H}}_{g,n}\longrightarrow \overline{\rm{H}}_{g,n}$ is of the form 
\[
\left[{\rm{Spec}}(A)\big/G\right]\longrightarrow {\rm{Spec}}(A^G)
\]
and normality is preserved after taking invariants. 
\end{proof}
\begin{remark}
We cannot conclude that $\overline{\rm{H}}_{g,n}$ is also Cohen--Macauly since locally, $A$ being Cohen--Macauly does not necessarily implies that $A^G$ also is. Even when $G$ is finite and acts linearly on a polynomial algebra $A$. See \cite{LP}.
\end{remark}

\section{Singularities of ordered Hurwitz spaces}
\label{section:sing_ordered}

Proposition \ref{prop:nonQfactorial} forces us to work with a better behaved birational model of $\overline{\rm{H}}_{g,n}$. This is provided by the coarse moduli space $\overline{\rm{H}ur}_{g,2}^n$. Let us recall some elementary notions on singularity of pairs, cf. \cite{Am, Ko, Ko2}. Let $X$ be a normal variety and $D=\sum d_iD_i$ a linear combination of Weil $\mathbb{R}$-divisors with $d_i\in \mathbb{R}$. Assume further that $K_X+D$ is $\mathbb{R}$-Cartier. Let $f:Y\to X$ be a birational morphism between normal varieties and consider the equality
$$K_Y=f^*\left(K_X+D\right)+\sum a(E,X,D)E$$
where the last sum runs over different prime divisors $E\subset Y$ with the convention that $E$ is either an exceptional divisor or of the form $E=f^{-1}D_i$ in which case $a(E,X,D)$ is declared to be $-d_i$. The real number $a(E,X,D)$ is called the {\textit{discrepancy}} of $E$ with respect to $X$ and $D$. This quantity does not depend on $Y$ in the sense that if $f':Y'\to X$ is a different birational model and $E'\subset Y'$ is the proper transform of $E$, then $a(E,X,D)=a(E',X,D)$. The {\textit{discrepancy of the pair}} $(X,D)$ is defined as \cite[Def. 2.9]{Ko2}, \cite[Def. 1.1]{Am},
$${\rm{discrep}}(X,D)={\rm{inf}}_{E}\left\{a(E,X,D)\mid E \hbox{ exceptional divisor over }X\right\}.$$ 
A pair $(X,D)$ is called \textit{canonical} or is said to have {\textit{canonical singularities}} if ${\rm{discrep}}(X,D)\geq 0$. An immediate remark is that $X$ having canonical singularities in the usual sense is equivalent to the pair $(X,0)$ being canonical. Another observation is that if $D'=\sum d_i'D_i$ with $d_i'\geq d_i$, then 
\begin{equation}
\label{eq:ineq_discrep}
{\rm{discrep}}(X,D')\leq{\rm{discrep}}(X,D).
\end{equation} 

Recall the comparison of discrepancies for finite morphisms:
\begin{proposition}[see Corollary 2.43 in \cite{Ko2}]
\label{prop:discrep_finite}
Let $f:X\to Y$ be a finite morphism between normal varieties of the same dimension, $D_Y$ a $\mathbb{Q}$-divisor on $Y$ and assume $K_Y+D_Y$ is $\mathbb{Q}$-Cartier. Define $D_X$ by the equality
$$K_X+D_X=f^*\left(K_Y+D_Y\right).$$
Then 
$${\rm{discrep}}\left(Y,D_Y\right)+1\leq {\rm{discrep}}\left(X,D_X\right)+1\leq \deg(f)\left({\rm{discrep}}\left(Y,D_Y\right)+1\right).$$
\end{proposition}

The map 
$$\alpha:\widehat{\rm{Hur}}_{g,2}^n\longrightarrow\overline{\rm{Hur}}_{g,2}^n$$
is a finite cover of degree $(2g+2)!$ simply ramified along the divisor $\pi^{c,*}\delta_{[2]+\varnothing}$ where two Weierstrass points lie on a rational bridge of the cover with no markings. There are branch morphisms $\pi^c, \hat{\pi}^c$ on coarse spaces that sit in the following commutative (not necessarily cartesian) diagram:
\begin{equation}
\begin{tikzcd}
\widehat{\rm{H}ur}_{g,2}^n\arrow[r, "\hat{\pi}^c"]\arrow[d, "\alpha"']&\overline{\rm{M}}_{0,2g+2+n}\arrow[d, "\sigma"] \\
\overline{\rm{H}ur}_{g,2}^n\arrow[r, "\pi^c"]&\overline{\rm{M}}_{0,[2g+2]+n}.
\end{tikzcd}
\end{equation}
The ramification divisors of $\hat{\pi}^c$ and $\pi^c$ can be computed using Proposition \ref{prop:ramif_brach1}. This will be considered in detail in the next section. In particular, we have the following identity (see Corollary \ref{prop:ramif_brach2} and Proposition \ref{prop:canonicalstack})
\begin{equation}
\label{eq:can_hur0}
K_{\widehat{\rm{H}ur}_{g,2}^n}=\hat{\pi}^{c,*}\left(K_{\overline{\rm{M}}_{0,2g+2+n}}+\frac{1}{2}D\right),
\end{equation}
where $D$ is a simple normal crossing divisor on $\overline{\rm{M}}_{0,2g+2+n}$ given by 
$$D=\sum_{p=0}^{\lfloor g/2\rfloor}\sum_{\substack{|T|=2p+1\\ S\neq\varnothing}} \delta_{T\cup S}.$$
Here $T$ runs over subsets of the first $2g+2$ markings and $S$ over subsets of the last $n$ markings. 

\begin{theorem}
\label{thm:cs_hat}
The coarse moduli space $\widehat{\rm{H}ur}_{g,2}^n$ has canonical singularities for all $g,n$ with $g\geq 2$ and $n\geq 0$.
\end{theorem}
\begin{proof}
Equation \eqref{eq:can_hur0} and Proposition \ref{prop:discrep_finite} imply that 
$${\rm{discrep}}\left(\overline{\rm{M}}_{0,2g+2+n},\frac{1}{2}D\right)\leq {\rm{discrep}}\left(\widehat{\rm{H}ur}_{g,2}^n, 0\right).$$
The result follows from the smoothness of $\overline{\rm{M}}_{0,2g+2+n}$ and the fact that $D$ is simple normal crossing, cf. \cite[Corollary 2.11]{Ko2}.
\end{proof}

The singularity analysis in the unordered case is more delicate and we treat it in the last section where we show in full Theorem \ref{thmsix}.



\section{The canonical class of pointed Hurwitz spaces}
\label{section:can_class}
If $\psi:\mathcal{X}\longrightarrow\mathcal{Y}$ is a finite map of Deligne--Mumford stacks, the ramification divisor ${\rm{Ram}}(\psi)\in {\rm{Pic}}_{\mathbb{Q}}\left(\mathcal{X}\right)$ is defined as the difference
\begin{equation}
\label{eq:ramif_explanation}
{\rm{Ram}}(\psi)=c_1\left(\Omega_{\mathcal{X}}\right)-c_1\left(\psi^*\Omega_{\mathcal{Y}}\right).
\end{equation}
In our situation, when $\mathcal{X}$ is a proper smooth Deligne--Mumford stack over $\mathbb{C}$ and $c:\mathcal{X}\longrightarrow X$ is a coarse moduli space, then for any point $x\in \mathcal{X}$, the finite stabilizer $G_x$ acts linearly on the tangent space $T_{x}\mathcal{X}$ (it actually acts on the full completed local ring $\widehat{O}_{\mathcal{X},x}$). Recall that the completed local ring pro-represents the corresponding deformation functor at the point $x$. We say that $x\in \mathcal{X}$ is a ramification point for $c$ if the action of $G_x$ on $T_{x}\mathcal{X}$ is non-trivial, or equivalently, if the quotient map $T_{x}\mathcal{X}\longrightarrow T_{x}\mathcal{X}\big/ G_x$ is ramified. Since $X$ has finite-quotient singularities, its singular locus has codimension at least two. The point $c(x)\in X$ is a smooth point if and only if $T_{x}\mathcal{X}\big/ G_x\cong T_{x}\mathcal{X}$. This only happens when elements of $G_x$ act as a {\textit{quasi-reflections}}, i.e., with eigenvalues $\eta, 1, 1,\ldots, 1$ with $\eta$ a root of unity. These are the elements fixing a hyperplane in $T_{x}\mathcal{X}$. Generically on the the locus where $G_x$ acts non-trivially and by quasi-reflections, one has that $G_x$ is cyclic, and the ramification order of the quotient $T_{x}\mathcal{X}\longrightarrow T_{x}\mathcal{X}\big.G_x$ at the origin is exactly $n={\rm{ord}}\left(G_x\right)$. We say that $c:\mathcal{X}\longrightarrow X$ is ramified along a reduced divisor $D\subset \mathcal{X}$ with order $n$ if for a general point $x$ on each component of $D$, the map on tangent spaces is as described above. One shows (see for instance \cite[Section 3]{GK}, and also \cite[Section 2]{HM}) that if $x\in D$ is such a point, with $c(x)\in X$ smooth, the map is locally defined by  
\[
\widehat{O}_{\mathcal{X},x}={\rm{Spec}}\left(\mathbb{C}\llbracket t_1,\ldots,t_d\rrbracket\right)\longrightarrow\widehat{O}_{X,x}={\rm{Spec}}\left(\mathbb{C}\llbracket s_1,\ldots,s_d\rrbracket\right)\
\]
where $d=\dim(\mathcal{X})$, $D$ is locally at $x$ defined by $(t_1=0)$ and the map is given by $s_1=t_1^n$, $s_2=t_2, \ldots,s_d=t_d$. If $ds_1\wedge\ldots\wedge ds_d$ is a local basis for ${\rm{det}}(\Omega_{X,x})$ and $dt_1\wedge\ldots\wedge dt_d$ a local basis for ${\rm{det}}\left(\Omega_{\mathcal{X},x}\right)$, then 
\[
c^*(ds_1\wedge\ldots\wedge ds_d)=n\cdot(t_1)^{n-1}(dt_1\wedge\ldots\wedge dt_d)
\]
vanishes with order $n-1$ at $D$. This is how we will compute \eqref{eq:ramif_explanation} in most of our cases. The ramification divisor for our coarse maps is ultimately the class $\sum_{i}(n_i-1)D
_i$ where $\cup D_i$ is the divisorial locus of objects with non-trivial automorphisms acting non-trivially on first order deformations, and $n_i$ is the order of such group on a general point.\\ 

The strategy to compute the canonical classes of the coarse spaces of interest is as follows. We compute the ramification of the map of stacks 
$\hat{\pi}:\widehat{{\mathcal{H}ur}}_{g,2}^n\longrightarrow \overline{\mathcal{M}}_{0,2g+2+n}$ and use the known formula \cite[Lemma 3.5]{KM} for $K_{\overline{\mathcal{M}}_{0,2g+2+n}}$ to compute the canonical class of $\widehat{{\mathcal{H}ur}}_{g,2}^n$. Second, we use the cartesian diagram on stacks:
\begin{equation}
\begin{tikzcd}
\widehat{\mathcal{H}ur}_{g,2}^n\arrow[r, "\hat{\pi}"]\arrow[d, "\hat{\sigma}"']&\overline{\mathcal{M}}_{0,2g+2+n}\arrow[d, "\sigma"] \\
\overline{\mathcal{H}ur}_{g,2}^n\arrow[r, "\pi"]&\overline{\mathcal{M}}_{0,[2g+2]+n},
\end{tikzcd}
\end{equation}
as well as the \'etale morphism (of stacks) $\overline{\mathcal{M}}_{2g+2+n}\longrightarrow\overline{\mathcal{M}}_{[2g+2]+n}$ to compute the ramification of $\pi$ and the canonical class of $\overline{\mathcal{H}ur}_{g,2}^n$. We then study the ramification of the coarse maps $\widehat{\mathcal{H}ur}_{g,2}^n\longrightarrow \widehat{\rm{H}ur}_{g,2}^n$ and $\overline{\mathcal{H}ur}_{g,2}^n\longrightarrow \overline{\rm{H}ur}_{g,2}^n$ to conclude the calculation of the canonical classes of the coarse spaces. Finally we use the commutativity (both in the ordered and unordered cases) of the diagram
\begin{equation}
\begin{tikzcd}
\overline{\mathcal{H}ur}_{g,2}^n\arrow[r, "\hat{\pi}"]\arrow[d]&\overline{\mathcal{M}}_{0,[2g+2]+n}\arrow[d] \\
\overline{\rm{H}ur}_{g,2}^n\arrow[r, "\pi^c"]&\overline{\rm{M}}_{0,[2g+2]+n},
\end{tikzcd}
\end{equation}
as well as the standard isomorphism on Chow groups between stack and coarse spaces to express the canonical class of $\overline{\rm{H}ur}_{g,2}^n$ (resp. ordered case) as a pull-back via $\pi^c$ of a class from $\overline{\rm{M}}_{0,[2g+2]+n}$. We use this to reduce the study positivity of the canonical class of $\overline{\rm{H}ur}_{g,2}^n$ to the positivity of classes in $\overline{\rm{M}}_{0,[2g+2]+n}$.

We start by recalling a well-known fact about the stack $\widehat{\mathcal{H}ur}_{g,2}^n$. 

\begin{proposition}[\cite{HM} or \cite{JKK, SvZ}]
\label{prop:ramif_brach1}
The stack $\widehat{\mathcal{H}ur}_{g,2}^n$ is smooth and the branch morphism
$$\hat{\pi}: \widehat{{\mathcal{H}ur}}_{g,2}^n\longrightarrow \overline{\mathcal{M}}_{0,2g+2+n}$$
is simply ramified over $\delta_{T\cup S}$ where $T\subset \{1,\ldots,2g+2\}$ has odd cardinality and $S$ is a subset of the remaining $n$ markings. Concretely, if we denote by $\sigma:\overline{\mathcal{M}}_{0,2g+2+n}\to\overline{\mathcal{M}}_{0,[2g+2]+n}$ the $S_{2g+2}$-quotient, then
$${\rm{Ram}}\left(\hat{\pi}\right)=\frac{1}{2}\hat{\pi}^*\sigma^*\left(\sum_{p=0}^{\lfloor g/2\rfloor}\sum_S \delta_{[2p+1]+S}\right).$$
\end{proposition}

\begin{proof}
Let $[f:C\to R, r_1,\ldots,r_{2g+2},p_1,\ldots,p_n]$ be a $\mathbb{C}$-point in $\widehat{\mathcal{H}ur}_{g,2}^n$. The completed local ring of $\overline{\mathcal{M}}_{0,2g+2+n}$ at $[R, f(r_1),\ldots,f(r_{2g+2}), f(p_1), \ldots, f(p_n)]$ is smooth and of the form 
$$\mathbb{C}\llbracket t_1,\ldots,t_{2g-1+n}\rrbracket.$$ We assume that $t_1,\ldots,t_\delta$ are smoothing parameters of the nodes $q_1,\ldots,q_\delta\in R$ and we write $f^{-1}(q_i)=\{q'_{i,1},\ldots,q'_{i,r_i}\}$ where $r_i=1$ or $2$ depending on whether $f$ is ramified at the node. Then, the completed local ring of $\widehat{\mathcal{H}ur}_{g,2}^n$ at $[f]$ is given by 
\begin{equation}
\label{eq:Hurhat_loc0}
\widehat{\mathcal{O}}_{\widehat{\mathcal{H}ur}_{g,2}^n, [f]}\cong \mathbb{C}\llbracket t_1,\ldots,t_{2g-1+n}, s_1,s_2, \ldots,s_\delta\rrbracket\big/\left(t_i=\left\{\begin{array}{cc} s_i^2&\hbox{ if }r_i=1\\ s_i&\hbox{ if }r_i=2\end{array} \right| 1\leq i\leq \delta\right).
\end{equation}
From this local description one concludes smoothness and observes that the map $\hat{\pi}:\widehat{\mathcal{H}ur}_{g,2}^n\to\overline{\mathcal{M}}_{0,2g+2+n}$ is simply ramified at the locus where the source curve $C$ of the admissible cover $f:C\to R$ has exactly one node with ramification index $2$, i.e., where exactly one $r_i=1$. 
\end{proof}
As a corollary, we have the analogous statement for the $S_{2g+2}$-quotient
\begin{corollary}\label{prop:ramif_brach2}
The ramification divisor of the branch morphism between moduli stacks in the unordered setting
$$\pi: \overline{{\mathcal{H}ur}}_{g,2}^n\longrightarrow \overline{\mathcal{M}}_{0,[2g+2]+n}$$
is given by 
$${\rm{Ram}}\left(\pi\right)=\frac{1}{2}\pi^*\left(\sum_{p=0}^{\lfloor g/2\rfloor}\sum_S \delta_{[2p+1]+S}\right).$$
\end{corollary}
\begin{proof}
The Cartesian diagram 
\begin{equation}
\begin{tikzcd}
\widehat{\mathcal{H}ur}_{g,2}^n\arrow[r, "\hat{\pi}"]\arrow[d, "\hat{\sigma}"']&\overline{\mathcal{M}}_{0,2g+2+n}\arrow[d, "\sigma"] \\
\overline{\mathcal{H}ur}_{g,2}^n\arrow[r, "\pi"]&\overline{\mathcal{M}}_{0,[2g+2]+n},
\end{tikzcd}
\end{equation}
holds on the level of stacks. Hence as $\hat{\sigma}$ and $\sigma$ are finite and \'{e}tale, the ramification divisor of $\hat{\pi}$ is given by $\hat{\sigma}^* {\rm{Ram}}(\pi)$.

\end{proof}


The canonical classes of the stacks $\widehat{\mathcal{H}ur}_{g,2}^n$ and $\overline{\mathcal{H}ur}_{g,2}^n$ can hence be computed via the branch morphisms and the canonical classes of $\overline{\mathcal{M}}_{0,2g+2+n}$ and $\overline{\mathcal{M}}_{0,[2g+2]+n}$. For convenience, we record this result as the following proposition.

\begin{proposition}[Canonical classes of moduli stacks]\label{prop:canonicalstack}
The canonical classes of the moduli stacks 
are given by:
\begin{eqnarray*}
\label{eq:K_M02g+2n}
K_{\overline{\mathcal{M}}_{0,[2g+2]+n}}&=&\sum_{j=0}^n\sum_{i=0}^{g+1}\left(\frac{(i+j)(2g+2+n-i-j)}{2g+n+1}-2\right)\delta_{[i]+j}.
\end{eqnarray*}
\begin{eqnarray*}
K_{ \widehat{{\mathcal{H}ur}}_{g,2}^n   }&=&\hat{\pi}^*\sigma^*\left( \sum_{j=0}^n\sum_{i=0}^{g+1}\left(\frac{(i+j)(2g+2+n-i-j)}{2g+n+1}-2\right)\delta_{[i]+j}  \right)\\
&&+\frac{1}{2}\hat{\pi}^*\sigma^*\left(\sum_{p=0}^{\lfloor g/2\rfloor}\sum_S \delta_{[2p+1]+S}\right)
\end{eqnarray*}
and
\begin{eqnarray*}
K_{ \overline{{\mathcal{H}ur}}_{g,2}^n   }&=&  {\pi}^*\left( \sum_{j=0}^n\sum_{i=0}^{g+1}\left(\frac{(i+j)(2g+2+n-i-j)}{2g+n+1}-2\right)\delta_{[i]+j}  \right)\\
&&+\frac{1}{2}\pi^*\left(\sum_{p=0}^{\lfloor g/2\rfloor}\sum_S \delta_{[2p+1]+S}\right).
\end{eqnarray*}
\end{proposition}

\begin{proof}
The formula for the canonical class of $\overline{\mathcal{M}}_{0,m}$ was computed in \cite[Lemma 3.5]{KM} and is given by
\begin{equation}
\label{eq:K_M0nb}
K_{\overline{\mathcal{M}}_{0,m}}=\sum_{s=2}^{\lfloor m/2\rfloor}\left(\frac{s(m-s)}{m-1}-2\right)\delta_{s}.
\end{equation}
Further, the morphism
$$\sigma:\overline{\mathcal{M}}_{0,2g+2+n}\longrightarrow \overline{\mathcal{M}}_{0,[2g+2]+n}$$
is finite and \'{e}tale. Hence by Riemann-Hurwitz 
$$K_{\overline{\mathcal{M}}_{0,2g+2+n}}=\sigma^*K_{\overline{\mathcal{M}}_{0,[2g+2]+n}}+\rm{Ram}(\sigma) $$
where $\rm{Ram}(\sigma)=0$, which gives the result for $K_{\overline{\mathcal{M}}_{0,[2g+2]+n}}$. 

Applying Riemann-Hurwitz to the morphisms $\hat{\pi}$ and $\pi$ we obtain the remaining canonical divisors similarly via the ramification specified in Proposition~\ref{prop:ramif_brach1} and Corollary~\ref{prop:ramif_brach2}.
\end{proof}

\begin{remark}\label{rem:stackram}
The map to the coarse space is simply ramified over the divisor $\delta_{[2]+\emptyset}$. This is because a general $[2g+2]+n$-pointed rational curve in $\delta_{[2]+\emptyset}$ has automorphism group of order two, i.e., the stabilizer group $G_x$ is cyclic of order two. This is the divisor in the stack and we have under the coarse map $\alpha: \overline{\mathcal{M}}_{0,[2g+2]+n}\longrightarrow \overline{\rm{M}}_{0,[2g+2]+n}$ that $\alpha^*\left(\frac{1}{2}\Delta_{[2]}\right)=\delta_{[2]}$, where $\Delta_{[2]}$ is the reduced locus whose general point corresponds to a curve having a rational tail with two unordered makings on it.
\end{remark}

Recall that we have the following commutative diagram 
\begin{equation}
\begin{tikzcd}
\widehat{\rm{H}ur}_{g,2}^n\arrow[r, "\hat{\pi}^c"]\arrow[d, "\hat{\sigma}^c"']&\overline{\rm{M}}_{0,2g+2+n}\arrow[d, "\sigma"] \\
\overline{\rm{H}ur}_{g,2}^n\arrow[r, "\pi^c"]&\overline{\rm{M}}_{0,[2g+2]+n}.
\end{tikzcd}
\end{equation}
on coarse spaces. For convenience, we now denote by
$$\varphi^c:=\sigma^c\circ\hat{\pi}^c=\hat{\sigma}^c\circ\pi^c: \widehat{\rm{H}ur}_{g,2}^n\longrightarrow \overline{\rm{M}}_{0,[2g+2]+n}  $$
the composition of maps on coarse spaces and obtain the following expressions for the canonical divisors of the coarse spaces. 

\begin{proposition}[Canonical classes of coarse spaces]\label{prop:canonicalcoarse}
The canonical classes of the coarse spaces 
are given by:
\begin{eqnarray*}
K_{\overline{\rm{M}}_{0,[2g+2]+n}}&=&-\delta_{[2]+0}+\sum_{j=0}^n\sum_{i=0}^{g+1}\left(\frac{(i+j)(2g+2+n-i-j)}{2g+n+1}-2\right)\delta_{[i]+j}.
\end{eqnarray*}
\begin{equation}
K_{\overline{\rm{Hur}}_{g,2}^n}=(\pi^c)^*(K_{\overline{\rm{M}}_{0,[2g+2]+n}}+\frac{1}{2}E)\hspace{0.5cm}\text{where}\hspace{0.5cm}E=\sum_{p=0}^{\lfloor g/2\rfloor}\sum_{j=1}^{n-1}\delta_{[2p+1]+j},
\end{equation}
and 
\begin{equation}
K_{\widehat{\rm{Hur}}_{g,2}^n}=(\varphi^c)^*(K_{\overline{\rm{M}}_{0,[2g+2]+n}}+\frac{1}{2}E') \hspace{0.5cm}\text{where}\hspace{0.5cm}  E'=2\delta_{[2]+0}+\sum_{p=0}^{\lfloor g/2\rfloor}\sum_{j=1}^{n-1}\delta_{[2p+1]+j}.
\end{equation}
\end{proposition}
\begin{proof}
As discussed in Remark~\ref{rem:stackram}, the map to the coarse space $\overline{\mathcal{M}}_{0,[2g+2]+n}\longrightarrow \overline{\rm{M}}_{0,[2g+2]+n}$ is simply ramified over the divisor $\delta_{[2]+\emptyset}$. Hence the first expression follows.

The coarse map $\overline{\mathcal{H}ur}_{g,2}^n\to\overline{\rm{Hur}}_{g,2}^n$ is simply ramified at
\begin{equation}
\label{eq:ramification1}
\pi^*\delta_{[2]+\emptyset}+\frac{1}{2}\pi^*\left(\sum_{p=1}^{g}\delta_{[2p+1]+\varnothing}\right),
\end{equation}
while the coarse map $\widehat{\mathcal{H}ur}_{g,2}^n\to\widehat{\rm{Hur}}_{g,2}^n$
is simply ramified at
\begin{equation}
\label{eq:ramification1}
\frac{1}{2}\varphi^*\left(\sum_{p=1}^{g}\delta_{[2p+1]+\varnothing}\right),
\end{equation}
(see \cite[Section 3]{GK} or \cite[Section 1.1]{FM}) and therefore we obtain the canonical divisor of this coarse space.
\end{proof}

\begin{theorem}
\label{thm:gen_type}
The divisor 
$$K_{\overline{\rm{M}}_{0,[2g+2]+n}}+\frac{1}{2}E$$
specified in Proposition~\ref{prop:canonicalcoarse} is effective for $n\geq 4g+6$ and big for $n\geq 4g+7$. More precisely, when $n\geq 4g+6$ it admits an effective decomposition of the form
$$K_{\overline{\rm{M}}_{0,[2g+2]+n}}+\frac{1}{2}E=\sum_{\substack{i,j\geq0\\ i+j\geq 2}}d_{[i]+j}\delta_{[i]+j},$$
where $d_{[i]+j}>0$ for all $i,j$ when $n\geq 4g+7$, and the coefficients $d_{[i]+j}$ are all positive except 
$$d_{[2]+0}=d_{[1]+1}=0$$
when $n=4g+6$.
 \end{theorem}

Assuming Theorem \ref{thmsix} we confirm the prediction of Schwarz (see Appendix \ref{Appendix}):

\begin{corollary}
The moduli space $\overline{\rm{H}}_{g,n}$ has non-negative Kodaira dimension when $n\geq 4g+6$ and is of general type when $n\geq 4g+7$.
\end{corollary}

\begin{proof}
The pullback of an effective (resp. big) divisor via a finite morphism is again effective (resp. big). The result follows from Proposition~\ref{prop:canonicalcoarse}, Theorem~\ref{thm:gen_type}, and Theorem \ref{thmsix}. 
\end{proof}

\begin{proof}[Proof of Theorem \ref{thm:gen_type}]
The Keel relations on ${\rm{Pic}}_{\mathbb{Q}}\left(\overline{\rm{M}}_{0,4}\right)$, cf. \cite{Ke}, yield an expression for the following divisor with trivial class
\begin{equation}\label{eqn:keel}
\sum_{j=2}^nc_{0,j}\delta_{[0]+j}+\sum_{j=1}^nc_{1,j}\delta_{[1]+j}+\sum_{j=0}^n{c_{2,j}}\delta_{[2]+j}+\sum_{i=3}^{g+1}\sum_{j=0}^{n}c_{i,j}\delta_{[i]+j} =0, 
\end{equation}
where
$$c_{i,j}=2i(2g+2-i)j(n-j)-i(i-1)(n-j)(n-j-1)-(2g+2-i)(2g+1-i)j(j-1). $$
More explicitly, consider the forgetful morphism $\mu:\Mbar{0}{2g+2+n}\longrightarrow \Mbar{0}{4}$ that forgets all but the marked points labelled in $\{1,2,2g+3,2g+4\}$. Then pulling back the trivial divisor $\delta_{\{1,2\}}-\delta_{\{1,3\}}$ in $\Mbar{0}{4}$ under $\mu$ and symmetrising via the action of $S_{2g+2}\times S_{n}$ we obtain an expression for the trivial divisor equal to the pullback of \eqref{eqn:keel} to $\overline{\rm{M}}_{0,2g+2+n}$ from $\overline{\rm{M}}_{0,[2g+2]+n}$.

This gives us an expression for a divisor with the same class as $\delta_{[2]+0}$ in terms of the other boundary divisors that we denote by $\mathfrak{D}$, namely
\begin{equation}
\label{eq:D}
\delta_{[2]+0}\equiv \mathfrak{D}=\frac{1}{2n(n-1)}\left(\sum_{j=2}^nc_{0,j}\delta_{[0]+j}+\sum_{j=1}^nc_{1,j}\delta_{[1]+j}+\sum_{j=1}^n{c_{2,j}}\delta_{[2]+j}+\sum_{i=3}^{g+1}\sum_{j=0}^{n}c_{i,j}\delta_{[i]+j}\right).
\end{equation}
Observe that the following expression hence trivially holds on the level of classes
$$ \frac{n-(4g+6)}{6g+3}\delta_{[2]+0}+\left(\frac{2(2g+n)}{2g+n+1}-3- \frac{n-(4g+6)}{6g+3}\right)\mathfrak{D} \equiv \left(\frac{2(2g+n)}{2g+n+1}-3  \right)\delta_{[2]+0}$$
and we obtain the following expression for the divisor of interest
$$K_{\overline{\rm{M}}_{0,[2g+2]+n}}+\frac{1}{2}E=\sum_{j=2}^nd_{[0]+j}\delta_{[0]+j}+\sum_{j=1}^nd_{[1]+j}\delta_{[1]+j}+\sum_{i=2}^{g+1}\sum_{j=0}^{n}d_{[i]+j}\delta_{[i]+j}$$
where
$$d_{[i]+j}=\begin{cases}
\frac{n-(4g+6)}{6g+3}&\text{for $i=2$, $j=0$}\\
\left(\frac{(i+j)(2g+2+n-i-j)}{2g+n+1}-\frac{3}{2}\right)+\left(\frac{2(2g+n)}{2g+n+1}-3- \frac{n-(4g+6)}{6g+3}\right)\frac{c_{i,j}}{2n(n-1)}&\text{for $i$ odd, $j\ne0,n,$}\\
\left(\frac{(i+j)(2g+2+n-i-j)}{2g+n+1}-2\right)+\left(\frac{2(2g+n)}{2g+n+1}-3- \frac{n-(4g+6)}{6g+3}\right)\frac{c_{i,j}}{2n(n-1)}&\text{otherwise}\\
\end{cases}   $$
It is this expression we will show that satisfies the positivity requirements of the proposition.

Observe that for $i\ne 2$, we have 
$$d_{[i]+j}=\begin{cases}
\frac{(F+G) }{6(2g+1)n(n-1)}+\frac{1}{2}&\text{for $i$ odd, $j\ne 0,n,$}\\
\frac{(F+G) }{6(2g+1)n(n-1)}&\text{otherwise.}\\
\end{cases}$$
where
\begin{eqnarray*}
F&=&n(n-1)(i^2n-in+2gi^2-3i^2+10gi+9i-24g-12) \\
G&=&2(2 g+1) j(  (2g^2+5g+3+g n-2n)j  +3 n^2 - i n^2 - 2 g i n - 2 i n - g n - n + 2 g i\\
&& + 3 i - 2 g^2 - 
 5 g - 3   ). \end{eqnarray*}
and we consider three subcases. 

 \textbf{For $i=0$:} Setting $n=4g+6+t$ we have 
 $$G=2(2g+1)j(((g-2)t+6g^2+3g-9)j+3t^2+23g t+35t +42g^2+129g+99)$$ 
 which is minimised for  $j\geq 2$ at $j=2$ and hence
 $$F+ G\geq  4 (2g+1) (2 g^2+5g+3 - 2 n + g n)  $$
 which is always positive and hence $d_{[0]+j}$ is always positive.
 
 \textbf{For $i=1$:}  Similarly, for $j\ne n$, setting $n=4g+6+t$ we have
 $$ G=2 (2 g+1) j (((g-2)t+6g^2-3g+3)j     +2t^2+13g t-3t+18g^2-15g)   $$
 which is minimised for  $1\leq j\leq n-1$ at $j=1$ and strictly increasing in $j$ and hence $d_{[1]+j}>d_{[1]+1}$ for $2\leq j\leq n-1$. But
 $$d_{[1]+1}= \frac{n-(4g+6)}{6n}   $$
 and hence $d_{[1]+1}=0$ for $n=4g+6$ and all other $d_{[1]+j}>0$ for $1\leq j\leq n-1$ and $n\geq 4g+6$. Further,
 $$d_{[1]+n}=\frac{1}{3}(2g^2+g n+3g-3)>0.$$

\textbf{For $3\leq i\leq 2g-1$:} Observe 
$$d_{[i]+j}\geq \frac{(F+G) }{6(2g+1)n(n-1)}$$
minimising the quadratic $G$ in $j$ which in these cases occurs at the vertex we obtain
$$G\geq G'= -\frac{(2 g+1)(in^2-3n^2+2gin+2in+gn+n-2gi-3i+2g^2+5g+3)^2}{2 (2g^2+5g+gn-2n+3)}. $$
Further, $F+G'$ becomes a rational function with strictly positive denominator and numerator that is quadratic in $i$ with negative coefficient of $i^2$ and maximum at $i=g+1$. Hence the minimum occurs at $i=3$ and setting $g=h+2$ and $n=4(h+2)+6+t$ we obtain the numerator of this rational function 
\begin{eqnarray*}
&\geq &485100 + 1362060 h + 1242963 h^2 + 512802 h^3 + 99540 h^4 + 
 7416 h^5 + 120204 t \\
 &&+ 314610 h t + 228714 h^2 t + 65064 h^3 t + 
 6456 h^4 t + 9639 t^2 + 25572 h t^2 + 13543 h^2 t^2 \\
 &&+ 2014 h^3 t^2 + 
 252 t^3 + 888 h t^3 + 264 h^2 t^3 + 12 h t^4
 \end{eqnarray*}
 and hence $d_{[i]+j}$ is clearly always positive for $h,t\geq0$.
 
\textbf{For $i=2$:}  These cases require a different approach. 
For $1\leq j\leq n$ the coefficient of $\delta_{[2]+j}$ is 
$$\left(\frac{(j+2)(2g+n-j)}{2g+n+1}-2\right)+\left(\frac{2(2g+n)}{2g+n+1}-3- \frac{n-(4g+6)}{6g+3}\right)\frac{c_{2,j}}{2n(n-1)}.$$
Letting $g=h+2$ and $n=4(h+2)+6+t$ we obtain 
$$d_{[2]+j}=\frac{H}{3(2g+1)n(n-1)} $$
where
\begin{eqnarray*}
H&=& t^3+(2h j+5j+14h+42)t^2\\
&&+(2h^2j^2+5hj^2+6h^2j+41hj+65j+64h^2+390h+587)t\\
&&+3 (2h+5) (2 h+7)(j(j-1)(h+1)+8h+26).
\end{eqnarray*}
Observe the denominator is always positive and the numerator is strictly increasing in $t\geq 0$ for fixed $j\geq1, h\geq 0$ and hence minimised at $t=0$ and giving the strict inequality 
$$d_{[2]+j}>0 $$
for $n\geq4g+6$ and $j\geq1$ as required.

Finally, recall the pseudo-effective cone of $\overline{\rm{M}}_{0,[n]}$  is generated by the boundary divisors $\delta_{[s]}$, cf. \cite[Thm. 1.3]{KM}. In particular, every divisor of the form $\sum d_{s}\delta_{[s]}$ with all coefficients $d_s>0$ lies in the interior of the pseudo-effective cone, i.e., is big. We have shown for $n\geq 4g+7$ our divisor can be expressed as a strictly positive sum of the boundary divisors in $\overline{\rm{M}}_{0,[2g+2]+n}$ and hence can be expressed as
$$\nu^*B+E$$
for the finite morphism $\nu:\overline{\rm{M}}_{0,[2g+2]+n}\longrightarrow \overline{\rm{M}}_{0,[2g+2+n]} $ with $B$ a big divisor and $E$ an effective divisor and is hence big.
\end{proof}

\begin{theorem}
\label{thm:gen_type2}
The divisor 
$$K_{\overline{\rm{M}}_{0,[2g+2]+n}}+\frac{1}{2}E'$$
specified in Proposition~\ref{prop:canonicalcoarse} is effective for $n=4$ and big for $n= 5$. More precisely, there exists an effective decomposition of the form 
$$K_{\overline{\rm{M}}_{0,[2g+2]+n}}+\frac{1}{2}E'=\sum d_{[t]+s}\delta_{[t]+s}$$
where $d_{[i]+j}>0$ for all $i,j$ when $n=5$, and the coefficients $d_{[i]+j}$ are all positive except 
$$d_{[2]+0}=d_{[1]+1}=d_{[2]+1}=0$$
when $n=4$.
\end{theorem}

As a corollary we have

\begin{corollary}
For all $g\geq2$ the moduli space $\widehat{\rm{H}}_{g,n}$ is of general type when $n\geq5$ and it has non-negative Kodaira dimension when $n\geq4$. 
\end{corollary}

\begin{proof}
The result for $0\leq n\leq 5$ follows from Proposition~\ref{prop:canonicalcoarse}, Theorems \ref{thm:cs_hat}, and \ref{thm:gen_type2}. For $n>5$, observe that there is a forgetful map $$\widehat{\mathcal{H}}_{g,n+1}\longrightarrow\widehat{\mathcal{H}}_{g,n}$$ whose general fibre is a smooth hyperelliptic curve of genus $g\geq2$. The result follows from sub-additivity of the Kodaira dimension for curve fibrations \cite{Kaw}.
\end{proof}

\begin{proof}[Proof of Theorem \ref{thm:gen_type2}]
As in the proof of Theorem~\ref{thm:gen_type} (see equation \eqref{eq:D}) utilising the identification $\delta_{[2]+0}\equiv \mathfrak{D}$ in ${\rm{Pic}}_{\mathbb{Q}}\left(\overline{\rm{M}}_{0,[2g+2]+n}\right)$ for $n=4$, and in the case $n=5$, the relation
$$\delta_{[2]+0}\equiv -\frac{g+3}{4(g+1)}\delta_{[2]+0} +\frac{5 g+7}{4(g+1)} \mathfrak{D},$$  
we obtain the following expression for the divisor of interest
$$K_{\overline{\rm{M}}_{0,[2g+2]+n}}+\frac{1}{2}E'=\sum_{j=2}^nd_{[0]+j}\delta_{[0]+j}+\sum_{j=1}^nd_{[1]+j}\delta_{[1]+j}+\sum_{i=2}^{g+1}\sum_{j=0}^{n}d_{[i]+j}\delta_{[i]+j}$$
where for $n=4$,
$$d_{[i]+j}=\begin{cases}
\frac{1}{6}(g-1)j(j-1)&\text{for $i=2$}\\
\frac{1}{6}((g-1) j^2-(g+3i-7)j+6i-9)&\text{for $i$ odd and $j\ne 0,n,$}\\
\frac{1}{6}((g-1) j^2-(g+3i-7)j+6i-12)&\text{otherwise.}\\
\end{cases}   $$
For $n=5$,
$$d_{[i]+j}=\begin{cases}
\frac{1}{4(g+1)} &\text{for $i=2$, $j=0$}\\
\left(\frac{(i+j)(2g+7-i-j)}{2g+6}-\frac{3}{2}\right)-\frac{(5g+7)}{160(g+3)(g+1)}c_{i,j}&\text{for $i$ odd and $j\ne 0,n,$}\\
\left(\frac{(i+j)(2g+7-i-j)}{2g+6}-2\right))-\frac{(5g+7)}{160(g+3)(g+1)}c_{i,j}&\text{otherwise.}\\
\end{cases}   $$
We show that these expressions satisfy the positivity requirements of the Theorem.

Consider the case $n=4$. Observe
$$d_{[2]+0}=d_{[1]+1}=d_{[2]+1}=0$$
and $d_{[2]+j}>0$ for $j\geq 2$. Further, by simply substituting $j=0,1,2,3,4$ in the remaining cases, we observe that the remaining coefficients are all positive for $g\geq 2$ and hence the proposition holds for $n=4$.

Consider the case $n=5$. We observe 
$$d_{[2]+j}=\frac{(10 g^2-g-11) j^2+(-10g^2+g+19) j+20}{160(g+1)}>0$$
for $j=1,2,3,4,5$ and 
$$d_{[1]+j}=\frac{(10 g^2-g-11) j^2+(-10g^2+41g+55) j-40-40g}{80(g+1)}  >0$$
for $j=1,2,3,4$ by simple substitution. Similarly
$$d_{[1]+5}=\frac{10g^2+5g-4}{4(g+1)}>0.$$
For $i$ odd, $3\leq i\leq g+1$ and $j\ne 0,5$ we have
\begin{equation}
\resizebox{.9\hsize}{!}{$d_{[i]+j}=\frac{10i^2+(80g-36j-40g j+70)i  -120 - 120 g  + 91 j + 81 g j - 10 g^2 j - 11 j^2 - g j^2 + 
 10 g^2 j^2}{80(g+1)}$}
 \end{equation}
the numerator is a positive definite quadratic in $i$ which attains a minimum at 
$$i=\frac{18j+20gj-40g-35}{10}.$$
However as this value falls outside our domain for all values $j=1,2,3,4$, testing the endpoints $i=3$ and $g+1$ for each $j=1,2,3,4$ suffices and proves positivity of $d_{[i]+j}$ in these cases.

Finally, for $i$ odd, $3\leq i\leq g+1$ and $j=0,5$ or $i$ even and $4\leq i\leq g+1$ and $0\leq j\leq 5$ we have
\begin{equation}
\resizebox{.9\hsize}{!}{$d_{[i]+j}=\frac{10i^2+(80g-36j-40g j+70)i  -160 - 160 g  + 91 j + 81 g j - 10 g^2 j - 11 j^2 - g j^2 + 
 10 g^2 j^2}{80(g+1)}$}
 \end{equation}
the numerator is again positive definite quadratic in $i$ which again attains a minimum at 
$$i=\frac{18j+20gj-40g-35}{10}.$$
For each value of $j$ this value again falls outside our domain for $i$ and hence testing these endpoints for each fixed $j=0,1,2,3,4,5$ suffices and proves positivity of $d_{[i]+j}$ in the remaining cases completing the proof.
Bigness for $n=5$ follows again from \cite[Thm. 1.3]{KM}, in particular the fact that our divisor of interest can be expressed as 
$$\nu^*B+\widetilde{E}$$
where $\widetilde{E}$ is effective, $\nu:\overline{\rm{M}}_{0,[2g+2]+n}\to\overline{\rm{M}}_{0,[2g+2+n]}$, and $B=\sum d_s\delta_{[s]}$ where all boundary coefficients $d_s$ are positive. This implies that $B$ is big and therefore so also is the divisor of interest.
\end{proof}


\section{Splitting curves}
\label{section:M}

Let $S$ be a subset of the last $n$ markings $S\subset \{2g+2+1,\ldots,2g+2+n\}$, $i$ a non-negative integer with $0\leq i\leq 2g$, and $I\subset\{1,\ldots,2g+2\}$ any subset of cardinality $i$ on the first $2g+2$ markings. Let
\begin{equation}
\label{eqn:bound.map}
\xi_{I,S}:\overline{\mathcal{M}}_{0,I^c\cup S^c\cup\{p\}}\times\overline{\mathcal{M}}_{0,I\cup S\cup\{q\}}\longrightarrow\overline{\mathcal{M}}_{0,2g+2+n}
\end{equation}
be the boundary map that glues two marked rational curves
$$\left[R_1, p_{I^c}, q_{S^c}, p\right]\in \overline{\mathcal{M}}_{0,I^c\cup S^c\cup\{p\}}\quad\hbox{and}\quad\left[R_2,p_I,q_S, q\right]\in\overline{\mathcal{M}}_{0,I\cup S\cup\{q\}}$$
by identifying $p$ and $q$. Here $p_I$ is short notation for the set of points $p_i$ with $i\in I$, and similarly for the rest.\\
 
Let $j\in I^c$ and 
\begin{equation}
\label{eqn:definitionBis}
\overline{\mathcal{M}}_{0,I^c\cup S^c\cup\{p\}}\longrightarrow \overline{\mathcal{M}}_{0,I^c\cup S^c\cup \{p\}-\{j\}}
\end{equation}
the corresponding forgetful morphism. We define the curve $B_{i, S}$ in $N_1\left(\overline{\mathcal{M}}_{0,\left[2g+2\right]+n}\right)$ as the pushforward of a fibre of \eqref{eqn:definitionBis}; 
$$F\times\{pt\}\subset \overline{\mathcal{M}}_{0,I^c\cup S^c\cup\{p\}}\times\overline{\mathcal{M}}_{0,I\cup S\cup\{q\}}$$
by the composition
$$\overline{\mathcal{M}}_{0,I^c\cup S^c\cup\{p\}}\times\overline{\mathcal{M}}_{0,I\cup S\cup\{q\}}\overset{\xi_{I,S}}{\longrightarrow}\overline{\mathcal{M}}_{0,2g+2+n}\longrightarrow\overline{\mathcal{M}}_{0,\left[2g+2\right]+n}.$$
Observe that since we are taking the $S_{2g+2}$-quotient, the definition of $B_{i, S}$ is independent of the set $I$ and $j\in I^c$, and only depends on the choice of $S$ and the cardinality $\left|I\right|=i$. Informally, the curve $B_{i, S}$ is obtained by gluing two fixed rational curves at fixed points, where one component has fixed points indexed by $S$ together with $i$-many fixed unordered points and the other component has fixed ordered points indexed by $S^c$ and $2g+2-i$ unordered points where one of them moves freely on the component and the rest are fixed. When $i+|S|\leq1$, the curve $B_{i, S}$ consists of a smooth rational curve with one of the unordered marked points moving freely while the remaining points are fixed.\\

\begin{figure}[h]
\includegraphics[scale=0.15]{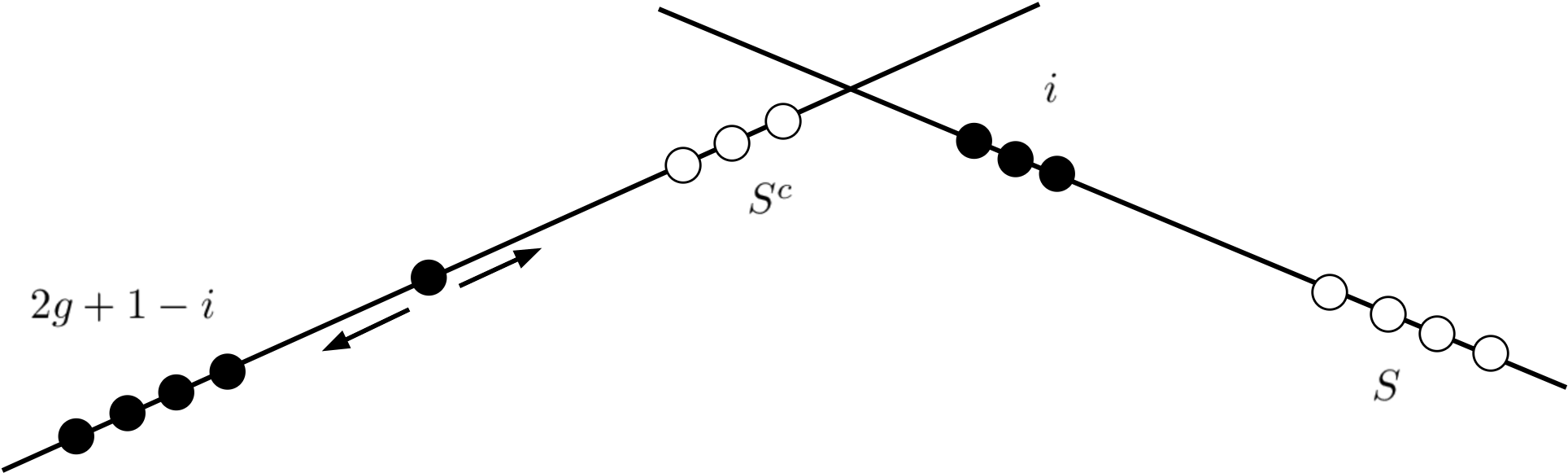}
\caption{Curve $B_{i,S}$ (filled points are unordered).}
\label{figure:BiS}
\end{figure}

\begin{proposition}
\label{prop:intersections_BiS}
Let $n$ be a non-negative integer, $0\leq i\leq 2g+1$, and $g\geq 2$. The curve $B_{i,S}\in N_1\left(\overline{\mathcal{M}}_{0,\left[2g+2\right]+n}\right)$ intersects the psi and boundary classes in ${\rm{Pic}}_\qq\left(\overline{\mathcal{M}}_{0,\left[2g+2\right]+n}\right)$ as follows:
\begin{equation*}
\begin{array}{rcl}
B_{0,\varnothing}\cdot \delta_{\left[2\right]+\varnothing}&=&2g+1,\\
B_{0,\varnothing}\cdot\delta_{\left[1\right]+\{j\}}&=&1,\\
B_{0,\varnothing}\cdot \psi_k&=&1,\\
B_{0,\varnothing}\cdot \psi_\star&=&4g+n,
\end{array}
\end{equation*}
and when $2\leq i+|S|\leq 2g+n-1$,
\begin{equation*}
\begin{array}{rcl}
B_{i,S}\cdot \delta_{\left[2\right]+\varnothing}&=&\begin{cases} 3&\hbox{if $i=2g-1$ and $S=\{1,\dots,n\}$,}   \\2g-2&\hbox{if $i=2$ and $S=\varnothing$,}   \\ 2g+1-i\;\;\;&\hbox{otherwise,} \end{cases}\\
B_{i,S}\cdot \delta_{\left[i\right]+S}&=&\begin{cases}2g-2&\hbox{if $i=2$ and $S=\varnothing$,}   \\
-1\;\;\;&\hbox{otherwise,} \end{cases}\\
B_{i,S}\cdot\delta_{\left[i+1\right]+S}&=&\begin{cases} 3&\hbox{if $i=2g-1$ and $S=\{1,\dots,n\}$,}   
\\1\;\;\;&\hbox{otherwise,} \end{cases}\\
B_{i,S}\cdot\delta_{\left[1\right]+\{j\}}&=&1\;\;\;\hbox{if}\;\;\;j\in S^c,\\
B_{i,S}\cdot \psi_k&=&1\;\;\;\hbox{if}\;\;\;k\in S^c,\\
B_{i,S}\cdot \psi_\star&=&2(2g-i)+1+\left|S^c\right|,
\end{array}
\end{equation*}
and all other intersections are trivial.
\end{proposition}

\begin{remark}
Note that our class  $\psi_\star$ is the unique class that pulls back under the $S_{2g+2}$-quotient morphism 
$$\Mbar{0}{2g+2+n}\longrightarrow \Mbar{0}{[2g+2]+n}$$
to give  $\psi_1+\ldots+\psi_{2g+2}$. This convention differs from some other authors.
\end{remark}

\begin{proof}
By the construction and the projection formula,
\begin{equation}
\label{eqn:int1_BiS}
B_{i,S}\cdot \psi_\star=F\cdot\theta_{I,S}^*\left(\psi_1+\ldots+\psi_{2g+2}\right),
\end{equation}
where $I$ is any subset of $\{1,\ldots,2g+2\}$ of cardinality at most $2g$, $\theta_{I,S}$ is the restriction of the boundary map \eqref{eqn:bound.map} to 
$$\overline{\mathcal{M}}_{0,I^c\cup S^c\cup\{p\}}\times \{pt\}\subset\overline{\mathcal{M}}_{0,I^c\cup S^c\cup\{p\}}\times\overline{\mathcal{M}}_{0,I\cup S\cup\{q\}},$$
and $F$ is the class of a fibre of the forgetful map \eqref{eqn:definitionBis} obtained by forgetting some marking $j\in I^c$. Recall that $F\cdot \psi_k=1$ if $k\neq j$. Similarly, by \cite[Lemma 4.38, Section 17]{ACG},
$$
\theta_{I,S}^*\psi_k=\left\{\begin{array}{ccc}\psi_k&\hbox{if}&k\in I^c,\\
0&\hbox{otherwise.}&\end{array}\right.
$$ 
Thus, for $k\in I^c$,
$$F\cdot \psi_k=1\;\;\hbox{for }k\neq j\quad\hbox{and}\quad F\cdot \psi_j=\left\{\begin{array}{ccc}2g-1+n&\hbox{for $i=0$ and $S=\varnothing$.}&\\-2+\left|I^c\right|+\left|S^c\right|&\hbox{for $2\leq i+|S|\leq 2g+n-1$.}
\end{array}\right.$$
This together with \eqref{eqn:int1_BiS} imply
\begin{equation}
\resizebox{1\hsize}{!}{$B_{i,S}\cdot \psi_\star=\left\{\begin{array}{cc}4g+n&\hbox{for $i=0$ and $S=\varnothing$.}\\2(2g-i)+1+\left|S^c\right|&\hbox{for $2\leq i+|S|\leq 2g+n-1$,}\end{array}\right.\quad\hbox{and}\quad B_{i,S}\cdot \psi_k=1\;\;\hbox{for }k\in S^c.$}
\end{equation}
Similarly, by the projection formula
\begin{equation}
\label{eqn:int2_BiS}
B_{i,S}\cdot \delta_{\left[i\right]+S}=F\cdot\theta_{I,S}^*\left(\sum_{\substack{J\subset\{1+\ldots+2g+2\}\\ \left|J\right|=i}}\delta_{J\cup S}\right)
\end{equation}
and \cite[Lemma 4.38, Section 17]{ACG} reads
$$\theta_{I,S}^*\delta_{0,J\cup S}=\left\{\begin{array}{ccc}-\psi_p&\hbox{if}&J=I\hbox{ or }J=I^c\hbox{ and }S=\varnothing\\0&\hbox{otherwise,}&\end{array}\right.$$
and $F\cdot\psi_p=1$. Recall that if $S=\varnothing$, then $\delta_{0,I}=\delta_{0,I^c}$ and we count it only once in the sum on the right hand side of \eqref{eqn:int2_BiS}. Thus for $2\leq i+|S|\leq 2g+n-1$ and $(i,S)\neq (2,\emptyset)$,
$$B_{i,S}\cdot \delta_{\left[i\right]+S}=-1.$$
The remaining intersections are clear from the construction.
\end{proof}

Proposition \ref{prop:intersections_BiS} implies 
$$B_{0,\varnothing}=B_{i,S}+B_{2g+1-i,S^c}$$
for any $i\leq 2g$ and $S$. Geometrically, for each $i$ and $S$, a smooth $\left(\left[2g+1\right]+n\right)$-pointed rational curve can be deformed to a curve on the boundary $\delta_{\left[i\right]+S}$, where the the free unordered point on the smooth curve, after degenerating will move freely on each component giving the two curves.\\

\section{The Kodaira dimension of $\overline{\rm{H}}_{g,n}$ and $\widehat{\rm{H}}_{g,n}$}
\label{section:KodDim}

In this section we assume Theorem \ref{thmsix} and complete the proofs of Theorems \ref{thm1} and \ref{thm2} by showing that 
\[
{\rm{Kod}}\left(\overline{\rm{H}}_{g,4g+6}\right)=4g+3, \;\;\; {\rm{Kod}}\left(\widehat{\rm{H}}_{g,4}\right)=1,
\]
and ${\rm{Kod}}\left(\widehat{\rm{H}}_{g,n}\right)=-\infty$ for $n=0,1,2,3$. The proof of Theorem \ref{thmsix} is presented in the last section. \\

We start with a well-known property of the Iitaka dimension. We add a proof for the sake of completeness.
\begin{proposition}[See Section 1.3 in \cite{Vie} or Chapters 2 and 3 in \cite{U}]
\label{prop:finite}
Let $f:X\to Y$ be a generically finite surjective morphism between two projective normal $\mathbb{Q}$-factorial varieties, $D_Y$ and $D_X$ $\mathbb{Q}$-effective divisors on $Y$ and $X$ respectively. Then
$$\kappa\left(X,f^*D_Y\right)=\kappa\left(Y,D_Y\right)$$
and 
$$\kappa\left(X,D_X\right)\leq \kappa\left(Y, f_*D_X\right).$$ 
\end{proposition}

\begin{proof}
Let $m>0$, then by projection formula 
$$H^0\left(X, \mathcal{O}_X(mf^*D_Y)\right)=H^0\left(Y, f_*\mathcal{O}_X(mf^*D_Y)\right)=H^0\left(Y, \mathcal{O}_Y(mD_Y)\otimes f_*\mathcal{O}_X\right),$$
and there is an embedding $\mathcal{O}_Y\hookrightarrow f_*\mathcal{O}_X$. This gives us the inequality $\kappa\left(Y,D_Y\right)\leq \kappa\left(X,f^*D_Y\right)$. On the other hand $f_*$ preserves linear equivalence. Moreover, the map
$$
\begin{array}{rcl}
\left|mf^*D_Y\right|&\longrightarrow&\left|mf_*f^*D_Y\right|\\
E&\mapsto&f_*E
\end{array}
$$
is well defined and cannot have positive dimensional fibres. In particular
$$\dim\left|mf^*D_Y\right|\leq \dim\left|\deg(f)\cdot mD_Y\right|$$
giving us 
$$\kappa\left(X, f^*D_Y\right)\leq \kappa\left(Y, f_*f^*D_Y\right)=\kappa(Y, \deg(f)D_Y)=\kappa\left(Y, D_Y\right).$$
The last inequality follows from the fact that the map between projective spaces
$$\begin{array}{rcl}
\left|mD_X\right|&\longrightarrow&\left|mf_*D_X\right|\\
E&\mapsto&f_*E
\end{array}$$
is well defined.
\end{proof}
 
In light of Proposition \ref{prop:finite} we can reduce our problem to a calculation on $\overline{\rm{M}}_{0,[2g+2]+n}$. By Proposition~\ref{prop:canonicalcoarse} and theorems~\ref{thm:gen_type} and~\ref{thm:gen_type2} we are in a very special situation where for $n\geq 4g+6$ (resp. $n\geq 4$) the canonical divisor of $\overline{\rm{H}ur}_{g,2}^n$ (resp. $\widehat{\rm{H}ur}_{g,2}^n$) is the pullback via a finite morphism of an effective divisor 
$$D=\sum_{i,S}d_{[i]+S}\delta_{[i]+S}\in \mathbb{Q}_{\geq 0}\langle \delta_{[i]+S}\mid i,S\rangle\subset\overline{\rm{Eff}}^1\left(\overline{\rm{M}}_{0,[2g+2]+n}\right).$$
Proposition \ref{prop:finite} together with theorems \ref{thm:cs_hat} and \ref{thmsix} imply that the Kodaira dimension of $\overline{\rm{Hur}}_{g,2}^n$ is given by $\kappa\left(\overline{\rm{M}}_{0,[2g+2]+n},D\right)$, where 
\begin{equation}
\label{eq:Dunordered}
D=K_{\overline{\rm{M}}_{0,[2g+2]+n}}+\frac{1}{2}\sum_{p=0}^{\lfloor g/2\rfloor}\sum_{j=1}^{n-1}\delta_{[2p+1]+j}.
\end{equation}
While the Kodaira dimension of $\widehat{\rm{Hur}}_{g,2}^n$ is given by $\kappa\left(\overline{\rm{M}}_{0,[2g+2]+n},D'\right)$, where 
\begin{equation}
\label{eq:Dordered}
D'=K_{\overline{\rm{M}}_{0,[2g+2]+n}}+\delta_{[2]+0}+\frac{1}{2}\sum_{p=0}^{\lfloor g/2\rfloor}\sum_{j=1}^{n-1}\delta_{[2p+1]+j}.
\end{equation}
\\

As explained in the introduction we separate the proof into two inequalities. The proof in the ordered and the unordered setting are analogous.\\ 

For simplicity, assume that we are in the unordered case. First we show that when $n=4g+6$ there exists a big class $B$ in $\overline{\rm{M}}_{0,n}$ such that $D$ equals $p^*B$ plus an effective class, where $p:\overline{\rm{M}}_{0,[2g+2]+n}\to \overline{\rm{M}}_{0,n}$ is the morphism that forgets the unordered markings. Since $p$ is surjective and has connected fibres, Zariski's Main Theorem gives us 
$$p_*\mathcal{O}_{\overline{\rm{M}}_{0,[2g+2]+n}}=\mathcal{O}_{\overline{\rm{M}}_{0,n}}$$
and therefore
$$H^0\left(\overline{\rm{M}}_{0,[2g+2]+n}, \mathcal{O}(p^*B)\right)=H^0\left(\overline{\rm{M}}_{0,n}, \mathcal{O}(B)\right).$$
This will give us
\begin{equation}
\label{sec6:eq_low}
n-3=\kappa\left(\overline{\rm{M}}_{0,[2g+2]+n}, p^*B\right)\leq\kappa\left(\overline{\rm{M}}_{0,[2g+2]+n}, D\right)={\rm{Kod}}\left(\overline{\rm{H}}_{g,n}\right).
\end{equation}

For the second inequality we use the abundance of splittings of the curve $B_{0,\varnothing}$ into covering curves for the divisors $\delta_{[i]+S}$ to identify rigid components in $|mD|$. Removing such rigid components provides a large number of conditions on the coefficients of the remaining effective divisor, in fact, enough to ensure that the class plus an effective divisor is a pullback of a big class on $\overline{\rm{M}}_{0,n}$, giving us 
\begin{equation}
\label{sec6:eq_up}
{\rm{Kod}}\left(\overline{\rm{H}}_{g,n}\right)\leq\kappa\left(\overline{\rm{M}}_{0,[2g+2]+n}, D\right)\leq n-3
\end{equation}
when $n=4g+6$.\\

We start with inequality \eqref{sec6:eq_low}. The only property of the effective decomposition of $K_{\overline{\rm{H}ur}_{g,2}^{4g+6}}$ that we will use is the following lemma on the positivity of the coefficients.

\begin{lemma}
\label{sec6:lemma_lowerbound}
Let $n\geq 4$ and $D$ be an effective divisor in $\overline{\rm{M}}_{0,[2g+2]+n}$ that admits an effective decomposition as follows
\begin{equation}
\label{sec6:eq_lemma:lowerbound}
D=\sum_{\substack{0\leq i\leq 2g+2\\ 0\leq S\leq\frac{n}{2}}} d_{[i]+S}\delta_{[i]+S},
\end{equation}
where $d_{[2]+S},d_{[1]+S}\geq 0$ for all $S\subset\{1,\ldots,n\}$ with $|S|=0$ or $1$, and all other coefficients are strictly positive. Then 
$$n-3\leq \kappa\left(\overline{\rm{M}}_{0,[2g+2]+n},D\right).$$
\end{lemma}

\begin{proof}
Let $m$ be the smallest positive coefficient in the decomposition \eqref{sec6:eq_lemma:lowerbound} 
$$m=\min\left\{d_{[i]+S}\left|(i,|S|)\neq (2,0), (1,1), (2,1)\right.\right\}.$$ 
One observes that 
$$m\cdot\left(\sum_{\substack{0\leq i\leq 2g+2\\ 2\leq\left|S\right|\leq \frac{n}{2}}}\delta_{[i]+S}\right)\leq D$$
and further
$$\sum_{\substack{0\leq i\leq 2g+2\\ 2\leq\left|S\right|\leq \frac{n}{2}}}\delta_{[i]+S}=p^*\delta,$$
where $p=\overline{\rm{M}}_{0:[2g+2]+n}\longrightarrow \overline{\rm{M}}_{0,n}$ is the standard forgetful morphism and 
\begin{equation}
\label{sec6:eq_boundaryM0n}\delta=\sum_{2\leq\left|S\right|\leq \frac{n}{2}}\delta_{S}
\end{equation}
is the full boundary divisor of $\overline{\rm{M}}_{0,n}$. Note that $\delta_{S}=\delta_{S^c}$ is counted only once in the sum \eqref{sec6:eq_boundaryM0n}. The morphism $p$ is surjective is it has connected fibres. In particular for any $r>0$ 
\[
H^0\left(\overline{\rm{M}}_{0,[2g+2]+n}, \mathcal{O}(rp^*\delta)\right)=H^0(\overline{\rm{M}}_{0,n}, \mathcal{O}(r\delta))
\]
giving us the following inequalities:
\[
\kappa\left(\overline{\rm{M}}_{0,n},\delta\right)=\kappa\left(\overline{\rm{M}}_{0,[2g+2]+n},p^*\delta\right)\leq\kappa\left(\overline{\rm{M}}_{0,[2g+2]+n},D\right).
\]
Finally we claim that $\delta$ is big. Indeed, if $q:\overline{\rm{M}}_{0:n}\to\overline{\rm{M}}_{0,[n]}$ is the $S_n$ quotient, then 
\[
\delta=q^*\left(\sum_{2\leq s\leq \frac{n}{2}}\delta_{[s]}\right).
\]
The pseudo-effective cone of $\overline{\rm{M}}_{0,[n]}$  is generated by the boundary divisors $\delta_{[s]}$, cf. \cite[Thm. 1.3]{KM}. In particular, every divisor of the form $\sum d_{s}\delta_{[s]}$ with all coefficients positive $d_s>0$ lies in the interior of the pseudo-effective cone, i.e., is big. This implies that $\delta$ is big as well giving us the desired lower bound.
\end{proof}

As a consequence we have:

\begin{corollary}
\label{sec6:cor_lowerbound}
The Kodaira dimension of $\overline{\rm{H}}_{g,4g+6}$ is bounded from below by
\[
4g+3\leq{\rm{Kod}}\left(\overline{\rm{H}}_{g,4g+6}\right)
\]
and the Kodaira dimension of $\widehat{\rm{H}}_{g,4}$ is bounded from below by
\[
1\leq{\rm{Kod}}\left(\widehat{\rm{H}}_{g,4}\right).
\]
\end{corollary}
\begin{proof}
This follows immediately from Proposition~\ref{prop:canonicalcoarse}, Theorems \ref{thm:gen_type}, \ref{thm:gen_type2} and Proposition \ref{prop:finite}.
\end{proof}

Now we prove the upper bound \eqref{sec6:eq_up}. Recall that a {\textit{covering curve for an irreducible divisor}} is a curve class such that irreducible curves with this class cover a Zariski open subset of the divisor. Let $D$ be an effective Cartier divisor on a normal projective variety $X$. An {\textit{effective decomposition}} of $D$ is a sum
$$D=\sum_{i}c_iD_i,$$
where the summands $D_i$ are again effective Cartier and the coefficients are non-negative $c_i\geq 0$. A well-known trick \cite[Lemma 4.1]{CC} for showing an irreducible divisor is rigid is to show the divisor has negative intersection with a covering curve class for the divisor. However, the utility of covering curves also lies in their ability to detect base loci which we present as the following lemma. For simplicity, we assume $\mathbb{Q}$-factoriality.

\begin{lemma}
\label{lem:baselocus}
Let $X$ be a normal $\mathbb{Q}$-factorial projective variety and let $\beta\in N_1(X)$ be a curve class such that irreducible curves with class equal to $\beta$ cover a Zariski open subset of an irreducible effective divisor $E$ with $\beta\cdot E<0$. Then for any effective divisor $D$ such that $\beta\cdot D<0$, 
\begin{equation}
\label{sec6:lemma_baselocus_eq1}
|kD|=\left|kD-k\left(\frac{\beta\cdot D}{\beta\cdot E}\right)E\right|+k\left(\frac{\beta\cdot D}{\beta\cdot E}\right)E
\end{equation}
for some sufficiently large $k>0$.
\end{lemma}

\begin{remark}
\label{rmk:baselocus}
In other words, $D-(\frac{\beta\cdot D}{\beta\cdot E})E$ is $\qq$-effective and
\begin{equation}
\label{sec6:eqIitaka}
\kappa\left(X,D\right)=\kappa\left(X,D-(\frac{\beta\cdot D}{\beta\cdot E})E\right).
\end{equation}

\end{remark}

\begin{proof}[Proof of Lemma~\ref{lem:baselocus}]
As irreducible curves with class equal to $\beta$ cover a Zariski dense subset of $E$, any effective divisor $kD$ with $\beta\cdot kD<0$ must contain an unmovable component supported on $E$.


\end{proof}

%
%

The key idea to obtain the upper bound \eqref{sec6:eq_up} is as follows. Assume a nef curve $M$ and effective divisor $D$ have intersection $M\cdot D=0$ and $M$ splits as
\[
M=\beta_1+\beta_2.
\]
Then either the intersections $\beta_1\cdot D$ and $\beta_2\cdot D$ are of opposite parities, or both are zero. If the curves $\beta_i$ are in fact covering curves for divisors $D_i$, then the first situation by Lemma~\ref{lem:baselocus} would identify a rigid component in $D$. More precisely if $\beta_1\cdot D>0$, then 
$$\kappa\left(X,D\right)=\kappa\left(X,D-aD_2\right)$$
where $a=\frac{\beta_2\cdot D}{\beta_2\cdot D_2}$. Since $M$ is nef, again $M\cdot(D-aD_2)=0$ and now both $\beta_1$ and $\beta_2$ intersect $(D-aD_2)$ trivially. In conclusion, for proposes of computing the Iitaka dimension of $D$ we can always assume $\beta_1\cdot D=\beta_2\cdot D=0$.\\

\begin{proposition}\label{prop:M0n}
Let $n\geq4$ and $D$ an effective divisor in $\overline{\rm{M}}_{0,[2g+2]+n}$ such that $B_{0,\varnothing}\cdot D=0$ then
\[
\kappa(\overline{\rm{M}}_{0,[2g+2]+n},D)\leq n-3.
\]
\end{proposition}

\begin{proof}
Proposition \ref{prop:intersections_BiS} implies 
\[
B_{0,\varnothing}=B_{i,S}+B_{2g+1-i,S^c}
\]
for all choices of $i$ and $S$. Further, as $B_{i,S}$ and $B_{2g+1-i,S^c}$ form covering curves for the divisors $\delta_{[i]+S}$ and $\delta_{[2g+1-i]+S^c}$ respectively, the divisor $D$ having negative intersection with either would identify a rigid component. As pointed out in the discussion above, there hence exists an effective divisor $G$ with
\[
\kappa(\overline{\rm{M}}_{0,[2g+2]+n},D)=\kappa(\overline{\rm{M}}_{0,[2g+2]+n},G)
\]
and $B_{i,S}\cdot G=0$ for all choices of $i$ and $S$. We show that this implies $G\leq p^*B$ for a big divisor $B$ on $\Mbar{0}{n}$.\\ 

Symmetrising the Kapranov basis \cite{Kap} we obtain the basis for ${\rm{Pic}}_\qq\left(\overline{\rm{M}}_{0,[2g+2]+n}\right)$ centred at the first unsymmetrised point. The basis consists on $H=\psi_1$ and the boundary divisors $\delta_{[i]+S\cup\{1\}}$ where $1\leq i+|S|\leq 2g+2+n-4$ and $0\leq i\leq 2g+2$. However, in this basis for $n\geq3$ 
\[
\delta_{[0]+\{2,3\}}=H-\sum_{2,3\not\in S}\delta_{[i]+S\cup\{1\}}.
\]
Hence we provide a non-standard basis, that will serve our needs by replacing $H$ by $\delta_{[0]+\{2,3\}}$ when $n\geq3$. Consider the unique expression for the class of $G$ in this basis
\begin{equation}
\label{sec6:eqG}
G=\sum c_{[i]+T}\delta_{[i]+T},
\end{equation}
where the sum runs over the elements of our basis.\\

At this point we observe that $B_{0,\varnothing}\cdot G$ implies $c_{[1]+\{1\}}=0$ and $B_{2g-1,\{1,\dots,n\}}\cdot G=0$ implies $c_{[2g-1]+\{1,\dots,n\}}=0$. While $B_{2g-i,\{1,\dots,n\}}\cdot G=0$ for $i=2,\dots,2g$ implies 
$$c_{[2g-i+1]+\{1,\dots,n\}}-c_{[2g-i]+\{1,\dots,n\}}=0$$
and hence $c_{[j]+\{1,\dots,n\}}=0$ for all $j$.\\

Similarly, $B_{2g+2-i,\{2,\dots,n\}}\cdot G=0$ for $i=2,\dots,2g+2$ gives 
$$c_{[i-1]+\{1\}}-c_{[i]+\{1\}}=0$$
and hence as $c_{[1]+\{1\}}=0$ we have $c_{[j]+\{1\}}=0$ for all $j$. \\

Finally, $B_{2g+2-i,\{1,\dots,n\}\setminus\{j\}}\cdot G=0$ for $i=3,\dots,2g+2$ and $j=2,\dots,n$ gives 
$$c_{[2g+3-i]+\{1,\dots,n\}\setminus\{j\}}-c_{[2g+2-i]+\{1,\dots,n\}\setminus\{j\}}=0$$
with the base case of $i=2$ giving $c_{[2g]+\{1,\dots,n\}\setminus\{j\}}=0$ and hence $c_{[j]+\{1,\dots,n\}\setminus\{j\}}=0$ for all $j$.\\

The only non-zero coefficients in \eqref{sec6:eqG} are those corresponding to boundary divisors $\delta_{[i]+T}$ where $2\leq \left|T\right|\leq n-2$, that is, divisors that arise as irreducible components of the pullback of boundary divisors on $\overline{\rm{M}}_{0,n}$ under the forgetful morphism 
$$p:\overline{\rm{M}}_{0,[2g+2]+n}\longrightarrow \overline{\rm{M}}_{0,n}.$$
However, any effective divisor on $\overline{\rm{M}}_{0,n}$ that can be expressed as a positive sum of all boundary components is big (see the proof of Lemma~\ref{sec6:lemma_lowerbound}). Hence we can choose a big divisor $B$ on $\overline{\rm{M}}_{0,n}$ such that $G\leq p^*B$. As the morphism $p$ is surjective and has connected fibres, we have $p_*\mathcal{O}=\mathcal{O}$ and 
\[
\kappa(\overline{\rm{M}}_{0,[2g+2]+n},D)\leq \kappa(\overline{\rm{M}}_{0,[2g+2]+n},p^*B)=n-3.
\]
\end{proof}

As a consequence, we have an upper bound on the Kodaira dimension.

\begin{corollary}
\label{sec6:coro_bound2}
The Kodaira dimension of $\overline{\rm{H}}_{g,4g+6}$ is bounded by 
\[
{\rm{Kod}}\left(\overline{\rm{H}}_{g,4g+6}\right)\leq4g+3
\]
and the Kodaira dimension of $\widehat{\rm{H}}_{g,4}$ is bounded by
\[
{\rm{Kod}}\left(\widehat{\rm{H}}_{g,4}\right)\leq1.
\]
\end{corollary}

\begin{proof}
Observe that Proposition~\ref{prop:canonicalcoarse} gives
\[
K_{\overline{\rm{H}ur}_{g,2}^{4g+6}}=(\pi^{c})^*D\;\;\;\hbox{and}\;\;\;K_{\widehat{\rm{H}ur}_{g,2}^{4}}=\left(\varphi^c\right)^*D'\]
where $D, D'$ are the divisors in $\overline{\rm{M}}_{0,[2g+2]+n}$ given in Equations~\eqref{eq:Dunordered} and \eqref{eq:Dordered}. Theorems~\ref{thm:gen_type} and \ref{thm:gen_type2} show these divisors are effective for $n=4g+6$ and $n=4$ respectively. Further, Proposition \ref{prop:intersections_BiS} gives
\[
B_{0,\varnothing}\cdot D=\frac{n-(4g-6)}{2}\;\;\;\hbox{and}\;\;\; B_{0,\varnothing}\cdot D'=\frac{n-4}{2}
\]
and hence
\[
B_{0,\varnothing}\cdot D=0\;\;\;\hbox{and}\;\;\; B_{0,\varnothing}\cdot D'=0.
\]
for $n=4g+6$ and $n=4$ respectively. By Proposition \ref{prop:finite}
\[
{\rm{Kod}}\left(\overline{\rm{H}ur}_{g,2}^{4g+6}\right)=\kappa\left(\overline{\rm{M}}_{0,[2g+2]+4g+6}, D\right)
\]
and 
\[
{\rm{Kod}}\left(\widehat{\rm{H}ur}_{g,2}^{4}\right)=\kappa\left(\overline{\rm{M}}_{0,[2g+2]+4}, D'\right).
\]
Hence Proposition \ref{prop:M0n} provides the result.
\end{proof}

Combining Theorem \ref{thm:gen_type} and Corollaries~\ref{sec6:cor_lowerbound} and \ref{sec6:coro_bound2} we obtain

\thmone*

Finally, observe that the curve $B_{0,\varnothing}$ gives more information on the positivity of the canonical classes of interest. For the divisors $D$ and $D'$ in $\overline{\mathcal{M}}_{0,[2g+2]+n}$ we have
$$B_{0,\varnothing}\cdot D=\frac{n-(4g+6)}{2}     \hspace{0.3cm}\text{and}\hspace{0.3cm}B_{0,\varnothing}\cdot D'=\frac{n-4}{2}.$$
But as $B_{0,\varnothing}$ is nef (it forms a moving curve), this implies that $D$ and $D'$ are not pseudo-effective for $n\leq 4g+5$ and $n\leq3$ respectively.

Said differently, as $(\pi^c)^* B_{0,\varnothing}$ and $(\varphi^c)^* B_{0,\varnothing}$ are nef (they form a moving curves) and intersect $K_{\overline{\rm{H}ur}_{g,2}^n}$ and $K_{\widehat{\rm{H}ur}_{g,2}^n}$ negatively for $n\leq 4g+5$ and $n\leq3$ respectively, the canonical divisors are not pseudo-effective in this range. In the unordered case, this provides a new proof of the previously known uniruledness results~\cite{Be, AB}. In the ordered case, we record this as the following proposition. 
\begin{proposition}
The moduli spaces $\widehat{\rm{H}}_{g,n}$ are uniruled for $n\leq 3$.
\end{proposition}
This concludes the proof of our Theorem:
\thmtwo*


\section{The effective cone of $\overline{\rm{H}}_{g,n}$}
\label{section:eff}

In this section we completely classify the structure of the pseudo-effective cone of $\overline{\rm{H}}_{g, n}$. Recall that for a normal projective variety $X$ we denote by ${\rm{N}}^1(X)$ the $\mathbb{R}$-vector space consisting on Cartier divisors with real coefficients up to numerical equivalence. The dimension (or rank) of ${\rm{N}}^1(X)$ is the \textit{Picard number} of $X$. The \textit{effective cone}, denoted ${\rm{Eff}}^1(X)$, is the convex cone inside ${\rm{N}}^1(X)$ generated by classes of effective divisors. This cone is in general not closed and the closure $\overline{\rm{Eff}}^1(X)$ is called the \textit{pseudo-effective cone}. The $\mathbb{R}$-vector space consisting on one-dimensional cycles with real coefficients up to numerical equivalence is denoted ${\rm{N}}_1(X)$ an it is dual to ${\rm{N}}^1(X)$ via the intersection pairing
$${\rm{N}}^1(X)\times{\rm{N}}_1(X)\longrightarrow \mathbb{R}.$$
The dual of $\overline{\rm{Eff}}^1(X)$ via the intersection paring is the cone of of \textit{nef curves} ${\rm{Nef}}_1(X)$ generated by effective irreducible curves having non-negative intersection with all pseudo-effective divisors. We use this duality to show the non-polyhedrality of $\overline{\rm{Eff}}^1(\overline{\rm{H}}_{g,n})$ for $n\geq2$.\\

We closely follow~\cite{Mu} and identify a nef curve that forms an extremal ray, at which this cone is non-polyhedral. The class group $CH^1(\overline{\mathcal{H}}_{g,n})_\mathbb{Q}=CH^1(\overline{\rm{H}}_{g,n})_\mathbb{Q}$ is described in \cite{Sca}. We are only interested in divisors on $\overline{\mathcal{H}}_{g,n}$ that are $\mathbb{Q}$-Cartier. Let $i:\overline{\mathcal{H}}_{g,n}\to \overline{\mathcal{M}}_{g,n}$ be the natural inclusion. Then there is an inclusion
\[
i^*\left(\frac{{\rm{Pic}}_{\mathbb{Q}}\left(\overline{\mathcal{M}}_{g,n}\right)}{\langle\lambda\rangle}\right)\subset{\rm{Pic}}_{\mathbb{Q}}\left(\overline{\mathcal{H}}_{g,n}\right).
\]
Moreover the Picard group on the interior ${\rm{Pic}}_{\mathbb{Q}}\left({\mathcal{H}}_{g,n}\right)$ is generated by the restriction of $\psi_1,\ldots,\psi_n$. This follows from \cite[Thm. 1.1]{Sca} and the fact that restriction of Cartier divisors are again Cartier. We use the same notation for divisors on $\overline{\mathcal{H}}_{g,n}$ that are the restriction of divisors on $\overline{\mathcal{M}}_{g,n}$.

\begin{lemma}
For $g\geq 2$, the cone $\overline{\rm{Eff}}^1(\overline{\rm{H}}_g)$ is polyhedral and generated by the irreducible components of the boundary.
\end{lemma}

\begin{proof}
The cone $\overline{\rm{Eff}}^1(\overline{\rm{M}}_{0,[2g+2]})$ is polyhedral and generated by the boundary divisors, cf. \cite[Thm. 1.3]{KM}. Hence the isomorphism $\alpha:\overline{\rm{H}}_g\overset{\sim}{\to}\overline{\rm{M}}_{0,[2g+2]}$ implies the result.
\end{proof}

\begin{lemma}
For $g\geq 2$, the cone $\overline{\rm{Eff}}^1(\overline{\rm{H}}_{g,1})$ is polyhedral generated by the irreducible components of the boundary and the Weierstrass divisor.
\end{lemma}

\begin{proof}
Consider the map 
$$\beta: \overline{\rm{M}}_{0,[2g+2]+1}\longrightarrow \overline{\rm{H}}_{g,1}.$$
The cone $\overline{\rm{Eff}}^1(\overline{\rm{M}}_{0,[2g+2]+1})$ is polyhedral and generated by the boundary divisors~\cite[Prop. 5.3]{Ru}. Moreover,
$$\beta_*\delta_{[1]+1}=W$$ 
where $a>0$ and $W$ is the closure of the locus of $1$-pointed hyperelliptic curves where the marking is a Weierstrass point. Therefore we have 
$$\overline{\rm{Eff}}^1(\overline{\rm{H}}_{g,1})=\beta_*(\overline{\rm{Eff}}(\overline{\rm{M}}_{0,[2g+2]+1}))$$
is generated by the irreducible components of the boundary and the Weierstrass divisor.
\end{proof}

\begin{remark}
The class of $W$ in $\overline{\rm{H}}_{g,1}$ was recently computed in \cite{EH}.
\end{remark}

The case of $n\geq2$ is more interesting. We follow closely~\cite{Mu} which proves non-polyhedrality of $\overline{\rm{Eff}}^1\left(\overline{\rm{M}}_{g,n}\right)$ for $g\geq2$ and $n\geq 2$. Denote by $[F]$ the curve class of the general fibre of the forgetful morphism 
$$\overline{\rm{H}}_{g,n}\longrightarrow\overline{\rm{H}}_{g,n-1}.$$
Irreducible curves with this class sweep out an open Zariski dense subset of $\overline{\rm{H}}_{g,n}$ and hence $[F]\cdot[D]\geq0$ for any pseudo-effective divisor $[D]$, that is, $[F]$ is a nef curve.\\

Following the strategy of~\cite{Mu} we show that $[F]$ is in fact an extremal ray in the cone $\Nefbar_1(\Hbars{g}{n})$ of nef (or moving) curve classes which is dual to $\overline{\rm{Eff}}^1(\overline{\rm{H}}_{g,n})$. However, we show the corank of pseudo-effective divisors $[D]$ with $[F]\cdot[D]=0$ is at least two in ${\rm{N}}^1(\overline{\rm{H}}_{g,n})$ and hence this curve forms an extremal but non-polyhedral edge of ${\rm{Nef}}_1(\overline{\rm{H}}_{g,n})$. \\

For any nef curve class $[B]$ we define the pseudo-effective dual space
$$[B]^\vee:=\{[D]\in \overline{\text{Eff}}^1(\overline{\rm{M}}_{g,n})\hspace{0.3cm}\big|\hspace{0.3cm} [B]\cdot[D]=0\}.$$
Our first Proposition bounds the rank of this dual space for the curve $[F]$.

\begin{proposition}\label{prop:corank}
For $g\geq 2$ and $n\geq2$, 
$$\rho(\overline{\rm{H}}_{g,n})-n\leq{\rm{rank}}([F]^{\vee}\otimes\mathbb{R})\leq \rho(\overline{\rm{H}}_{g,n})-2,$$
where $\rho(\overline{\rm{H}}_{g,n})$ denotes the Picard number of $\overline{\rm{H}}_{g,n}$.
\end{proposition}

\begin{proof}
Consider an irreducible effective divisor $D$ in $\overline{\rm{H}}_{g,n}$ that intersects the interior $\rm{H}_{g,n}$ with $[F]\cdot [D]=0$. For any point $[C,p_1,\dots,p_{n}]\in\overline{\rm{H}}_{g,n}$ contained in $D$, denote by $B$ the curve that is the fibre of $[C,p_1,\dots,p_{n-1}]$ under the forgetful morphism $\pi:\overline{\rm{H}}_{g,n}\longrightarrow \overline{\rm{H}}_{g,n-1}$ that forgets the $n$th point. As $[B]=[F]$, by assumption we have $[B]\cdot[D]=0$ so necessarily as $B$ is irreducible and intersects $D$ set theoretically, $B$ is contained in the support of $D$. Hence the divisor $D$ is supported on the pullback of an irreducible effective divisor under $\pi$. Further, if $D$ is an irreducible divisor in $\overline{\rm{H}}_{g,n}$ that does not intersect the interior $\rm{H}_{g,n}$ with $[F]\cdot [D]=0$ then $D$ is an irreducible component of the boundary outside of $\delta_{0:\{i,n\}}$ for $i=1,\dots,n-1$. This shows 
$${\rm{rank}}([F]^{\vee}\otimes\mathbb{R})\geq \rho(\overline{\rm{H}}_{g,n})-n.$$

A key observation from~\cite{Mu} is that the intersection in $\overline{\rm{H}}_{g,2}$ of the surface equal to the general fibre of the forgetful morphism to $\overline{\mathcal{H}}_g$ with any effective divisor offers a potential way to obtain a covering curve for the divisor. More formally, let $C$ be a general smooth hyperelliptic genus $g$ curve and
$$i:C\times C\longrightarrow \overline{\rm{H}}_{g,2}$$
the natural morphism. For a fixed divisor $[D]$ in $\overline{\rm{H}}_{g,2}$, define
$$[B_D]:=i_*i^*[D].$$

Further, \cite[Prop. 1.6]{Mu} holds verbatim when replacing $\overline{\rm{M}}_{g,2}$ by $\overline{\rm{H}}_{g,2}$ and we have that if
$$[B_D]\cdot [D]=(2g-2)\left((4g-4)c_{\psi_1}c_{\psi_2}+(c_{\psi_1}+c_{\psi_2})^2-c_{0:\{1,2\}}^2  \right)<0$$
then $[D]$ contains a rigid strictly effective component.\\

Our argument for the lower bound of the pseudo-effective corank of $[F]$ in the case of $n=2$ implies that $\text{rank}([F]^{\vee}\otimes\mathbb{R})= \rho(\overline{\rm{H}}_{g,2})-1$ if and only if there exists a strictly pseudo-effective class in $\overline{\rm{Eff}}^1(\overline{\rm{H}}_{g,2})$ with zero intersection with $[F]$. That is, this would imply the existence of a $[D]\in\overline{\rm{Eff}}^1(\overline{\rm{H}}_{g,2})\setminus \text{Eff}^1(\overline{\rm{H}}_{g,2})$ with class
$$[D]=((1-2g)a-b)\psi_1+a\psi_2+b\delta_{0:\{1,2\}}+\text{other boundary components}$$
for $a,b\in\mathbb{R}$ with $a\ne 0$. However, for such a class we obtain
$$[B_D]\cdot [D]=-8a^2(g-1)^2g<0$$
implying the existence of a strictly effective divisor class $[E]$ with non-zero $\psi_2$-coefficient and $[F]\cdot[E]=0$, hence providing a contradiction.\\

The last remaining task is to extend this upper bound to the case $n\geq3$. Though the fibre of the forgetful morphism $\overline{\rm{H}}_{g,n}\longrightarrow\overline{\rm{H}}_{g,n-1}$ over a point on the interior ${\rm{H}}_{g,n-1}$ is irreducible, for $n\geq3$ the fibre over a general point in $\delta_{0:\{1,\dots,n-1\}}$ is reducible with two irreducible components. As a result, $[F]$ can be expressed as an effective sum of these two curves which form covering curves for $\delta_{0:\{1,\dots,n-1\}}$ and $\delta_{0:\{1,\dots,n\}}$. Any extremal $[D]\in \overline{\rm{Eff}}^1(\overline{\rm{H}}_{g,n})$ outside of $\delta_{0:\{1,\dots,n-1\}}$ or $\delta_{0:\{1,\dots,n\}}$ must have nonnegative intersection with each of these covering curves and if further $[F]\cdot[D]=0$ then each of these intersections must in fact be zero. \\

Let $\varphi:\overline{\rm{H}}_{g,2}\longrightarrow \overline{\rm{H}}_{g,n}$ be the morphism gluing in a general $n$-pointed rational curve at the first marked point. For such a $[D]$ we obtain $[F]\cdot\pi^*[D]=0$. The proof of Lemma 3.3 from~\cite{Mu} for $\overline{\rm{M}}_{g,n}$ holds verbatim for $\overline{\rm{H}}_{g,n}$ and shows in this situation $\pi^*[D]$ is pseudo-effective and hence by pulling back the class of $[D]$ we have $c_{\psi_n}=0$. This provides a second linearly independent condition on $[F]^\vee\otimes\mathbb{R}$ for $n\geq 3$ completing the proof of the upperbound.
\end{proof}

The remaining task is to show that $[F]$ is indeed extremal in ${\rm{Nef}}_1(\overline{\rm{H}}_{g,n})$. We do this through the use of well-chosen families of effective Cartier divisor classes such that for any potential nef decomposition of the class $[F]$ we can provide an effective Cartier divisor class with a negative intersection with a summand.\\

Define $D_k$ and $E_k$ as the following loci in $\overline{\rm{H}}_{g,2}$ defined via conditions in the Jacobian of hyperelliptic curves
\small
\begin{equation*}
\resizebox{.95\hsize}{!}{$D_k:={\left\{ [C,p_1,p_2]\in\rm{H}_{g,2} \big| \exists [C,p_1,p_2,q_1,\dots,q_{g-1}]\in\rm{H}_{g,g+1}, \mathcal{O}_C((k(g-1)+1)p_1-p_2-k\sum_{i=1}^{g-1}q_i)\sim\mathcal{O}_C   \right\}}$}
\end{equation*}
\begin{equation*}
\resizebox{.95\hsize}{!}{$E_k:={\left\{ [C,p_1,p_2]\in\rm{H}_{g,2} \big| \exists [C,p_1,p_2,q_1,\dots,q_{g-1}]\in\rm{H}_{g,g+1}, \mathcal{O}_C((k(g-1)-1)p_1+p_2-k\sum_{i=1}^{g-1}q_i)\sim\mathcal{O}_C   \right\}}.$}
\end{equation*}
\normalsize
Taking the closure we obtain two families of divisors $\overline{D}_k$ and $\overline{E}_k$ for $k\geq 2$ in $\overline{\rm{H}}_{g,2}$. The following lemma proves these are indeed effective and Cartier.

\begin{lemma}\label{lem:eff}
For any $k\geq 2$, $\overline{E}_k$ and $\overline{D}_k$ are effective and Cartier.
\end{lemma}

\begin{proof}
By \cite{Mu} both are effective in $\overline{\rm{M}}_{g,2}$. Since $\overline{\rm{M}}_{g,2}$ is $\mathbb{Q}$-factorial and the divisors are defined by restriction to $\overline{\rm{H}}_{g,2}$ it is enough to show that there is a single 2-pointed hyperelliptic curve $[C,p_1,p_2]\in \overline{\rm{H}}_{g,2}$ not contained in $\overline{D}_k$, respectively $\overline{E}_k$. Let $[C,q]$ be a general pointed hyperelliptic curve of genus $g\geq 1$ where $q$ is a Weierstrass point if $g\geq2$. We will prove the following more general statement by induction on $g$. For any integers $n_1,n_2$, not both zero, there exist points $p_1,p_2$ in $C$ such that for any set of points $q_1,\ldots,q_{r}$ with $r\leq g-1$ and $q_i\neq p_j$ when $n_j\neq0$
$$\mathcal{O}_{C}\left(n_1p_1+n_2p_2-k\sum_{i=1}^r q_i-(n_1+n_2-kr)q\right)\neq \mathcal{O}_C.$$ 
The lemma then will follow choosing $n_1=k(g-1)+1$ and $n_2=-1$ for $\overline{D}_k$ and $n_1=k(g-1)-1$ and $n_2=1$ for $\overline{E}_k$.\\

When $g=1$ the condition becomes a non-trivial torsion condition on three points on an elliptic curve and hence the statement clearly holds. Now let $C$ be the hyperelliptic curve of genus $g$ obtained by gluing $C_1$ and $C_2$ along a Weierstrass point $q$ and $n_1,n_2$ to integers not both zero. We assume $C_1$ has genus $1$ and therefore $C_2$ has genus $g-1$. Take $p_1, p_1'$ in $C_1$ satisfying the induction hypothesis for integers $(n_1,0)$ and $p_2,p_2'\in C_2$ satisfying the induction hypothesis for integers $(n_2,0)$. Further, observe that as there is no condition on $p_1'$ and $p_2'$ this will hold for these points chosen freely in the curves $C_1$ and $C_2$ respectively. Then for any set of $r\leq g-1$ points, in $C$, say $q_1,\ldots,q_{r_1}\in C_1$ and $q_1',\ldots,q_{r_2}'\in C_2$ with $r=r_1+r_2$.
The equation in the Jacobian of $C$ then becomes
\begin{equation}\label{sec7:eqEff}
\resizebox{.95\hsize}{!}{$\mathcal{O}_{C_1}\left(n_1p_1-k\sum_{i=1}^{r_1}q_i-(n_1-kr_1)q\right)= \mathcal{O}_{C_1}\quad\hbox{and}\quad\mathcal{O}_{C_2}\left(n_2p_2-k\sum_{i=1}^{r_2} q_i'-(n_2-k{r_2})q\right)= \mathcal{O}_{C_2}.$}
\end{equation}
But either $r_1<1$ or $r_2<g-1$. Thus, by induction hypothesis there are no points 
$$q_1,\ldots,q_{r_1},q_1',\ldots,q_{r_2}'$$
with $q_i\neq p_1$ if $n_1\neq 0$ and $q_j'\neq p_2$ if $n_2\neq 0$ satisfying \eqref{sec7:eqEff}. 
\end{proof}

The classes of the corresponding divisors in $\overline{\rm{M}}_{g,2}$ was computed in~\cite[Corollary 4.2]{Mu}. The statement in the proof of Lemma~\ref{lem:eff} further shows that $\delta_{0:\{1,2\}}$ is not an effective component of the pullback of these divisors from $\overline{\rm{M}}_{g,2}$. That is, under the inclusion
$$i:\overline{\rm{H}}_{g,2}\longrightarrow \overline{\rm{M}}_{g,2}$$
 we obtain the classes of these divisors in $\Pic_\mathbb{Q}(\overline{\rm{H}}_{g,2})$ as \small
$$  [ \overline{D}_k]=\frac{1}{2}(gk+1)(gk-k+1)k^{2g-2}\psi_1+\frac{1}{2}(1-k)k^{2g-2}\psi_2 -\frac{1}{2}(gk^{2g}-k^{2g}+2)g\delta_{0:\{1,2\}}        +\dots$$ \normalsize
and \small
$$   [\overline{E}_k]=\frac{1}{2}(gk-1)(gk-k-1)k^{2g-2}\psi_1+\frac{1}{2}(k+1)k^{2g-2}\psi_2-\frac{1}{2}(gk^{2g}-k^{2g}+2)g\delta_{0:\{1,2\}}+\dots    $$ \normalsize

\begin{proposition}\label{prop:extremal}
The curve class $[F]$ is extremal in $\overline{\text{Nef}}_1(\overline{\rm{H}}_{g,2})$ for $g\geq2$ and $n\geq1$.
\end{proposition}

\begin{proof}
Consider the case of $n=2$. If $[F]$ is not extremal it is the sum of nef curve classes with zero intersection with $[F]^\vee$. Hence a non-trivial decomposition of $[F]$ into a sum of nef curves implies the existence of a nef curve class $[F^t]$ for a fixed $t\in\mathbb{R}\setminus\{0\}$ with intersections
$$[F^t]\cdot\psi_1=1,\hspace{1cm}[F^t]\cdot \psi_2=(2g-1)+t,\hspace{1cm}[F^t]\cdot\delta_{0:\{1,2\}}=1$$
and zero intersection with the other generators of $\Pic_\mathbb{Q}(\overline{\rm{H}}_{g,2})$.  But for any fixed $t>0$ 
$$ [F^t]\cdot [\overline{D}_k]= (k^{2g - 2}-1) g +t\frac{1}{2}(1-k)k^{2g-2} =\frac{-t}{2}k^{2g-1}+\mathcal{O}(k^{2g-2})<0 \text{  for  }k\gg0,$$
and for any fixed $t<0$ 
$$ [F^t]\cdot [\overline{E}_k]=(k^{2g - 2}-1) g   +t \frac{1}{2}k^{2g-2}(k+1)=\frac{t}{2}k^{2g-1}+\mathcal{O}(k^{2g-2}) <0 \text{  for  }k\gg0.$$
Hence $[F^t]\in \Nefbar_1(\overline{\rm{H}}_{g,2})$ if and only if $t=0$. For $n\geq 3$ we consider the the forgetful morphisms
$$\pi_j:\overline{\rm{H}}_{g,n}\longrightarrow\overline{\rm{H}}_{g,2}$$
for $j=1,\dots,n-1$ that forget all but the $j$th and $n$th points. As ${\pi_j}_*[F]=[F]$ is an extremal nef curve and the pushforward of a nef curve is nef, if $[F']$ is a nef curve class appearing in a nef decomposition of $[F]$ then ${\pi_j}_*[F']=k_j[F]$ for some $0\leq k_j\leq 1$. 

Hence if $[F]$ is not an extremal nef curve class, there must exist a nef curve with class $[F^{\underline{t}}]$ for $\underline{t}=(t_1,\dots,t_{n-1})\in \mathbb{R}^{n-1}$ with $\underline{t}\ne \underline{0}$ where
 $$[F^{\underline{t}}]\cdot\psi_i=[F^{\underline{t}}]\cdot\delta_{0:\{i,n\}}=1+t_i,\text{ for $i=1,\dots,n-1,$} \hspace{1cm}[F^{\underline{t}}]\cdot \psi_n=2g-(n-3),  $$

By the observation that \cite{AC}
 $$\pi_j^*\psi_1=\psi_j-\sum_{j\in S,n\notin S}\delta_{S},\hspace{0.5cm}\pi_j^*\psi_2=\psi_n-\sum_{n\in S,j\notin S}\delta_{S},\hspace{0.5cm}\pi_j^*\delta_{0:\{1,2\}}=\sum_{j,n\in S}\delta_{S},$$
 and an application of the projection formula for each $j$ we obtain 
 $$ 2g-1-\sum_{i\ne j}t_i=k_j(2g-1)\hspace{0.5cm}\text{ and }\hspace{0.5cm}1+t_j=k_j.$$
 Giving $n-1$ linearly independent relations in the $t_j$, hence the requiring all $t_i=0$ and providing the contradiction.
\end{proof}

Proposition~\ref{prop:corank} and Proposition~\ref{prop:extremal} provide the main result of this section.

\begin{proposition}
For $g\geq 2$ and $n\geq2$, the cone $\overline{\rm{Eff}}^1(\overline{\rm{H}}_{g,n})$ is non-polyhedral.
\end{proposition}

\begin{proof}
The curve class $[F]$ forms an extremal ray of ${\rm{Nef}}_1(\overline{\rm{H}}_{g,n})$ where the corank of the pseudo-effective dual space is at least two. Hence ${\rm{Nef}}_1(\overline{\rm{H}}_{g,n})$ and the dual $\overline{\rm{Eff}}^1(\overline{\rm{H}}_{g,n})$ are non-polyhedral.
\end{proof}

\begin{remark}
A long standing question is whether symmetric quotients of $\overline{\rm{M}}_{0,n}$ are Mori dream spaces. This is only known for the full symmetric quotient when $n\leq 7$, cf. \cite{Mo}. If both $\overline{\rm{M}}_{0,[2g+2]}$ and $\overline{\rm{M}}_{0,[2g+2]+1}$ are Mori dream spaces, by~\cite[Thm 1.1]{O}, the isomorphism $\overline{\rm{M}}_{0,[2g+2]}\overset{\sim}{\to} \overline{{\rm{H}}}_{g}$ and the morphism $ \overline{\rm{M}}_{0,[2g+2]+1}{\to}\overline{{\rm{H}}}_{g,1}$ would show $\overline{{\rm{H}}}_{g}$ and $\overline{\rm{H}}_{g,1}$ are also Mori dream spaces completing the classification for the moduli of pointed hyperelliptic curves. 
 \end{remark}
 
 
\section{Singularities of unordered Hurwitz spaces}
\label{section:sing_unordered}

Let $[f:C\to R, r_1,\ldots,r_{2g+2},p_1,\ldots,p_n]$ be a $\mathbb{C}$-point in $\widehat{\mathcal{H}ur}_{g,2}^n$. The completed local ring of $\overline{\mathcal{M}}_{0,2g+2+n}$ at $[R, f(r_1),\ldots,f(r_{2g+2}), f(p_1), \ldots, f(p_n)]$ is smooth and of the form 
$$\mathbb{C}\llbracket t_1,\ldots,t_{2g-1+n}\rrbracket.$$ We assume that $t_1,\ldots,t_\delta$ are smoothing parameters of the nodes $q_1,\ldots,q_\delta\in R$ and we write $f^{-1}(q_i)=\{q'_{i,1},\ldots,q'_{i,r_i}\}$ where $r_i=1$ or $2$ when $f$ is ramified or unramified at the node respectively. Then, the completed local ring of $\widehat{\mathcal{H}ur}_{g,2}^n$ at $[f]$ is given by 
\begin{equation}
\label{eq:Hurhat_loc}
\widehat{\mathcal{O}}_{\widehat{\mathcal{H}ur}_{g,2}^n, [f]}\cong \mathbb{C}\llbracket t_1,\ldots,t_{2g-1+n}, \tau_1,\tau_2, \ldots,\tau_\delta\rrbracket\big/\left(t_i=\left\{\begin{array}{cc} \tau_i^2&\hbox{ if }r_i=1\\ \tau_i&\hbox{ if }r_i=2\end{array} \right| 1\leq i\leq \delta\right).
\end{equation}
In particular, if $\mathbb{C}_{t}^{2g-1+n}$ be the versal deformation space of the $2g+2+n$-pointed target rational curve $[R, f(r_1),\ldots,f(r_{2g+2}), f(p_1), \ldots, f(p_n)]$ with coordinates $t=(t_1,\ldots,t_{2g-1+n})$ and $\mathbb{C}_\tau^{2g-1+n}$ is the versal deformation space of $[f]$ with coordinates $\tau=(\tau_1,\ldots,\tau_{2g-1+n})$, then locally at $[f]$ the map $\hat{\pi}$ is of the form $\tau_i\mapsto  \tau_i^2=t_i$ when $1\leq i\leq \delta$ and $r_i=1$ and $\tau_i\mapsto t_i$ otherwise. On coarse spaces the map $\hat{\pi}^c$ is locally given by 
$$\mathbb{C}_\tau^{2g-1+n}\big/{\rm{Aut}}(f)\longrightarrow \mathbb{C}_t^{2g-1+n}.$$
An automorphism of an $n$-pointed (ordered/unordered) admissible double cover $[f]$ is a $\sigma$-equivariant automorphism of $C$ that fixes the ($2g+2+n$ resp. $n$) markings. Here $\sigma$ stands for the cover involution. We denote the group automorphisms by ${\rm{Aut}}_\sigma(C,r_1,\ldots, r_{2g+2},p_1,\ldots,p_n)$ and ${\rm{Aut}}_\sigma(C,p_1,\ldots,p_n)$ respectively and they sit in the following commutative diagram

\begin{equation}
\begin{tikzcd}
{\rm{Aut}}_\sigma(C,r_1,\ldots r_{2g+2},p_1,\ldots,p_n)\arrow[hookrightarrow,r]\arrow[d]&{\rm{Aut}}_\sigma(C,p_1,\ldots,p_n)\arrow[d, "\Pi"]\\
\{Id\}\arrow[r]&{\rm{Aut}}(R,Br(f), f(p_1),\ldots,f(p_n)),
\end{tikzcd}
\end{equation}
where the last group consists of automorphisms of $R$ fixing the image of the markings $f(p_i)$, and the set of branch points away from the nodes, i.e., fixing the branch divisor $Br(f)=f(r_1)+\ldots+f(r_{2g+2})$. The group ${\rm{Aut}}_\sigma(C,p_1,\ldots,p_n)$ acts on $\mathbb{C}_\tau^{2g-1+n}$ and the map $\pi^c:\overline{\mathcal{H}ur}_{g,2}^n\to\overline{\mathcal{M}}_{0,[2g+2]+n}$ is analytically locally at $[f:C\to R,p_1,\ldots,p_n]$ of the form
\begin{equation}
\label{eq:pi_local}
\pi^c_{[f]}:\mathbb{C}_\tau^{2g-1+n}\big/{\rm{Aut}}_\sigma(C,p_1,\ldots,p_n)\longrightarrow \mathbb{C}_t^{2g-1+n}\big/{\rm{Aut}}_\sigma\left(R,Br(f), f(p_1),\ldots,f(p_n)\right).
\end{equation}
We follow \cite{HM} and study the singularities of $\overline{\rm{Hur}}_{g,2}^n$ by means of the {\textit{Reid--Tai criterion}} that we now recall.\\ 

Let $V$ be a finite $d$-dimensional vector space and $G$ a finite group acting linearly on $V$. Let $g\in G\subset GL(V)$ with eigenvalues $e^{2\pi i r_i}$, $0\leq r_i<1$. The {\textit{age}} of $g$ is defined as the sum ${\rm{age}}(g)=r_1+\ldots+r_d$. The Reid--Tai criterion \cite{Re} states that if $G$ contains no quasi-reflections, then the quotient $V\big/G$ has canonical singularities if and only if for every $g\in G$, the age of $g$ satisfies
$${\rm{age}}(g)\geq1.$$ 
A quasi-reflexion is an element $\tau\in GL(V)$ that fixes a hyperplane or equivalently, that is diagonalisable as
$$\tau={\rm{diag}}\left(e^{2\pi i r_1}, 1, 1,\ldots,1\right).$$\\
There are two ways to deal with groups containing quasi-reflexions. One is to consider a modified Reid--Tai sum as in \cite[Prop. 5.11]{GHS} with the downside that we lose the `only if' part of the criterion. Let $g\in GL(V)$ be an element order $m=s\cdot k$ where $k$ is the smallest positive integer such that $g^k$ is a quasi-reflexion and assume that the only non-trivial eigenvalue of $g^k$ is $\xi=e^{2\pi i r_1\cdot k}$. Then $\xi$ is an $s$-root of unity. We define the {\textit{modified age of $g$}} as 
$${\rm{age}}'(g)=(r_1+\ldots+r_n)+sr_1.$$
The weaker criterion then says that $V\big/G$ has canonical singularities if for all $g\in G$, the modified age satisfies ${\rm{age}}'(g)\geq1$.\\

The second strategy is to look at the action of $G$ on $V/H$, where $H$ is the subgroup generated by quasi-reflexions in $G$, see \cite{P}. Then $V/H$ is smooth, $G$ acts without quasi-reflections, and the original criterion applies. A difficulty of this strategy is describing the subgroup generated by quasi-reflections and describing the action of the quotient on the new coordinate space $V/H$.\\

Before dealing with quasi-reflexions we need a more canonical description of the versal deformation space in order to understand the action of ${\rm{Aut}}_\sigma\left(C,p_1,\ldots,p_n\right)$. We will make full use of the map already considered in \eqref{eq:phi_tilde} given by
\begin{equation}
\label{eq:phi_tilde_stack}
\widetilde{\Phi}:\overline{\rm{Hur}}_{g,2}^n\longrightarrow\overline{\rm{M}}_{g,[2g+2]+2n}
\end{equation}
sending an $n$-pointed admissible double cover
$$[f:C\to R, p_1,\ldots,p_n]\mapsto[C,r_1+\ldots+r_{2g+2}, p_1, \overline{p_1},\ldots,p_n,\overline{p_n}]$$
where $\overline{p_i}$ is the conjugate of $p_i$ and $r_1+\ldots+r_{2g+2}$ is the ramification divisor of $f$ away from the nodes. This map is an isomorphism onto the image, cf. Lemma \ref{sec2:lemmaSvZ} and \cite[Thm 3.7]{SvZ} (see also proof of  Proposition \ref{prop:exc_loc}).\\

In \cite[Ch. XI, Lemma 6.15]{ACG} smoothness of $\overline{\mathcal{H}}_g$ is obtained implicitly via the morphism \eqref{eq:phi_tilde_stack} for $n=0$ 
$$\overline{\mathcal{H}ur}_{g,2}\longrightarrow\overline{\mathcal{H}}_{g}.$$ 
First order deformations of $[C]\in \overline{\mathcal{H}}_g$ are parameterised by $\sigma$-invariant first order deformations of $C$ where $\sigma$ is the hyperelliptic involution. One can see that the data of the cover $C\to R$ determines uniquely $(C,\sigma)$.  This is no longer the case for $\overline{\mathcal{H}}_{g,n}$, since there is stable reduction involved and the resulting stable model may not have an involution. By remembering the ramification divisor and the conjugate markings the resulting curve is always stable. The involution $\sigma$ induced by $f$ acts on 
$$Ext^1\left(\Omega_C,\mathcal{O}_C\left(-p_1-\overline{p_1}-\ldots-p_n-\overline{p_n}\right)\right)$$
and the same argument as in \cite[Ch. XI, Lemma 6.15]{ACG} (see also \cite[Section 3]{ACV}) identifies first order deformations of $[f:C\to R, p_1,\ldots,p_n]$ with the $\sigma$-invariant subspace 
\begin{equation}
\label{eq:ext_sigma}Ext^1\left(\Omega_C,\mathcal{O}_C\left(-p_1-\overline{p_1}-\ldots-p_n-\overline{p_n}\right)\right)^{\sigma}\cong \left(H^0\left(C,\Omega_C\otimes\omega_C\left(\sum p_i+\overline{p_i}\right)\right)^\vee\right)^{\sigma}.
\end{equation}
Then the coarse space $\overline{\mathcal{H}ur}_{g,2}^n$ around $\left[f:C\to \mathbb{P}^1,p_1,\ldots,p_n\right]$ is analytically locally isomorphic to the quotient
\begin{equation}
\label{eq:ann_loc1}
Ext^1\left(\Omega_C,\mathcal{O}_C\left(-p_1-\overline{p_1}-\ldots-p_n-\overline{p_n}\right)\right)^{\sigma}\big/ {\rm{Aut}}_\sigma(C,p_1,\ldots,p_n).
\end{equation}
Or equivalently to the quotient
\begin{equation}
\label{eq:ann_loc2}
\resizebox{.95\hsize}{!}{$H^0\left(C,\Omega_C\otimes\omega_C\left(p_1+\overline{p_1}+\ldots+p_n+\overline{p_n}\right)\right)^{\sigma}\big/ {\rm{Aut}}_\sigma(C,p_1,\ldots,p_n)\cong \mathbb{C}_\tau^{2g-1+n}\big/{\rm{Aut}}_\sigma(C,p_1,\ldots,p_n).$}
\end{equation}
Let $\coprod_\alpha R_\alpha$ be the normalisation of the target rational curve $R$. Note that for each component $R_\alpha$ there are two possibilities for the pre-image $C_\alpha=f^{-1}(R_\alpha)$, either it is smooth, irreducible, and $\sigma$ acts as the hyperelliptic involution or $C_\alpha\cong R_\alpha\coprod R_\alpha$ is the union of two copies of $R_\alpha$ being interchanged by $\sigma$ (see Figures \ref{figure:Pic1} and \ref{figure:contraction}). Similarly, if $q\in Sing(R)$, then either $f^{-1}(q)$ consists on two nodes interchanged by $\sigma$ of one node being fixed by $\sigma$. Recall the exact sequence \cite[p. 33]{HM}:
\begin{equation}
\label{eq:tor_sequence}
\scalebox{.9}{$
0\to \bigoplus_{q\in Sing(R)}{\rm{tor}}_{f^{-1}(q)}\to H^0\left(C,\Omega_C\otimes\omega_C\left(\sum_{i=1}^n p_i+\overline{p_i}\right)\right) \to \bigoplus_\alpha H^0\left(C_\alpha,\omega_{C_\alpha}^{\otimes 2}\left(N_\alpha+P_\alpha\right)\right) \to0,$}
\end{equation}
where $P_\alpha=\sum_{p_i\in C_\alpha}(p_i+\overline{p_i})$ is the sum of the preimages of markings (and conjugates) that land in $C_\alpha$ after normalisation, and $N_\alpha$ is the sum of the preimages of the nodes that land in $C_\alpha$ after normalisation. Here 
$${\rm{tor}}_{f^{-1}(q)}=\left\{\begin{array}{lcl}{\rm{tor}}_p&\hbox{if}&f^{-1}(q)=p\\
{\rm{tor}}_p\oplus {\rm{tor}}_{p'}&\hbox{if}&f^{-1}(q)=\{p,p'\},
\end{array}\right.$$
and ${\rm{tor}}_p$ is the one-dimensional vector space of torsion differentials based at $p\in Sing(C)$. The one dimensional vector space ${\rm{tor}}_p\cong \mathbb{C}_z$ parameterises deformations that smooth the node $p$ (we call $z$ a {\textit{smoothing parameter}}) and the summands on the right of \eqref{eq:tor_sequence} correspond to deformations that preserve the nodes and deform the corresponding pointed components. Note that $\sigma$ acts on each vector space on the sequence \eqref{eq:tor_sequence}. Moreover, $\sigma$ acts as the identity on ${\rm{tor}}_{f^{-1}(q)}$ if $f^{-1}(q)=p$ and by interchanging the factors if $f^{-1}(q)=\{p,p'\}$, see the proof of \cite[Ch. XI, Lemma 6.15]{ACG} for the dual statement. The cover involution $\sigma$ also acts on the normalisation $\coprod C_\alpha$ by acting on each $C_\alpha$ as the hyperelliptic involution $\sigma_\alpha$ of $C_\alpha\to R_\alpha$ if $C_\alpha$ is irreducible and by interchanging the two copies of $R_\alpha$ if $C_\alpha= R_\alpha\coprod R_\alpha$. We partition the nodes of $R$ as $Sing(R)=Q_1\sqcup Q_2$, where 
$$Q_i=\left\{q\in Sing(R)\left|\left|f^{-1}(q)\right|=i\right.\right\}.$$
Then we have the following exact sequence:
\begin{equation}
\label{eq:torinvariant_sequence}
\scalebox{.75}{$
0\to \bigoplus_{q\in Q_1}{\rm{tor}}_{f^{-1}(q)}\oplus \bigoplus_{q\in Q_2}V_{f^{-1}(q)}\to H^0\left(C,\Omega_C\otimes\omega_C\left(\sum_{i=1}^n p_i+\overline{p_i}\right)\right)^\sigma \to \bigoplus_\alpha H^0\left(C_\alpha,\omega_{C_\alpha}^{\otimes 2}\left(N_\alpha+P_\alpha\right)\right)^{\sigma_\alpha} \to0.$}
\end{equation}
Here $V_{f^{-1}(q)}$ is the one-dimensional subspace of ${\rm{tor}}_p\oplus {\rm{tor}}_{p'}$ where $\sigma$ acts trivially. Naturally one has that the vector space ${\rm{tor}}_{f^{-1}(q)}$ for $q\in Q_1$ parameterises smoothings of the node $f^{-1}(q)=p$. While the vector space $V_{f^{-1}(q)}$ for $q\in Q_2$ parameterises simultaneous smoothing of $p$ and $p'$, and the summands on the right of \eqref{eq:torinvariant_sequence} parameterise deformations of the pointed components preserving the nodes and cover structure. \\ 

We denote the degree $2g+2$ branch divisor of $f:C\to R$ disjoint from the nodes by $B$ and $B_\alpha$ the restriction to $R_\alpha$. The divisor $B_\alpha$ is contained in the branch divisor $Br(f_\alpha)$ of $f_\alpha:C_\alpha\to R_\alpha$, but not necessarily equal. We have the analogous exact sequence on $R$:
\begin{equation}
\label{eq:tor_sequenceR}
\scalebox{.8}{$
0\to \bigoplus_{q\in Sing(R)}{\rm{tor}}_q\to H^0\left(R,\omega_R\otimes\Omega_R\left(B+\sum_{i=1}^n f(p_i)\right)\right) \to \bigoplus_\alpha H^0\left(R_\alpha,\omega_{R_\alpha}^{\otimes 2}\left(B_\alpha+f(P)_\alpha+M_\alpha\right)\right) \to0,$}
\end{equation} 
where $f(P)_\alpha=\sum_{p_i\in C_\alpha} f(p_i)$ and $M_\alpha$ is the sum of the preimages of the nodes of $R$ that land in $R_\alpha$ after normalisation $\coprod R_\alpha\to R$. One observes the natural identification

\begin{equation}
\label{eq:tor_pq}\left(\bigoplus_{q\in Sing(R)}{\rm{tor}}_{f^{-1}(q)}\right)^\sigma\cong \bigoplus_{q\in Q_1}{\rm{tor}}_{f^{-1}(q)}\oplus \bigoplus_{q\in Q_2}V_{f^{-1}(q)}\cong \bigoplus_{q\in Q_1}{\rm{tor}}_q\oplus \bigoplus_{q\in Q_2}{
\rm{tor}}_q.
\end{equation}
Similarly, the natural identification 
$$H^0\left(C_\alpha, \omega_\alpha^{\otimes 2}(P_\alpha)\right)^{\sigma_\alpha}\cong H^0\left(R_\alpha, \omega_{R_\alpha}^{\otimes 2}(Br(f_\alpha)+f(P)_\alpha)\right)$$ 
induces an inclusion
\begin{equation}
\label{eq:H0_pq}
\bigoplus_\alpha H^0\left(R_\alpha,\omega_{R_\alpha}^{\otimes 2}\left(B_\alpha+f(P)_\alpha+M_\alpha\right)\right)\subset \bigoplus_{\alpha}H^0\left(C_\alpha,\omega_{C_\alpha}^{\otimes 2}\left(N_\alpha+P_\alpha\right)\right)^{\sigma_\alpha}.
\end{equation}
We want to obtain a lower bound on the age of a $\sigma$-invariant automorphism of the $n$-pointed curve $C$ acting on \eqref{eq:ext_sigma} by looking at the action of the induced automorphism $\Pi(g)$ on the $([2g+2]+n)$-pointed rational curve $R$ and use the identifications \eqref{eq:tor_pq} and \eqref{eq:H0_pq}. 

\begin{lemma}
\label{lemma:action_up-down}
Let $$\Pi:{\rm{Aut}}_\sigma\left(C,p_1,\ldots,p_n\right)\longrightarrow {\rm{Aut}}\left(R,Br(f), f(p_1),\ldots,f(p_n)\right)$$ be the morphism sending a $\sigma$-equivariant automorphism $g$ to the induced automorphism $\Pi(g)$ on $R$. Then, $\Pi(g)$ acts linearly on the versal deformation space
$$\mathbb{C}_t^{2g-1+n}\cong H^0\left(R,\omega_R\otimes\Omega_R\left(B+\sum_{i=1}^n f(p_i)\right)\right),$$ and the action can be decomposed as
\begin{equation}
\label{eq:structurePi(g)}
\Pi(g)=T_1\oplus T_2\oplus B,
\end{equation}
where $T_i$ acts on $\oplus_{q\in Q_i}{\rm{tor}}_{q}$, and $B$ acts on the right hand side of \eqref{eq:tor_sequenceR}. Moreover, the action of $g$ can be decomposed as 
$$g=\widetilde{T}_1\oplus T_2\oplus B',$$
where the action of $T_2$ is via the identification \eqref{eq:tor_pq} and $B'$ acts on the right hand side of \eqref{eq:H0_pq} and restricts to the left hand side as $B$. If we assume further that the automorphism $g$ preserves a node of type $Q_1$, i.e., $\widetilde{T}_1$ acts on ${\rm{tor}}_{f^{-1}(q_i)}$ by $\widetilde{T}_1(\tau_i)=\eta\cdot \tau_i$, then $T_1$ acts on ${\rm{tor}}_{q_i}$ by $T_1(t_i)=\eta^2\cdot t_i$.  
\end{lemma}

\begin{proof}
The decomposition \eqref{eq:structurePi(g)} follows simply from the fact that $\Pi(g)$ must preserve the nodal structure, i.e., smoothing parameters are sent to smoothing parameters. Moreover, since it is induced by $g$, nodes of type $Q_i$ must be sent to nodes of type $Q_i$. Locally if $q_i\in Q_1=\left\{q_1,\ldots,q_{\delta_1}\right\}$, and ${\rm{tor}}_{q_i}\cong \mathbb{C}_{t_i}$ is the smoothing parameter of $q_i$ and ${\rm{tor}}_{p_i}\cong \mathbb{C}_{\tau_i}$ is the smoothing parameter of $f^{-1}(q_i)=p_i$, then locally the map that forgets the cover is of the form $\tau_i\mapsto \tau_i^{2}=t_i$ for $0\leq i\leq \delta_{1}$ and $\tau_i=t_i$ for all other $\delta_1<i\leq 2g-1+n$. Assume further that $g$ acts on each $\mathbb{C}_{\tau_i}$ as $g(\tau_i)=\eta\tau_i$. Then $\Pi(g)(t_i)=\eta^2t_i$ if $i\leq\delta_1$ and $\Pi(g)(t_i)=\eta t_i$ otherwise.
\end{proof}

Following \cite{Lu}, we call automorphisms $g\in {\rm{Aut}}_\sigma\left(C,p_1,\ldots,p_n\right)$ that act trivially on $R$, i.e., such that $\Pi(g)=Id$, {\textit{inessential}} and denote them by ${\rm{Aut}}_\sigma\left(C,p_1,\ldots,p_n\right)^\circ={\rm{Ker}}(\Pi)$. 

\begin{lemma}
\label{lemma:inessential}
Let $g\in {\rm{Aut}}_\sigma\left(C,p_1,\ldots,p_n\right)^\circ$ be an inessential automorphism acting as quasi-reflexion on $\mathbb{C}_{\tau}^{2g-1+n}$. Then, $C$ is singular and has a connected subcurve $C'\subset C$ (not necessarily smooth) with no markings meeting the complementary curve $\overline{C\setminus C'}$ at only one node $p=f^{-1}(q)$ of type $Q_1$. In such case $g$ is the identity on $\overline{C\setminus C'}$ and the involution $\sigma'$ on $C'$ and $g$ acts on ${\rm{tor}}_q\cong\mathbb{C}_{\tau_i}$ as $\tau_i\mapsto -\tau_i$ and as the identity $\tau_j\mapsto \tau_j$ for all other parameters on $\mathbb{C}_{\tau}^{2g-1+n}$. Moreover, the group of inessential automorphisms ${\rm{Aut}}_\sigma\left(C,p_1,\ldots,p_n\right)^\circ$ is generated by such quasi-reflexions.
\end{lemma}

\begin{proof}
Let $g\in {\rm{Aut}}_\sigma(C,p_1,\ldots,p_n)$ such that the induced automorphism on the $([2g+2]+n)$-pointed rational curve is the identity. This can only happen if for every $p\in C$, $g(p)$ is either $p$ or the conjugate $\overline{p}$, i.e., if $g$ restricted to every component $C_\alpha$ is either the identity or the involution $\sigma_\alpha$. In particular $C$ cannot be smooth, since $\sigma$ acts as the identity on \eqref{eq:ext_sigma}. Let $C'$ be the union of component of $C$ where $g$ acts as non-trivially, i.e, as the involution $\sigma'$. Then $C'$ has no markings and meets the curve $\overline{C\setminus C'}$ only on nodes of type $Q_1$ and we can assume $C'\subsetneq C$. Then, $g$ acts as $\tau_i\mapsto-\tau_i$ on ${\rm{tor}}_p\cong \mathbb{C}_{\tau_i}$ for all nodes $p$ on the intersection $C'\cap\overline{C\setminus C'}$ and as the identity on the right hand side of \eqref{eq:torinvariant_sequence}. Then, $g$ is diagonalised as ${\rm{diag}}\left(-1,-1,\ldots,1\right)$, where the number of $-1$'s is the number of $Q_1$-type nodes $C'\cap\overline{C\setminus C'}$. One observes then that $g$ is an inessential quasi-reflexion if and only if $C'$ is connected and meets the rest of the curve $\overline{C\setminus C'}$ at only one node of type $Q_1$. Since any inessential automorphism is either the identity of the involution on each component, it is clear that ${\rm{Aut}}_\sigma\left(C,p_1,\ldots,p_n\right)^\circ$ is generated by such quasi-reflections.
\end{proof}

We start with a proposition that allows us to rule out the case when $C$ is smooth and will provide us with lower bonds on the age in the general case. 

\begin{proposition}
\label{prop:actionP1}
Let $h\in {\rm{Aut}}(\mathbb{P}^1, b_1+\ldots+b_r, p_1,\ldots,p_s)$ be an automorphism of order $\ell\geq 2$ on an $([r]+s)$-pointed smooth rational curve acting linearly on the versal deformation space $H^0\left(\omega_{\mathbb{P}^1}^{\otimes 2}\left(\sum b_i+\sum p_j\right)\right)$. Then $h$ acts as a quasi-reflexion if and only if one of the following conditions is satisfied:
\begin{itemize}
\item[(i)] $r+s=4$, where ${\rm{age}}(h)$ is $2/\ell$ if $s=0$ and no point in $\{b_1,\ldots,b_r\}$ is fixed by $h$, and $1/\ell$ otherwise.
\item[(ii)] $s=0$, $r=5$, $\ell=2$, and at least one unordered point is fixed by $h$. In this case ${\rm{age}}(h)=1/2$.
\item[(iii)] $s=0$, $r=6$, and $\ell=2$. In this case $\{b_1,\ldots,b_6\}$ is the union of three $h$-orbits and ${\rm{age}}(h)=1/2$.
\item[(iv)] $s=1, r=4$, and $\ell=2$. In this case $p_1$ is fixed by $h$, $\{b_1,\ldots,b_4\}$ is the union of two $h$-orbits, and ${\rm{age}}(h)=1/2$.
\end{itemize}
If $h$ is not a quasi-reflexion, then ${\rm{age}}(h)\geq 1$ with the only exceptions 
\begin{itemize}
\item $s=0, r=5, \ell=4$, and $\{b_1,\ldots,b_5\}$ consist of one fixed point and one non-trivial $h$-orbit.
\item $s=1, r=4, \ell=4$, $p_1$ being fixed by $h$ and $\{b_1,\ldots,b_4\}$ is one $h$-orbit. 
\end{itemize}
In these cases ${\rm{age}}(h)=\frac{3}{4}$. Moreover, if $h$ is not a quasi-reflexion, then $h^k$ is a quasi-reflexion only in two cases:
\begin{itemize}
\item $s=0, r=5, \ell=4$ and $\{b_1,\ldots,b_5\}$ consist on a fixed point and a single $h$-orbit. In this case $h^2$ is as in (ii) and the modified age satisfies ${\rm{age}}'(h)\geq1$.
\item $s=0, r=6, \ell=6$ and $\{b_1,\ldots,b_6\}$ is a single $h$-orbit. In this case $h^3$ is as in (iii) and the modified age satisfies ${\rm{age}}'(h)\geq1$. 
\item $s=1, r=4, \ell=4$, $p_1$ being fixed by $h$ and $\{b_1,\ldots,b_4\}$ is one $h$-orbit. In this case $h^2$ is as in (iv) and the modified age satisfies ${\rm{age}}'(h)\geq1$.
\end{itemize}
In particular ${\rm{age}}'(h)\geq 1$ for any non-trivial $h\in {\rm{Aut}}(\mathbb{P}^1, b_1+\ldots+b_r, p_1,\ldots,p_s)$.
\end{proposition}

\begin{proof}
The automorphism $h$ induces a cyclic ramified cover $\pi_{h}:\mathbb{P}^1\to \mathbb{P}^1\big/\langle h\rangle\cong\mathbb{P}^1$ ramified at two points $0,\infty\in \mathbb{P}^1$. Moreover, we may assume $s\leq 2$ (otherwise $g$ is trivial), $[r]+s\geq4$ so that the $([r]+s)$-pointed rational curve is not rigid, and $g$ acts on $\mathbb{P}^1$ as $z\mapsto\eta\cdot z$ where $\eta=e^{\frac{2\pi i}{\ell}}$ is a primitive $\ell$-root of unity. Since the map $z\mapsto 1/z$ preserves non-trivial $h$-orbits, if there is an unordered point being fixed by $h$ we can always assume is $0$.  

There are four cases to consider, when $s=0$ and no point is fixed among unordered point, when $s=0$ and at least one unordered point is fixed, when $s=1$ and $p_1=0$, and when $s=2$ and $p_1=0, p_2=\infty$. We label them:
\begin{equation}
\begin{array}{c}
\label{eq:ABC}
(A1): s=0\;\;\hbox{and no fixed unordered point,}\\
(A2): s=0\;\;\hbox{and at least one fixed unordered point,}\\
(B): (s,p_1)=(1,0),\;\;\;\hbox{and}\;\;\;(C):(2, p_1, p_2)=(2,0,\infty).
\end{array}
\end{equation} 
The set of unordered markings $\{b_1,\ldots,b_r\}$ is a union of $a$-many $h$-orbits over the points $t_1,\ldots,t_a\in \mathbb{P}^1\big/\langle h \rangle$ on the quotient. An $h$-eigenbasis for the space of quadratic differentials on $\mathbb{P}^1$ having at worst simple poles at the markings is given by $\{\omega, \omega\cdot Z, \ldots, \omega\cdot Z^{r+s-4}\}$, where
\begin{equation}
\label{eq:basis_dz2}
\omega=\frac{dZ^{\otimes 2}}{Z^m\cdot \prod_{j=1}^a\left(Z^\ell-t_j\right)},
\end{equation}
with $r=\ell\cdot a$, and $m=0$ in case (A1), $r=\ell\cdot a+\delta$ and $m=1$ in cases (A2), (B) and (C). Here $\delta$ is the number of fixed unordered points, i.e., $\delta\in \{1,2\}$ in case (A2), $\delta\in\{0,1\}$ in case (B) and $\delta=0$ in case (C).  One observes that $h$ acts on the linear form $Z$ by $Z\mapsto \eta\cdot Z$, thus
$$h={\rm{diag}}\left(\eta^i, \eta^{i+1},\ldots,\eta^{i+(r+s-4)}\right),$$
where $i=2$ in case $(A1)$ and $i=1$ otherwise. A simple case-by-case study shows that this diagonal matrix has a single non-trivial eigenvalue only in the listed cases.
\end{proof}

\begin{remark}
When the $([r]+s)$-pointed rational curve has a node, in most cases $h$ will cease to act as quasi-reflections on the versal deformation space $\mathbb{C}_t^{2g-1+n}$, even if $h$ is non-trivial only on one component. This is due to the contribution to the age coming from $\oplus_q {\rm{tor}}_q$.  
\end{remark}

We obtain the following immediate corollaries.
\begin{corollary}
For any $r,s$ with $r+s\geq4$ the moduli space $\rm{M}_{0,[r]+s}$ has at worst canonical singularities.
\end{corollary}
\begin{proof}
It follows immediately from the Proposition together with \cite[Prop. 5.11]{GHS}.
\end{proof}
\begin{corollary}
The moduli space ${\rm{Hur}}_{g,2}^n$ has at worst canonical singularities.
\end{corollary}
\begin{proof}
On the interior the only inessential automorphism is the covering involution $\sigma$ that acts trivially on the versal deformation space. Moreover, the forgetful morphism 
$${\rm{Hur}}_{g,2}^n\longrightarrow{\rm{M}}_{0,[2g+2]+n}$$
is \'{e}tale. 
\end{proof}

From now on we can assume that $C$ has at least one node. An automorphism $g\in {\rm{Aut}}_\sigma\left(C,p_1,\ldots,p_n\right)$ is called a {\textit{rational bridge involution}} if $C\to R$ has a rational bridge $C_\alpha\to R_\alpha$ such that $\Pi(g)$ acts on $R_\alpha$ by interchanging the two branch points and fixing the node, and $\Pi(g)$ acts trivially on $\overline{R\setminus R_\alpha}$. 

\begin{lemma}
\label{lemma:allqr}
The normal subgroup $H\subset{\rm{Aut}}_{\sigma}\left(C,p_1,\ldots,p_n\right)$ generated by all quasi-reflections is generated by inessential automorphisms ${\rm{Aut}}_{\sigma}\left(C,p_1,\ldots,p_n\right)^{\circ}$ and rational bridge involutions. In particular, if $C$ has no unmarked rational bridge, then $H={\rm{Aut}}_{\sigma}\left(C,p_1,\ldots,p_n\right)^{\circ}$.
\end{lemma}

\begin{proof}
Let $g\in{\rm{Aut}}_{\sigma}\left(C,p_1,\ldots,p_n\right)$ be a quasi-reflexion with a single non-trivial eigenvalue $\eta=e^{\frac{2\pi i}{\ell}}$. By Lemma \ref{lemma:inessential} it is enough to show that if $\Pi(g)$ is non-trivial, then $g$ is a rational bridge involution. Assume $\Pi(g)$ is a non-trivial automorphism of $(R,Br(f),f(p_1),\ldots,f(p_n))$. Then it must act non-trivially on at least one component $R_\alpha$. Let $f_\alpha:C_\alpha\to R_\alpha$ be the corresponding $2$-to-$1$ cover. We can assume that $g$ acts non-trivially on the factor
\begin{equation}
\label{eq:lemma_qr0}\left(\bigoplus_{q\hbox{ node in }R_\alpha}{\rm{tor}}_{f^{-1}(q)}\right)^{\sigma}\oplus H^0\left(C_\alpha, \omega^{\otimes 2}_{C_\alpha}\left(N_\alpha+P_\alpha\right)\right)^{\sigma_\alpha}.
\end{equation}

There are two cases to consider.

{\bf{Case 1.}}  If $R_\alpha$ meets the rest of the curve $\overline{R\setminus R_\alpha}$ at one node $q$, then $g$ must fix $q$ and it follows from the proof of Proposition \ref{prop:actionP1} that $\Pi(g)$ acts non-trivially on the factor 
\begin{equation}
\label{eq:lemma_qr1}
H^0\left(R_\alpha, \omega_{R_\alpha}^{\otimes 2}\left(B_\alpha+f(P)_\alpha+q\right)\right)
\end{equation}
corresponding to $R_\alpha$ on the right hand side of \eqref{eq:tor_sequenceR}, unless \eqref{eq:lemma_qr1} has dimension zero. If \eqref{eq:lemma_qr1} is not zero, then $s=1$ and we are in cases (i) or (iii) of Proposition \ref{prop:actionP1}. Recall that if locally $q$ in $R$ is of the form $\{x\cdot y=0\}$ with $x$ the local parameter of $R_\alpha$ and $y$ the one of $\overline{R\setminus R_\alpha}$ around the node. Then, ${\rm{tor}}_q\cong \mathbb{C}_t$ is generated the torsion differential $y(dx)^{\otimes 2}/x=x(dy)^{\otimes 2}/y$, see \cite[p. 33]{HM}, and $\Pi(g)$ acts by $x\mapsto \eta_1 x$ on $R_\alpha$ and as $y\mapsto \eta_2y$ on the component $R_\beta\subset R$ meeting $R_\alpha$ at $q$. In particular, $t\mapsto\eta_1\cdot \eta_2\cdot t$ and $\Pi(g)$ acts non-trivially on both \eqref{eq:lemma_qr1} and $\mathbb{C}_t$ unless $\eta_1=\eta_2^{-1}$ in which case $\Pi(g)$ acts on
\begin{equation}
\label{eq:lemma_qr2}
H^0\left(R_\beta, \omega_{R_\beta}^{\otimes 2}\left(B_\beta+f(P)_\beta+q\right)\right).
\end{equation}
and either $\Pi(g)$ acts non-trivially on ${\rm{tor}}_{q'}$ for a node $\{q'\}=R_\beta\cap R\setminus R_\alpha$ or the action of $\Pi(g)$ on \eqref{eq:lemma_qr2} is non-trivial. By Lemma \ref{lemma:action_up-down} one concludes that $\Pi(g)$ cannot be a quasi-reflexion unless $\left|B_{\alpha}\right|=2$, $\left|f(P)_\alpha\right|=0$, and \eqref{eq:lemma_qr1} is zero-dimensional. In this case, if one identifies $q=0\in R_\alpha$, then $\Pi(g)$ restricted to $R_\alpha$ has order $2$ and acts as $z\mapsto -z$. This forces $g$ to act on the $2$-to-$1$ cover $f_\alpha:C_\alpha\to R_\alpha$ with order $2$ or $4$ preserving the preimages of the two fixed points on $R_\alpha$ (the node and $\infty$) and permuting the two ramification points. Let $p, p'\in C_\alpha$ be the two preimages of the node $0\in R_\alpha$. One observes that $g$ acts on ${\rm{tor}}_p\oplus{\rm{tor}}_{p'}$ as $(t,t')\mapsto(-t,-t')$ or $(t,t')\mapsto(-t',-t)$. In any case acts with eigenvalue $-1$ on the $\sigma_\alpha$-invariant subspace $V_{f^{-1}(0)}$ and therefore on \eqref{eq:lemma_qr0}. Since $g$ is a quasi-reflexion, it must act trivially on all other factors of \eqref{eq:lemma_qr0}. This shows that $g$ is a rational bridge involution. 

{\bf{Case 2.}} The case when $R_\alpha$ meets the rest of the curve $\overline{R\setminus R_i}$ at two nodes $q_1, q_2$ is similar and follows again from the contribution to the age coming from Proposition \ref{prop:actionP1} plus the one coming from the nodes. One reduces to the case when $\Pi(g)$ acts on $R_\alpha$ as an involution interchanging the nodes and the $2$-ramification points, but then the action of $g$ on the cover $C_\alpha \to R_\alpha$ has non-trivial contribution to the age from both, $V_{f^{-1}(q_1)}\oplus V_{f^{-1}(q_2)}$ and from the left factor on \eqref{eq:lemma_qr0}. This contradicts the assumption that $g$ is a quasi-reflexion.
\end{proof}

\begin{theorem}
\label{thm:non-can_sing}
An $n$-pointed admissible cover $[f:C\to R, p_1,\ldots,p_n]\in\overline{\mathcal{H}ur}_{g,2}^n$ is a non-canonical singularity if and only if it has an elliptic tail $E$ with no markings and $j$-invariant $0$. If we fix the node at the orign $0\in E$, then $f$ restricts to the quotient $f:E\to E\big/\langle\pm 1\rangle\cong \mathbb{P}_1$ and the remining Weierstrass points are sent to a $\mathbb{Z}\big/3\mathbb{Z}$-orbit $\{z, \eta z, \eta^{2}z\}\subset \mathbb{P}^1$, with $\eta$ a $3$-root of unity.
\end{theorem}

\begin{proof}
First, note that by Lemmas \ref{lemma:inessential} and \ref{lemma:allqr}
$$\mathbb{C}_{\tau}^{2g-1+n}\big/H\cong \mathbb{C}_{t'}^{2g-1+n}$$
and ${\rm{Aut}}_\sigma\left(C, p_1,\ldots,p_n\right)$ acts without non-trivial quasi-reflections on the latter. Moreover, 
$$\mathbb{C}_{t}^{2g-1+n}\cong \mathbb{C}_{t'}^{2g-1+n}$$
via $t_i=t_i'$ if $C$ ha no unmarked rational bridge and $t_i^2=t_i'$ if $C$ has as an unmarked rational bridge and $t_i$ is the smoothing parameter of the corresponding node on $R$. Let 
$$h\in \Pi\left({\rm{Aut}}_\sigma\left(C, p_1,\ldots,p_n\right)\right)\subset{\rm{Aut}}\left(R, Br(f), f(p_1),\ldots,f(p_n)\right)$$
be a non-trivial automorphism of the singular $([2g+2]+n)$-pointed rational curve $R$ induced by a $\sigma$-equivariant automorphism on the pointed cover. Then $h$ must act non-trivially on at least one (smooth) component $R_\alpha$.\\

{\bf{Case 1.}} Assume first that $R_\alpha$ meets $\overline{R\setminus R_\alpha}$ in two nodes. Stability forces the number of (ordered and unordered) markings and nodes $r+s$ on $R_\alpha$ to be at least $3$. If the nodes are fixed by $h$, then $s\geq2$ and we can assume $s=2$ and $r\geq 2$, otherwise $h$ is the identity. We fix $0,\infty\in R_\alpha$ as the fixed nodes and $h$ acts by $z\mapsto \eta z$ where $\eta$ ia a primitive $\ell$-root of unity. Then one observes that the action on ${\rm{tor}}_{q_i}\cong \mathbb{C}_{t_i}$ the node at $0=q_i$ is given by $t_i\mapsto \eta t_i$ and the action on the node at infinity ${\rm{tor}}_{q_j}\cong \mathbb{C}_{t_j}$ is given by $t_j\mapsto \eta^{-1}t_j$. In any case the contribution to the age is at least $1$. If the nodes are interchanged, then the contribution from $\mathbb{C}_{t_i}\oplus \mathbb{C}_{t_j}$ is of $1/2$. Moreover, since the automorphism is in the image of $\Pi$, both nodes have to be of the same type and $r$ must be even. Then, from the proof of Proposition \ref{prop:actionP1} is follows that the contribution from 
$$H^0\left(R_\alpha, \omega^{\otimes 2}_{R_\alpha}\left(B_\alpha+f(P)_\alpha+q_i+q_j\right)\right)$$
is at least $1/2$ unless $\left|B_\alpha\right|=0$ and $s=\left|f(P)_\alpha\right|=1$. But one immediately observes that if this is the case the the two components $R_{\beta_1}$ and $R_{\beta_2}$ meeting $R_\alpha$ at the interchanged nodes must be sent isomorphically to each other giving a contribution of at least $1/2$ from the action of $h$ on 
$$H^0\left(R_{\beta_1}, \omega^{\otimes 2}_{R_{\beta_1}}\left(B_{
\beta_1}+f(P)_{\beta_1}+M_{\beta_1}\right)\right)\oplus H^0\left(R_{\beta_2}, \omega^{\otimes 2}_{R_{\beta_2}}\left(B_{
\beta_2}+f(P)_{\beta_2}+M_{\beta_2}\right)\right)$$
or, if the above spaces are trivial, a contribution of at least $1/2$ coming from the action of ${\rm{tor}}$-spaces.\\

{\bf{Case 2.}} Assume that $R_\alpha$ meets $\overline{R\setminus R_\alpha}$ at one node $q$ that we fix as $0\in R_\alpha$. Then $h$ acts as $z\mapsto\eta z$ and $q$ is fixed by $h$. The contribution from ${\rm{tor}}_q$ is at least $1/\ell$. By Proposition \ref{prop:actionP1}, $1/\ell$ plus the contribution to the age coming from 
$$H^0\left(R_\alpha, \omega^{\otimes 2}_{R_\alpha}\left(B_\beta+f(P)_\alpha+q\right)\right)$$
is at least $1$ unless $r+s=3$ or $r=3, s=1$ and we are in case (i) of the proposition, i.e., the contribution is $1/3$. In the first case, either $s=2$, $r=1$ and the action is trivial, or $s=1, r=2$ and the action is a rational bridge involution. In any case the action is trivial on $\mathbb{C}_{t'}^{2g-1+n}$. When $r=3$ and $s=1$, then the age is at least $\frac{1}{2}+\frac{1}{3}$, and if $h$ is the identity on all components except $R_\alpha$, then $C_\alpha$ is an elliptic curve meeting the rest of the curve at one node of type $Q_1$ with Weierstrass points sitting over a $\mathbb{Z}\big/3\mathbb{Z}$-orbit. In this case $h$ does not act as quasi-reflexion and ${\rm{age}}(h)=\frac{2}{3}<1$, i.e., the corresponding point on the moduli space is a non-canonical singularity. 
\end{proof}

Curves with an elliptic tail of $j$-invariant $0$ are a recurrent source of non-canonical singularity on moduli spaces of curves. They admit an automorphism $g$ of order $6$ where $g^3$ is the elliptic involution and $g^2$ permutes the three non-zero $2$-torsion points. The second one is the one induced by the order three automorphism on the quotient $E/g^3\cong \mathbb{P}^1$. This appeared first \cite{HM}, and then in \cite[Thm. 3.1]{Lu} for moduli spaces of spin curves and \cite[Thm. 6.7]{FL} for moduli spaces of Prym curves. In all these cases, not only pluricanonical forms defined on the regular locus extend but also the source of non-canonical singularities is always the same and has only to do with the stable model of the curve and not with the extra datum (spin structure, Prym structure, or cover structure). In each case, the Reid--Tai analysis is different and uses the extra datum at hand in a fundamental way. In our case, the cover structure and the lower bound of the age obtained via the action on the versal deformation space of the pointed target rational curve. That being said, it does not come as a surprise that the argument for extending pluricanonical forms works more or less uniformly in all these cases. For the sake of completeness, while also trying to avoid the repetition of minor modifications of Harris--Mumford's original argument, we sketch the argument in our situation omitting some details.

\subsection{Extending pluri-canonical forms}
\label{section:extension_forms}

Recall the following refinement of Raid--Tai criterion:

\begin{proposition}[Appendix 1 to \S1 in \cite{HM} and Prop. 4.2 in \cite{Lu}]
\label{prop:apendix}
Let $G\subset GL(V)$ be a finite group acting linearly on $V$, $V^\circ\subset V$ the subset where $G$ acts freely, and
$$\rho:\widetilde{V\big/G}\to V\big/G$$
a resolution of singularities. Let $U$ be an open subset of $V\big/G$ containing $V^\circ\big/G$ with the property that for every $g\in G$ with $0<{\rm{age}}(g)<1$, $U$ intersects non-trivially the image of the fixed locus ${\rm{Fix}}(g)\subset V$ on the quotient $V\big/G$. Then, a $G$-invariant pluri-canonical form $\omega$ on $V$ considered as a meromorphic form on $\widetilde{V\big/G}$ is holomorphic if it is holomorphic on $\widetilde{U}=\rho^{-1}(U)$.
\end{proposition}

By Theorem \ref{thm:non-can_sing} it is enough to show that pluri-canonical forms $\omega$ defined on the regular locus of $\overline{\rm{H}ur}_{g,2}^n$ extend locally around points in the moduli space consisting on covers having an elliptic tail. As in \cite{HM}, one separates the proof into two steps. First, we show extendability locally around generic non-canonical points, that is, covers 
\begin{equation}
\label{eq:attachE}
\left[f:C\cup_{p\sim 0} E\to \mathbb{P}^1\cup_q \mathbb{P}^1, p_1,\ldots,p_n\right],
\end{equation}
where $p_i\in C$, both components $C$ and $E$ are smooth curves of genus $g-1$ and $1$ meeting at a node of type $Q_1$, and the only nontrivial automorphism of $C$ is the cover involution $\sigma_C$, in particular $(C,p_1,\ldots,p_n)$ has no non-trivial automorphisms if $n\geq1$. Then one uses this together with Proposition \ref{prop:apendix} to show extendability for an arbitrary non-canonical point.\\

\begin{proof}[Proof of Theorem \ref{thmsix}]
We start by fixing and $n$-pointed smooth admissible double cover $\left[f:C\to\mathbb{P}^1,p_1,\ldots,p_n\right]$ of genus $g-1$ without non-trivial automorphisms. Let $p\in C$ be a ramification point and consider the map
$$\phi:\overline{\rm{M}}_{1,1}\cong \mathbb{P}^1\longrightarrow\overline{\rm{Hur}}_{g,2}^n$$
given by attaching an elliptic tail as in \eqref{eq:attachE}. The key idea in \cite{HM} is to contract elliptic tails to cusps and study locally the map $\overline{\mathcal{M}}_g\to\overline{\mathcal{M}}^{cusp}_g$ on coarse spaces around the point $\left[C\cup_pE\right]$. The same argument then shows that there exists an open neighbourhood $S$ of $\phi(\overline{\rm{M}}_{1,1})\subset \overline{\rm{Hur}}_{g,2}^n$ and a birational map $\eps:S\to B$, such that $B$ is smooth, $S\to B$ is an isomorphism over $B\setminus Z$, where $Z\subset B$ has codimension two, and $B\setminus Z\cong \eps^{-1}(B\setminus Z)$ sits in the regular locus of $\overline{\rm{Hur}}_{g,2}^n$. The diagram is as follows (see \cite[p. 43]{HM}, \cite[Sec. 4]{Lu}):

\begin{equation}
\begin{tikzcd}
\phi\left(\overline{\rm{M}}_{1,1}\right)\arrow[hook,r]&\arrow[hook,r]S\arrow[dl, "\eps"']\arrow[hook,r]&\overline{\rm{Hur}}_{g,2}^n\\
B&B\setminus Z\arrow[hook,r]\arrow[hook,u]\arrow[hook',l]&\left(\overline{\rm{Hur}}_{g,2}^n\right)_{reg}.\arrow[hook, u]
\end{tikzcd}
\end{equation}
Let $\widetilde{S}\to S$ be a resolution and consider the composition $\widetilde{S}\to B$. Any pluricanonical form on $\left(\overline{\rm{Hur}}_{g,2}^n\right)_{reg}$ restricts to a holomorphic form on $B\setminus Z$. Since $B$ is smooth, it extends holomorphically to $B$ and therefore, by pulling it back, to $\widetilde{S}$. This shows extendability on a resolution around points in $\overline{\rm{Hur}}_{g,2}^{n}$ as above, in particular, by Theorem \ref{thm:non-can_sing}, around a general non-canonical singularity. Let $[f]\in\overline{\rm{Hur}}_{g,2}^{n}$ be any non-canonical point, and $(E_1,x_1), \ldots, (E_k, x_k)$ the elliptic tails with $j$-invariant $j(E_i,x_i)=0$. For each $(E_i,x_i)$ one considers a general deformation $[f^{(i)}]$ of $[f]$ preserving $(E_i,x_i)$, that is, fixing $\tau_i=\tau_j=0$ where $\tau_i$ is the smoothing parameter ${\rm{tor}}_{x_i}\cong \mathbb{C}_{\tau_i}$ and $\tau_j$ is the deforming paramter of $(E_i,x_i)$, i.e., $H^0(E_i, \omega_{E_i}^{\otimes 2}(x_i))^{\sigma_{E_i}}\cong \mathbb{C}_{\tau_j}$. Then $[f^{(i)}]$ is as before and there exists an open set $S^{(i)}$ containing the corresponding $\phi^{(i)}(\overline{\mathcal{M}}_{1,1})$ where forms extend to a resolution. Let $V^\circ$ be the locus in $V=\mathbb{C}_{\tau}^{2g-1+n}$ where $G={\rm{Aut}}_\sigma\left(C,p_1,\ldots,p_n\right)$ acts freely. Then one shows that 
$$U=V^\circ\big/G\cup\left(\bigcup_{i=1}^k S^{(i)}\cap V\big/G\right)$$
satisfies the hypothesis of Proposition \ref{prop:apendix}.

\end{proof} 


\appendix
\numberwithin{equation}{section}

\section{The canonical class revisited}
\label{Appendix}
\addtocontents{toc}{\protect\setcounter{tocdepth}{0}}
\section*{by Irene Schwarz}

In this appendix, we revisit an earlier attempt to calculate the Kodaira dimension of $\overline{\mathcal{H}}_{g, n}$ by studying the positivity properties of the canonical class. The coarse space $\overline{\rm{H}}_{g,n}$ has been shown to be non-$\mathbb{Q}$-factorial for $g\geq 3$ and $n\geq2$, see Proposition \ref{thm:Hgnsing}. This makes the extendability of pluricanonical forms untenable, necessitating a better behaved birational model to compute the Kodaira dimension when this value is non-negative. We calculate the canonical divisor of the moduli space $\overline{\rm{H}}_{g,n}$ and use Logan's divisor \cite[Thm. 5.4]{Lo} together with a fundamental relation between the Hodge class and boundary divisors on $\overline{\mathcal{H}}_{g,n}$  to show that the canonical divisor is effective when $n=4g+6$ and big when $n\geq 4g+7$.


Recall that the coarse morphism $\overline{\mathcal{H}}_{g,n}\longrightarrow \overline{\rm{H}}_{g,n}$ induces a natural isomorphism on codimension one cycles 
\begin{equation}
\label{appendix:eq1}
{\rm{CH}}^1\left(\overline{\rm{H}}_{g,n}\right)\otimes\qq\cong{\rm{CH}}^1\left(\overline{\mathcal{H}}_{g,n}\right)\otimes\qq.
\end{equation}
When $n=0$, the stack $\overline{\mathcal{H}}_g$ is smooth \cite[Ch. XI, Lemma 6.15]{ACG} and equation \eqref{appendix:eq1} is an isomorphism of Picard groups with rational coefficients. In what follows we recall the main result in \cite{Sca} about the structure of ${\rm{CH}}^1(\overline{\mathcal{H}}_{g,n})\otimes \qq$ generalising \cite{Co} in the unpointed case.\\

The divisors $\Delta_{i:S}\subset \overline{\mathcal{M}}_{g,n}$ for $i\geq 1$ intersect $\overline{\mathcal{H}}_{g,n}$ transversally on an irreducible divisor whose class on $\overline{\mathcal{H}}_{g,n}$ we also denote $\delta_{i:S}$.\footnote{When $n=0$, $\delta_{i}$ is defined as $\frac{1}{2}$ times the pullback of $\delta_{i}$ on $\overline{\mathcal{M}}_g$. This is again to take into account the fact that a general curve in $\delta_i$ has twice as many automorphisms as the general hyperelliptic curve. Similarly, $\delta_{i:\varnothing}$ is also $\frac{1}{2}$ times the pullback of $\delta_{i:\varnothing}$ on $\overline{\mathcal{M}}_{g,n}$.} A general point in $\delta_{i:S}$ consist on two hyperelliptic curves of genus $g-i$ and $i$ meeting transversally at a common Weierstrass point. The divisor $\Delta_{\rm{irr}}$ intersects $\overline{\mathcal{H}}_{g}$ on 
$$\lfloor\frac{g-1}{2}\rfloor+1$$ 
many components whose reduced class we denote by $\eta_0,\eta_1,\ldots, \eta_{\lfloor\frac{g-1}{2}\rfloor}$, where the multiplicity of the intersection along $\eta_{i}$ for $i>0$ is $2$. Hence
\begin{equation}
\label{sec2:eq2}
{\delta_{\rm{irr}}}=\eta_{irr}+2\sum_{i=1}^{\lfloor\frac{g-1}{2}\rfloor}\eta_i.
\end{equation}
The element $\eta_{irr}\in{\rm{Pic}}_\qq\left(\overline{\mathcal{H}}_{g}\right)$ is the class of the locus $\Xi_{irr}$ of curves whose general element is an hyperelliptic curve with two conjugate points identified. Similarly, $\eta_i\in{\rm{Pic}}_\qq\left(\overline{\mathcal{H}}_{g}\right)$ is the class of the locus $\Xi_i$ of curves whose general element consist of two hyperelliptic curves $C_1, C_2$ and two pairs of conjugate points $p,\overline{p}\in C_1$ and $q,\overline{q}\in C_2$, where the point $p$ is identified with $q$ and $\overline{p}$ with $\overline{q}$. When $n\geq 1$, we denote by $\eta_{i:S}\in{\rm{Pic}}_\qq\left(\overline{\mathcal{H}}_{g, n}\right)$ the class of the locus $\Xi_{i:S}$ whose general element is a curve on $\Xi_i$, and the marked points indexed by $S$ are in the component of genus $i$.

\begin{theorem}[\cite{Co, Sca}]
For $g\geq 2$, the classes $\psi_i$, $\eta_{irr}$, $\eta_{i:S}$, and $\delta_{i:S}$ freely generate ${\rm{CH}}^1_\qq\left(\overline{\mathcal{H}}_{g,n}\right)$.
\end{theorem}

In particular, the restriction of the Hodge class $\lambda\in{\rm{Pic}}_{\mathbb{Q}}\left(\overline{\mathcal{M}}_{g,n}\right)$ to $\overline{\mathcal{H}}_{g,n}$ can be expressed as a linear combination of boundary classes, cf. \cite[Prop. 4.7]{CH}: 
\begin{equation}
\label{appendix:eq_lambda}
(8g+4)\lambda= g\eta_{irr} +2\sum_{i=1}^{\lfloor\frac{g-1}{2}\rfloor}\sum_S (i+1)(g-i)\eta_{i:S} +4\sum_{i=1}^{\lfloor\frac{g}{2}\rfloor}\sum_S i(g-i)\delta_{i:S}.
\end{equation}

The branch isomorphism induces a map $\alpha$ fitting in the following diagram \cite[Ch. XIII, eq. 8.7]{ACG}
\begin{equation}
\label{intro:diag_alpha}
\begin{tikzcd}
&\overline{{\rm{Hur}}}_{g,2}\arrow[dl]\arrow[dr]&\\
\overline{\rm{H}}_g\arrow[rr,"\alpha"]&&\overline{\rm{M}}_{0,[2g+2]}.
\end{tikzcd}
\end{equation}
For a general rational curve in the divisor classes $\delta_{0:[j]}\subset\overline{\mathcal{M}}_{0,[2g+2]}$, the parity of $j$ determines whether the pullback along $\alpha$ is of $\delta$-type, that is, two hyperelliptic curves glued at a Weierstrass point, or of $\eta$-type, that is, two hyperelliptic curves glued at two conjugate points.
The calculations  in \cite[Ch. XIII, eq. 8.7]{ACG} give rise to the following identities:

\begin{equation}
\label{eqn:pullback}
\alpha^*\delta_{0:[2]}=\frac{1}{2}\eta_{irr},\quad\alpha^* \delta_{0:\left[2i+2\right]}=\eta_i\quad\hbox{and}\quad\alpha^*\delta_{0:\left[2i+1\right]}=2\delta_{i}.
\end{equation}
From the formula \cite[Lemma 3.5]{KM} for the canonical class of $\overline{\rm{M}}_{0,[2g+2]}$ 
one computes
\begin{align}
K_{\overline{\rm{H}}_g}=&-\left(\frac{1}{2}+\frac{1}{2g+1}\right)\eta_{irr}+2\sum_{i=1}^{\lfloor\frac{g-1}{2}\rfloor}\left(\frac{(i+1)(2g-2i)}{2g+1}-1\right)\eta_i\\
&+\sum_{i=1}^{\lfloor\frac{g}{2}\rfloor}\left(\frac{2(2i+1)(2g-2i+1)}{2g+1}-4\right)\delta_i.
\end{align}

The coarse map $c:\overline{\mathcal{H}}_g\to\overline{{\rm{H}}}_g$ is simply ramified at the divisor $\delta_i$. This accounts for the fact that a general hyperelliptic curve in $\delta_i$ has an automorphism group of order four as opposed to the general one in the interior or in any $\eta$-boundary divisor whose only non-trivial automorphism is the hyperelliptic involution. Thus we can calculate the canonical class $K_{\overline{\rm{H}}_g}$ from $K_{\overline{\mathcal{H}}_g}$.
Moreover, if $p:\overline{\mathcal{M}}_{g,n}\to\overline{\mathcal{M}}_g$ is the forgetful morphism and $q:\overline{\mathcal{H}}_{g,n}\to\overline{\mathcal{H}}_g$ the pullback via the restriction $i:\overline{\mathcal{H}}_{g,n}\hookrightarrow\overline{\mathcal{M}}_{g,n}$, then from the identity
$$c_1\left(\Omega_q^1\right)=c_1\left(i^*\Omega_p^1\right)=i^*\left(K_{\overline{\mathcal{M}}_{g,n}}-p^*K_{\overline{\mathcal{M}}_g}\right)=\psi-2\sum_{\left|S\right|\geq 2}\delta_{0:S}$$
together with the fact that when $n\geq 2$ the coarse map $\overline{\mathcal{H}}_{g,n}\to \overline{\rm{H}}_{g,n}$ is simply ramified at $\delta_{i:\varnothing}$ one obtains

\begin{align}
\label{appendix:eq:can_class}
K_{\overline{\rm{H}}_{g,n}}=&\psi-\left(\frac{1}{2}+\frac{1}{2g+1}\right)\eta_{irr}+2\sum_{i\geq 1,S}\left(\frac{(i+1)(2g-2i)}{2g+1}-1\right)\eta_{i:S}\\
&+\sum_{i\geq1, S}\left(\frac{2(2i+1)(2g-2j+1)}{2g+1}-3\right)\delta_{i:S}-2\sum_{\left|S\right|\geq 2}\delta_{0:S}-\sum_{i\geq 1}\delta_{i:\varnothing}.
\end{align}

Recall \cite{Lo} Logan's divisor $\overline{\rm{D}}_g$ defined as the closure in $\overline{\mathcal{M}}_{g,g}$ of the locus 
$${\rm{D}}_g=\left\{\left[C,p_1,\ldots,p_g\right]\in \mathcal{M}_{g,g}\left| h^0\left(C,\mathcal{O}_C(p_1+\ldots+p_g)\right)\geq 2\right.\right\}.$$
\begin{proposition}
The hyperelliptic locus $\mathcal{H}_{g,g}$ is not contained in ${\rm{D}}_g$.
\end{proposition}
\begin{proof}
Take any $(C,p_1,\ldots,p_g)\in {\rm{D}}_g\cap \mathcal{H}_{g,g}$, then by Riemann-Roch there must be $q_1, \ldots, q_{g-2}\in C$ such that $p_1+\ldots+p_g+q_1+\ldots+q_{g-2}$ is a canonical divisor of $C$. Thus, there is some $i\neq j$ with $p_i$ and $p_j$ conjugate. This is a codimension one condition on $\mathcal{H}_{g,g}$.
\end{proof}

The formula for $\overline{\rm{D}}_g$ in terms of generators of ${\rm{Pic}}_\qq\left(\overline{\mathcal{M}}_{g,g}\right)$ was computed in \cite[Thm 5.4]{Lo}: 
\begin{equation}\label{Dg}
\overline{{\rm{D}}}_g = -\lambda + \psi -0\cdot\delta_{irr} -3\delta_{0:2} - \hbox{higher boundary terms}.
\end{equation}
When $n\geq g$ one can consider the {\it{symmetric pullback}} of $\overline{\rm{D}}_g$ to $\overline{\mathcal{M}}_{g,n}$ given by 
$$\overline{\rm{D}}_n=\frac{1}{\binom{n-1}{g-1}}\sum_{\substack{T\subset\{1,\ldots,n\}\\ \left|T\right|=g}}\pi_T^*\overline{\rm{D}}_g=-\frac{n}{g}\lambda + \psi-0\cdot\delta_{irr}-\left(2+\frac{g-1}{n-1}\right)\delta_{0:s} -\ldots,$$
where $\pi_T:\overline{\mathcal{M}}_{g,n}\to \overline{\mathcal{M}}_{g,g}$ is the forgetful map that remembers only the markings indexed by $T$. By pulling back $\overline{\rm{D}}_n$ to $\overline{\mathcal{H}}_{g,n}$ and using the relation \eqref{appendix:eq_lambda} one can write
$$K_{\overline{\rm{H}}_{g,n}}=\varepsilon\left(\psi_1+\ldots+\psi_n\right)+\left(1-\varepsilon\right)\overline{\rm{D}}_{n}+E,$$
for $\varepsilon>0$, with 
$$E=e_{irr}\eta_{irr}+\sum_{i=1}^{\lfloor\frac{g-1}{2}\rfloor}\sum_S e_{i:S}\eta_{i:S}+\sum_{i=0}^{\lfloor\frac{g}{2}\rfloor}\sum_S d_{i:S}\delta_{i:S}.$$ 
An immediate computation shows that for $n>g$, all coefficients of higher boundary divisors $e_{i:S}, d_{i:S}$ are positive for $0<\varepsilon<1$ sufficiently small. Similarly, using \eqref{appendix:eq_lambda} one computes 
$$e_{irr}= \frac{(1-\varepsilon)n-(4g+6)}{8g+4},$$
which is positive for $n>4g+6$ and $\varepsilon>0$ suficiently small and it vanishes for $n=4g+6$ and $\varepsilon=0$. Then bigness of the sum of $\psi$-classes gives us the theorem.

\begin{theorem}
The canonical divisor $K_{\overline{\rm{H}}_{g,n}}$ is effective when $n\geq 4g+6$ and big when $n\geq 4g+7$.
\end{theorem}


\end{document}